\documentclass{article}
\usepackage[utf8]{inputenc}
\usepackage[T1]{fontenc}
\usepackage{tgpagella}
\usepackage[margin=1.2in]{geometry}
\usepackage[all]{xy}
\usepackage{url}
\usepackage{hyperref}

\hypersetup{
colorlinks,
linkcolor={red!50!black},
citecolor={blue!50!black},
urlcolor={blue!80!black}
}

\usepackage{enumitem}
\usepackage{amsfonts}
\usepackage{amsmath}
\usepackage{amssymb}
\usepackage{bbm}
\usepackage{extpfeil}
\usepackage{amsthm}
\usepackage{mathtools}
\usepackage{tikz-cd}
\usepackage{xcolor}
\usepackage{caption}
\newcommand{\op}{\mathrm{op}}
\newcommand{\cat}{\textup{\textbf{Cat}}_\infty}
\newcommand{\id}{\textup{id}}

\newcommand{\C}{\mathcal{C}}

\newcommand{\D}{\mathcal{D}}
\newcommand{\F}{\mathcal{F}}

\newcommand{\Fun}{\textup{Fun}}
\newcommand{\lax}{\textup{lax}}
\newcommand{\cl}{\textup{cl}}
\newcommand{\e}{\epsilon}
\usepackage{stackrel}

\renewcommand{\rightarrowtail}{\tikz\draw[>->] (0,0) -- (10pt,0);}

\newcommand{\rrightarrow}{\mathrel{\mathrlap{\rightarrow}\mkern1mu\rightarrow}}
\newcommand{\lleftarrow}{\mathrel{\mathrlap{\leftarrow}\mkern1mu\leftarrow}}
\newcommand{\closedrightarrow}{\mathrel{\mathrlap{\rightarrowtail}\mkern1mu\rightarrow}}
\renewcommand{\twoheadrightarrow}{\rrightarrow}
\renewcommand{\twoheadleftarrow}{\lleftarrow}
\renewcommand{\xtwoheadrightarrow}[2][]{%
  \xrightarrow[#1]{#2}\mathrel{\mkern-14mu}\rightarrow
}

\renewcommand{\xtwoheadleftarrow}[2][]{%
  \leftarrow\mathrel{\mkern-14mu}\xleftarrow[#1]{#2}
}

\numberwithin{equation}{subsection}

\theoremstyle{definition}
\newtheorem{defn}[equation]{Definition}
\newtheorem{example}[equation]{Example}
\newtheorem{rmk}[equation]{Remark}
\newtheorem{nota}[equation]{Notation}

\theoremstyle{plain}
\newtheorem{thm}[equation]{Theorem}
\newtheorem{lem}[equation]{Lemma}

\newtheorem{prop}[equation]{Proposition}
\newtheorem{cor}[equation]{Corollary}

\newtheorem{thmx}{Theorem}
\newtheorem{propx}[thmx]{Proposition}

\newcommand{\var}{\textup{\textbf{Var}}}

\newcommand{\comp}{\textup{\textbf{Comp}}}
\newcommand{\smcomp}{\textup{\textbf{SmComp}}}

\renewcommand{\span}{\textup{\textbf{Span}}}

\newcommand{\sh}{\textup{Sh}}

\newcommand{\hsh}{\textup{HSh}}

\renewcommand{\hom}{\textup{Hom}}

\newcommand{\spectra}{\mathbf{\mathcal{S}}\textup{\textbf{pectra}}}

\newcommand{\set}{\textup{Set}}

\newcommand{\psh}{\textup{PSh}}
\renewcommand{\S}{\mathcal{S}}

\newcommand{\corr}{\textup{\textbf{Corr}}}
\newcommand{\fun}{\textup{Fun}}

\newcommand{\fin}{\textup{\textbf{Fin}}}
\renewcommand{\part}{\textup{part}}

\newcommand{\A}{\mathcal{A}}
\newcommand{\B}{\mathcal{B}}
\newcommand{\ex}{\textup{ex}}
\newcommand{\prst}{\textup{\textbf{PrSt}}_\infty}
\newcommand{\Prst}{\textup{\textbf{PrSt}}_\infty}
\newcommand{\PrSt}{\textup{\textbf{PrSt}}_\infty}

\newcommand{\Prcat}{\textup{\textbf{Pr}}_\infty}

\newcommand{\alg}{\textup{Alg}}
\newcommand{\Alg}{\textup{Alg}}
\newcommand{\loc}{\textup{loc}}
\newcommand{\E}{\mathcal{E}}
\renewcommand{\textbf}{}

\date{}

\begin{document}

\title{An axiomatization of six-functor formalisms}
\author{Josefien Kuijper}
\maketitle

\begin{abstract}
In this paper, we consider two variations on Mann's definition $\infty$-categorical definition of abstract six-functor formalisms. We consider Nagata six-functor formalisms, that have the additional requirement of having Grothendieck and Wirthm\"uller contexts. We also consider local six-functor formalisms, which in addition to this, take values in presentable stable $\infty$-categories, and have recollements. Using Nagata's compactification theorem, we show that Nagata six-functor formalism on varieties can be given by just specifying adjoint triples for open immersions and for proper morphisms, satisfying certain compatibilities. The existence of recollements is (almost) equivalent to a hypersheaf condition for a Grothendieck topology on the category of ``varieties and spans consisting of an open immersion and a proper map''. Using this characterization, we show that the category of local six-functor formalisms embeds faithfully into the category of lax symmetric monoidal functors from the category of smooth and complete varieties to the category of presentable stable $\infty$-categories and adjoint triples.  We characterize which lax symmetric monoidal functors on complete varieties, taking values in the category of presentable stable $\infty$-categories and adjoint triples, extend to local six-functor formalisms.    
\end{abstract}

\tableofcontents
 
\section{Introduction}
A six-functor formalism encapsulates the structure that is often present in sheaf theories in algebraic geometry, which underlie cohomology theories. Roughly speaking, a six-functor formalism assigns to every scheme (or variety/topological space/...) $X$ a closed symmetric monoidal category\footnote{We use the word category to refer both 1-categories and $\infty$-categories. We specify which one of the two is meant, when it is relevant.} $D(X)$. This gives, for $A$ in $D(X)$, an adjunction
\begin{center}
\begin{tikzcd}
D(X) \arrow[r, "A\otimes -", shift left] & D(X). \arrow[l, "{\textup{hom}(A,-)}", shift left]
\end{tikzcd}
\end{center}
 A morphism of schemes (varieties/topological spaces/...) $f:X \to Y$  should give rise to two adjunctions $f^*\dashv f_*$ and $f_!\dashv f^!$. Collectively, this structure is also known as \textit{Grothendieck's six operations}. 
 
 Six-functor formalisms that occur in nature often look like ``the derived category of (some adjective) sheaves on $X$''. A classical example is the assignment of the bounded derived category $D^b_{\textup{constr}}(X)$ of constructible sheaves on a scheme $X$, where the six operations are given by the internal hom and tensor product of sheaves, the direct and inverse image $f_*$ and $f^*$, and the exceptional direct and inverse image $f_!$ and $f^!$ (sometimes, the latter two functors are only defined for a certain subclass of morphisms).  Six-functor formalisms have, for example, been constructed and studied for manifolds and topological spaces in \cite{sheaves_on_manifolds} and \cite{volpe}, for motivic homotopy theory in \cite{ayoub_i}, \cite{ayoub_ii}, for \'etale cohomology of Artin stacks in \cite{liu_zheng_six_ops}, for equivariant motivic homotopy theory in \cite{hoyois}, for rigid analytic motives over general rigid analytic spaces in \cite{ayoub_gallauer_vezzani} and for Morel-Voevodsky's stable $\mathbb{A}^1$-homotopy theory in \cite{drew_gallauer}.
 
 Six-functor formalisms that arise from a sheaf theory, generally satisfy additional properties, such as base change and the projection formula, which indicate relations between the functors $f_!$, $f^*$ and $\otimes$. Another common feature is the existence of \textit{Grothendieck and Wirthm\"uller contexts}. This means that for some class $P$ of morphisms, there exists a natural equivalence $f_*\cong f_!$ (a {Grothendieck context}), and for another class $I$, a natural equivalence $f^!\cong f^*$ ({a Wirthm\"uller context}). As the notation suggest, the morphisms in $P$ are often proper maps and those in $I$ open immersions. In that case, for $i:U\to X$ an open immersion and $j:X\setminus U \to X$ its closed complement, there is a diagram
 \begin{equation}\label{eq:recollement_intro}
\begin{tikzcd}[sep=large]
D(U\setminus X) \arrow[r, "j_!\cong j_*" description] & D(X) \arrow[l, "j^ !", shift left=2] \arrow[l, "j^ *"', shift right=2] \arrow[r, "i^*\cong i^!" description] & D(U) \arrow[l, "i_!"', shift right=2] \arrow[l, "i_*", shift left=2]
\end{tikzcd}.
\end{equation}
Such a diagram is called a \textit{recollement} if it satisfies a list of properties, including that $j^* i_*=0$ and that $i_!$, $j_!$ and $j_*$ are embeddings. If for every open immersion $i:U\to X$ the associated diagram of this shape is a recollement, then we say that the six-functor formalism $D$ satisfies \textit{localization}.

In this article, we consider two variations on the definition of six-functor formalisms, as defined by Mann in \cite{mann_thesis}. There, abstract six-functor formalisms are defined on an arbitrary \textup{geometric setup}, which consists of a category $\C$, and a class of morphisms $E$ in $\C$ for which not only the direct and inverse image, but also the exceptional direct and inverse image are defined. We first consider \textit{Nagata six-functor formalisms}, which are six-functor formalisms that have Grothendieck and Wirthmüller contexts for classes of morphisms $P$ and $I$. Moreover, the geometric setup and the classes $P$ and $I$ are required to form a \textit{Nagata setup}, meaning that every morphism in $\C$ factors as a morphism in $I$ followed by a morphism in $P$. We then show that a Nagata six-functor formalism is uniquely determined by its restriction to morphisms in $I$ and $P$. More precisely, the category of Nagata six-functor formalisms on a fixed Nagata setup is equivalent to a subcategory of functors out of the category of ``correspondences consisting of a $P$- and an $I$-morphism''.

Second, we consider \textit{local six-functor formalisms}, which are defined on a \textit{noetherian Nagata setup}, of which the category of varieties is the motivating example. Local six-functor formalisms are Nagata six-functor formalisms that take values in presentable stable $\infty$-categories, and have recollements. We show that local six-functor formalisms are uniquely determined by their restriction to a class of ``complete objects'', which reduce to complete varieties in the noetherian Nagata setup of varieties. Moreover, the category of local six-functor formalisms on varieties over a field of characteristic zero embeds into a category of functors out of the category of smooth varieties.

\subsection{Various $\infty$-categorical approaches to six-functor formalisms}
In recent years, there have been several $\infty$-categorical axiomatizations of the notion of a six-functor formalism. A convenient way to wrap up all of the structure present in a six-functor formalism, and to capture some of the aforementioned properties, is by defining a six-functor formalism on a category with products $\mathcal{C}$ to be a lax symmetric monoidal functor 
\begin{equation}\label{eq:intro_6ff}
    D:\textup{\textbf{Corr}}(\mathcal{C}) \to \textup{\textbf{Cat}}_\infty.
\end{equation}
such that for every $X$ in $\mathcal{C}$, $D(X)$ is a closed symmetric monoidal $\infty$-category (where the symmetric monoidal structure is induced by the diagonal $\Delta:X\to X\times X$ and $D$ being a lax symmetric monoidal functor) and such that for every morphism $f$ in $\textup{Corr}(\mathcal{C})$, $D(f)$ is a left adjoint. In practice, the category $\mathcal{C}$ is some ($\infty$-)category of geometric objects, such as varieties, (derived) schemes, stacks or sufficiently nice topological spaces. The category $\textup{Corr}(\C)$ is the category of correspondences in $\C$. The objects of $\corr(\C)$ are objects in $\C$, and the morphisms are spans
$$X \xlongleftarrow{f}  Y \xlongrightarrow{g} Z $$
see also e.g. \cite[Definition A.5.2]{mann_thesis}.
We denote the image of such a span under $ D$ by
$$D(X)\xlongrightarrow{f^*}D(Y) \xlongrightarrow{g_!} D(Z),$$
and we denote the right adjoint of $f^*$ by $f_*$ and the right adjoint of $g_!$ by $g^!$.
Composition in $\textup{Corr}(\C)$ is given by pullbacks. Base change then follows from functoriality, and the projection formula comes from $D$ being a lax symmetric monoidal functor. This idea is attributed to Lurie and worked out in \cite{gaitsgory}, in order to develop a theory of ind-coherent sheaves on schemes and stacks. An $(\infty,2)$-categorical approach is developed by Gaitsgory and Rozenblyum in \cite{gaitsgory_rozenblyum}. An alternative approach, avoiding $(\infty,1)$-categories, is taken by Liu and Zheng in \cite{liu_zheng_six_ops}, where they construct a six-functor formalism encoding étale cohomology of Artin stacks. In \cite{Hormann}, H\"ormann uses the theory of (op)fibrations of 2-multicategories to define abstract six-functor-formalisms. In his thesis \cite{mann_thesis}, Mann defines the category of abstract six-functor formalisms as lax symmetric monoidal functors out of the $(\infty,1)$-category of correspondences in a \textit{geometric setup}, which is a category that has the minimal structure that is needed structure to admit a notion of six-functor formalism. Grothendieck and Wirthm\"uller contexts are not part of his definition, but \cite[Proposition A.5.10]{mann_thesis} shows how six-functor formalisms that have them, can be constructed from just a lax symmetric monoidal functor
$$D^*:\C^\op \to \textup{\textbf{Cat}}_\infty$$
satisfying a number of conditions, involving the existence of a left adjoint $i_!$ and a right adjoint $i_*$ of $D(i)=i^*$ for $i$ in $I$, and for $p$ in $P$ a right adjoint $p_*$ of $D(p)=p^*$, and a further right adjoint $p^!$ of $p_*$. In a recent preprint \cite{khan}, Khan axiomatizes six-functor formalisms on derived schemes, named ``weaves'', with Grothendieck contexts for proper morphisms, Wirthm\"uller contexts for \'etale morphisms, and \textit{Poincar\'e duality} for smooth morphisms. The latter means that for $f$ smooth, there is a natural isomorphism $f^! \simeq  f^!(\mathbbm{1}_Y)\otimes f^*(-)$. Drew and Gallauer have a similar approach in \cite{drew_gallauer} using ``coefficient systems'', that in addition to Grothendieck and Wirthm\"uller contexts and Poincar\'e duality, also have recollements and satisfy $\mathbb{A}^1$-homotopy invariance. 

\subsection{Motivation and main results}

In this text we build upon Mann's $(\infty,1)$-categorical definition of six-functor formalisms, and axiomatize what structure is needed in a geometric setup, to allow for the existence of Grothendieck and Wirthm\"uller contexts, and to allow for a notion of localization. To this end, we consider six-functor formalisms on a \textit{Nagata setup}. This is an $\infty$-category $\C$ with fiber products and terminal object, and two sufficiently nice classes of morphisms $I$ and $P$, such that every morphism in $\C$ can be factorized as a morphism in $I$ followed by a morphism in $P$. An important example of a Nagata setup is the category of qcqs schemes over a base scheme $S$ and separated, finite type morphism, with $I$ and $P$ open immersions and proper morphisms respectively. Given such a setup, we define a category $\textup{\textbf{6FF}}$ of \textit{Nagata six-functor formalisms}, which are six-functor formalisms such that for $i$ in $I$, the functors form an adjoint triple  $i_!\dashv (i^!\cong i^*) \dashv i_*$, and for $p$ in $P$, the functors form an adjoint triple $p^* \dashv (p_*\cong p_! )\dashv p^!$.

\begin{rmk}
    
In a separate paper, joint with Adam Dauser \cite{conjecture}, we show that this notion of six-functor formalism coincides with lax symmetric monoidal functors
$$D:\corr(\C)\to \cat $$
such that morphisms in $P$ are cohomologically proper and morphisms in $I$ are cohomologically étale in the sense of \cite[Lecture VI]{scholze_notes}.
\end{rmk}

One motivation for this definition is that a six-functor formalism as described above, is an extremely over-determined object. For example, the six operations come in pairs of adjoints, and therefore instead of giving all of them explicitly, giving only half of them suffices. 
Indeed, in existing $(\infty,1)$-categorical frameworks, the data of a six-functor formalism is a lax symmetric monoidal functor (\ref{eq:intro_6ff}), encoding only the left adjoints $f^*$, $f_!$ and $A\otimes-$ explicitly; this already boils down the amount of data significantly. Now for a Nagata six-functor formalism, if $f$ is in $I$ or $P$, then the four functors $f^*, f_*, f_!$ and $f^!$ form an adjoint triple, so in theory giving only one of the four should suffice. This is particularly effective since $P$ and $I$ form a factorization system of all morphisms in $\C$.

 One of the axioms of a Nagata setup is that both $I$ and $P$ are stable under pullback by arbitrary morphisms. Therefore we can define the following $\infty$-category.

\begin{defn}\label{defn:span}
We denote by $\corr(\C,P,I) \subseteq \corr(\C)$ the subcategory spanned by all objects in $\C$, and by morphisms of the form
$$X\xleftarrow{p} U \xrightarrow{i} Y  $$
with $i\in I$ and $p\in P$.  
\end{defn}

We call a functor of $\infty$-categories a left-left adjoint if it is a left adjoint, and its right adjoint is also a left adjoint; in other words if it occurs as leftmost functor in an adjoint triple. Using methods from \cite{mann_thesis} and \cite{liu_zheng_six_ops} we can show the following, showing that a Nagata six-functor formalism is indeed determined just by its restriction to $I$- and $P$-morphisms. 
\begin{propx}[Proposition \ref{prop:w6FF}]\label{propx:intro0}
    The $\infty$-category $ \textup{\textbf{6FF}}$ is equivalent to the $\infty$-category of lax symmetric monoidal functors 
$$F:\corr(\C,P,I) \to \cat$$
such that \begin{itemize}
    \item[(i)]for $X$ in $\C$, $F(X)$ with the induced symmetric monoidal structure is a closed symmetric monoidal $\infty$-category,
    \item[(ii)] for any morphism $f$ in $\corr(\C,P,I)$, $F(f)$ is a left-left adjoint,
    \item[(iii)] and $F$ sends certain types of squares to adjoinable squares (the \textit{BC property}, see (B) in Definition \ref{defn:bc_span}).
\end{itemize}
and of symmetric natural transformations that are compatible with the right adjoints of the functors $F(f)$. 
\end{propx}

\begin{rmk}
In a sense, Proposition \ref{propx:intro0} looks like a dual the conjecture stated in \cite[page 47]{scholze_notes}, and \cite[Theorem 2.51]{khan}. Scholze conjectures that Mann's construction \cite[Proposition A.5.10]{mann_thesis} induces an equivalence between the $\infty$-category of ``three-functor formalisms'' with Grothendieck and Wirthm\"uller contexts, and the $\infty$-category of functors $D^*:\C^\op \to \cat$ that encode the data of $f^*$ for all morphisms and satisfy certain conditions that involve the existence of a left adjoint adjoint $i_\natural\vdash i^*$ for $i$ in $I$ and a right adjoint $p^*\vdash p_*$ for $p$ in $P$ (this can be refined to a statement about six-functor formalisms). In joint work with Adam Dauser \cite{conjecture}, we prove this conjecture with the methods also used in the current paper. This also implies a proof of \cite[Theorem 2.51]{khan} that avoids the use of $(\infty,2)$-categories. 
\end{rmk}

Then we specialize to the setting of particularly well-behaved Nagata setups, which we call \textit{noetherian Nagata setups}. The archetypal example is the Nagata setup $(\var,I,P)$ of varieties \footnote{reduced, separated, finite type schemes over a field} over a field $k$, with $I$ the open immersions and $P$ the proper morphisms. For the rest of this introduction, we assume that this is the noetherian Nagata setup we are working with. Let us consider the following composable morphisms in $\corr(\var,P,I)$:
\begin{equation}\label{eq:localization}
U \xleftarrow{=} U \xrightarrow{i} X \xleftarrow{j} X\setminus U \xrightarrow{=} X \setminus U   
\end{equation}
We define a subcategory 
$$ \textup{\textbf{6FF}}^{\loc} \subseteq  \textup{\textbf{6FF}}$$
of \textit{local six-functor formalisms}, which are six-functor formalisms $D$ such that $D(X)$ is a presentable stable $\infty$-category for all $X$, and such that each sequence (\ref{eq:localization}) is sent to a recollement (\ref{eq:recollement_intro}).
\begin{propx}[Proposition \ref{prop:6ff_to_span}] \label{propx:intro1}
 The equivalence in Proposition \ref{propx:intro0}  restricts to an equivalence between $\textup{\textbf{6FF}}^{\loc}$ and the $\infty$-category of lax symmetric monoidal functors 
$$ F:\corr(\var,P,I)\to \PrSt$$
that take values in presentable stable $\infty$-categories and left-left adjoints, have the BC-property, and that in addition send sequences (\ref{eq:recollement_intro}) to recollements, abstract blowups to pullbacks, and satisfy $F(\emptyset)=\mathbbm{1}$; and of symmetric monoidal natural transformations that are compatible with the right adjoints of the functors $F(f)$.
\end{propx}

The condition that $ F:\corr(\var,P,I) \to \PrSt$ sends sequences (\ref{eq:localization}) to recollements and abstract blowups to pullbacks, is equivalent to the condition that $F$ is a hypersheaf for a certain topology on $\corr(\var,P,I)^\op$, and moreover satisfies $F(\emptyset)=\mathbbm{1}$. 

We can describe local six-functor formalisms even more efficiently than by Proposition \ref{propx:intro1}. For every variety $U$, there is a complete variety $X$ with an open immersion $U\hookrightarrow X$, and it follows that the complement $X\setminus U$ is also complete. In the recollement (\ref{eq:recollement_intro}), the object $D(U)$ is uniquely determined by $D(U\setminus X)$, $D(X)$ and the adjoint triple $j^*\vdash( j_*\cong j_!)\vdash j^!$. Therefore, at least morally, a six-functor formalism with recollements should be completely determined by its restriction to complete varieties. We make this precise as follows. Let $\comp$ be the category of complete varieties. Note that $\comp^\op$ is a subcategory of $\corr(\var,P,I)$.
\begin{thmx}[Theorem \ref{thm:cs6ff}]\label{thmx:intro}
   Restricting to $\comp^\op$ induces an equivalence between $\textup{\textbf{6FF}}^{\loc}$ and the $\infty$-category of lax symmetric monoidal functors
$$ F:\comp^\op \to \PrSt$$
that take values in presentable stable $\infty$-categories and left-left adjoints, that satisfy a BC-property (see  (2) in definition \ref{defn:bc_comp}), that send abstract blowups to pullbacks and that satisfy $F(\emptyset) = \mathbbm{1}$; and of symmetric monoidal natural transformations that are compatible with the right adjoints of the functors $F(f)$. 
\end{thmx}
In other words, a local six-functor formalism is uniquely determined by the $(-)^*$-part restricted to complete varieties. Moreover, we can characterize exactly which functors $\comp^\op \to \PrSt$ can be uniquely extended to a fully-fledged six-functor formalism. In a similar vain to the remark after Proposition \ref{propx:intro1}, the condition that $ F:\comp^\op \to \PrSt$ sends abstract blowups to pullback squares and satisfies $F(\emptyset) = \mathbbm{1}$, is equivalent to the condition that $F$ is a (hyper)sheaf for a certain topology on $\comp$. \\

Lastly, under the additional assumption that our base field $k$ has characteristic zero, using resolution of singularities we can show the following.
\begin{thmx}[Theorem \ref{thm:6FF}]\label{thmx:intro2}
 Over a field of characteristic zero, restricting to smooth and complete varieties induces a fully faithful embedding of $\textup{\textbf{6FF}}^\loc$   
into the $\infty$-category of $\infty$-operad maps 
$$F: \smcomp^\op  \to \PrSt^\otimes$$
that satisfy a BC condition (see (3) in Remark \ref{rmk:smcomp}) and that satisfy descent for blowups and $F(\emptyset)=\mathbbm{1}$; and of symmetric monoidal natural transformations that are compatible with the right adjoints of the functors $F(f)$. 
\end{thmx}

Again, sending blowups to pullbacks and $F(\emptyset)=\mathbbm{1}$ is equivalent to a hypersheaf condition for a certain topology on $\smcomp$. This corollary shows that a local six-functor formalism on varieties is, in fact, uniquely determined by the restriction of the $(-)^*$-part on smooth and complete varieties. However, we do not have a good characterization of which lax symmetric monoidal functors $\smcomp^\op \to \PrSt$ uniquely extend to six-functor formalisms.\\

There is an interesting analogy between Theorem \ref{thmx:intro} and Theorem \ref{thmx:intro2}, and a result about ``abstract compactly supported cohomology theories'' in our previous work \cite{kuij_descent}. There, the notation $\span$ is introduced for the category $\corr(\C,I,P)=\corr(\C,P,I)^\op$ . For $F:\span^\op \to \spectra$ a functor, we denote $H^n_c(X)=\pi_{-n}(F(X))$. If $F$ sends sequences (\ref{eq:localization}) to fiber sequences, then these groups satisfy the long exact sequence
 $$\dots \to H^n_c(U)\to H^n_c(X) \to H^n_c(X\setminus U) \to H^{n+1}_c(U)\to \dots $$
 and can therefore reasonably be called a \textit{compactly supported cohomology theory}. This condition on $F$ is the same hypersheaf condition that is mentioned after Proposition \ref{propx:intro1}. We can consider the $\infty$-category $\hsh(\span;\spectra)_\emptyset$, of hypersheaves that in addition satisfy $F(\emptyset)=*$,  as the $\infty$-category of ``abstract compactly supported cohomology theories''. Local six-functor formalisms can be seen as a more higher, more structured analogue of these. 
 
 In \cite{kuij_descent}, we showed that for $\D$ a complete, cocomplete and pointed $\infty$-category, restriction along the inclusion $\comp \hookrightarrow \span$ induces an equivalence between categories of hypersheaves
 \begin{equation}\label{eq:intro1}
      \hsh_{\tau^c_{A\cup L}}(\span; \D)_\emptyset \simeq \hsh_{\tau_{AC}}(\comp;\D).
 \end{equation}
Over a field of characteristic zero, restricting further to the category of smooth and complete varieties $\smcomp$, induces an equivalence between $\infty$-categories of hypersheaves
\begin{equation}\label{eq:intro2}
     \hsh_{\tau_{AC}}(\comp;\D) \simeq \hsh_{\tau_B}(\smcomp;\D)
 \end{equation}
The Grothendieck topologies on $\span$, $\comp$ and $\smcomp$ are those that occur in Proposition \ref{propx:intro1}, Theorem \ref{thmx:intro} and Theorem \ref{thmx:intro2}. One can interpret (\ref{eq:intro1}) and (\ref{eq:intro2}) as saying that abstract compactly supported cohomology theories are determined by their restriction to $\smcomp$, and they characterize which functors on $\comp$, and even which functors on $\smcomp$, uniquely extend to compactly supported cohomology theories. Theorem \ref{thmx:intro} and Theorem \ref{thmx:intro2} can be seen as the higher analogues of these results. In fact, we derive them from a variation on (\ref{eq:intro1}) and (\ref{eq:intro2}); we show that these can be improved to equivalences of categories of lax symmetric monoidal functors whose underlying functor is a hypersheaf.

\subsection{Linear overview}
The outline of the paper is as follows.

\begin{itemize}
    \item Section \ref{sect:prel} is structured as follows.
    \begin{itemize}
        \item In Subsections \ref{subsect:prereq1} and \ref{subsect:prereq2} we recall terminology around symmetric monoidal $\infty$-categories, multisimplicial sets, and the tensor product of presentable $\infty$-categories.
        \item In Subsection \ref{subsect:prereq_recollements} we recall terminology around adjoinable squares and recollements. We also prove some technical lemmas about adjoinable squares, which are used heavily in proving Theorem \ref{thmx:intro}.
        \item In Subsection \ref{subsect:prereq_passing_to_adjoints} we prove a categorified version of Liu and Zheng's ``Partial Adjoint'' lemma, which is instrumental in proving Proposition \ref{propx:intro0}.
    \end{itemize}
    \item Section \ref{sect:reducingtospan} is structured as follows.
    \begin{itemize}
        \item In Subsections \ref{subsect:nagata_setup_and_6ff} and \ref{subsect:nagata_6ff} we recall Mann's definition of six-functor formalisms, and define Nagata setups and Nagata six-functor formalisms.
        \item In Subsection \ref{subsect:proper_open} we prove Proposition \ref{propx:intro0}.   \end{itemize}
    \item Section \ref{sect:reducing_to_comp} is structured as follows. \begin{itemize}
        \item  In Subsections \ref{subsect:noeth_nagata_setup} and \ref{subsect:span_and_comp} we define the notion of a noetherian Nagata setup and show that in such a setup, the equivalence (\ref{eq:intro1}) holds. These subsections can be skipped if one is only interested in local six-functor formalisms on the category of varieties, or reduced, separated, finite type schemes over a noetherian base scheme.
        \item   In Subsection \ref{subsect:local_6ff} we define local six-functor formalisms on a noetherian Nagata setup, and prove Proposition \ref{propx:intro1}. 
        \item  In Subsection \ref{subsect:monoidal_compact_support} we show that variants of the equivalences (\ref{eq:intro1}) and (\ref{eq:intro2}) for categories of lax symmetric monoidal hypersheaves hold. 
        \item In Subsection \ref{subsect:comp} we prove Theorem \ref{thmx:intro} and Theorem \ref{thmx:intro2}. 
    \end{itemize} 
\end{itemize}

\subsection{Acknowledgements}
I want to thank Dan Petersen for sharing with me the idea that one should be able to recover a six-functor formalism from a functor defined on $\span$, and also for many comments on earlier versions of this manuscript. I also thank Adam Dauser, Elden Elmanto, D\'enis-Charles Cisinski, Adeel Khan and Lucas Mann for comments on earlier versions of this manuscript. I thank Fabian Hebestreit for helpful discussions, and for making me think twice about invoking \textit{presentable} stable $\infty$-categories. Lastly, I thank the anonymous referee for many comments that greatly improved the readability of this article.

\section{Preliminaries}\label{sect:prel}
In this section we discuss some preliminaries. Throughout, $\infty$-categories are as in \cite[Definiton 1.1.2.4]{htt}, also known as quasicategories. For $X$ a simplicial set and $\C$ an $\infty$-category, the mapping space $\textup{Map}(X,\C)$ is an $\infty$-category (\cite[Proposition 1.2.7.3]{htt}), and we denote it by $\textup{Fun}(X,\C)$.

\subsection{Symmetric monoidal $\infty$-categories and multisimplicial nerves}\label{subsect:prereq1}

We introduce the following notation. Note that we denote symmetric monoidal $\infty$-categories as cocartesian fibrations over the category of finite sets and partial maps $\fin_\part$ rather than the isomorphic category of pointed finite sets and pointed maps, as is done in \cite{HA}.

\begin{nota}
\label{nota:Cotimes} 
    For $\C$ an $\infty$-category, let $\C^\sqcup \to \fin_\part$ denote the cocartesian fibration from \cite[Construction 2.4.3.1]{HA}. 
\end{nota}
An edge $f:(X_i)_I\to (Y_j)_J$ in $((\C^\op)^\sqcup)^\op$, which we from now on denote by $\C^{\op,\sqcup,\op}$, consists of a map $\alpha:J \to I$ in $\fin_\part$, and for every $j\in J$ with $\alpha(j)=i$ a map $X_i\to Y_j$. If $\C$ has products, then this is equivalent to a map $X_i\to \prod_{j\in \alpha^{-1}(i)} Y_j$ for every $i\in I$. Then the cocartesian fibration
$$\C^{\op, \sqcup} \to \fin_\part$$
classifies the symmetric monoidal structure on $\C^\op$ given by the product in $\C$ (hence the coproduct in $\C^\op$). 
\begin{nota}\label{nota:C_times}
    For $(\C, \otimes)$ a symmetric monoidal 1-category, $\C^\op$ is also a symmetric monoidal 1-category, which is classified by a cocartesian fibration
    $$ (\C^\op)^\otimes \to \fin_\part.$$
    We denote $((\C^\op)^{\otimes})^\op$ by $\C_\otimes$.
\end{nota}
If $(\C,\times)$ is a cartesian symmetric monoidal 1-category, then by definition, $\C_\times$ coincides with $\C^{\op,\sqcup,\op}$.

\begin{nota}\label{nota:morphisms} Let $\C$ be an $\infty$-category, and $A$ a set of morphisms in $\C$.
\begin{itemize} 
\item[(1)] We denote by $A^\op$ the corresponding set of morphisms in $\C^\op$.
    \item[(2)]  We denote by $A_\sqcup$ the set of morphisms in $\C^{\op,\sqcup,\op} $ consisting of all
$$(X_i)_I \to (Y_j)_J $$
lying over $\alpha:J\dashrightarrow I$ such that for all $i \in I$ and $j\in J$ with $\alpha(j)=i$, the map $X_i \to Y_j$ is in $A$. 
\item[(3)] We denote by $A_-$ the subset of $A_\sqcup$ consisting of only morphisms over identities in $\textup{Fin}^\part$.
\item[(4)] If $\C$ has products, then we denote by $A_\times$ the set of morphisms 
$$(X_i)_I \to (Y_j)_J $$
lying over $\alpha:J\dashrightarrow I$, such that for all $i \in I$ the map $X_i\to \prod_{j\in \alpha^{-1}(i)} Y_j$ is in $A$. 
\end{itemize}
\end{nota}

We recall here only the most essential terminology around multisimplicial sets from \cite[Section 1.3]{liu_zheng_six_ops} and refer the reader to loc. cit. for more details.

\begin{defn}[{\cite[Definition 1.3.1]{liu_zheng_six_ops}}]
    We denote by $\set_\Delta$ the category of simplicial sets. For $I$ a finite set, we denote by $\set_{I\Delta}$ the category of \textit{I-simplicial sets}, which is the functor category
    $$\fun(\fun(I,\Delta)^\op,\set).$$
    For $I = \{1,\dots,k\}$, we denote $\set_{I\Delta}$ as $\set_{k\Delta}$ and call its objects \textit{k-simplicial sets}.
\end{defn}
For $I$ a finite set, there is a natural inclusion $\delta_I:\Delta\to \fun(I,\Delta).$
\begin{defn}
   We denote by $$\delta_I^*:\set_{I\Delta}\to \set_\Delta$$
    the functor induced by precomposition with $\delta_I$. We call this the \textit{diagonal functor}. When $I = \{1,\dots,n\}$, we denote $\delta_I^*$ by $\delta_n^*$.
\end{defn}
For the following definition, we use notation that is not used in \cite{liu_zheng_six_ops}. Note however that what we call cartesian multisimplicial nerves, can be obtained as a special case of \cite[Definition 1.3.14]{liu_zheng_six_ops} (c.f. Definition 1.3.16 of loc. cit.) and the functor $\op^{I}_J$ introduced on page 30 of loc. cit.
\begin{defn}\label{defn:multisimplicial_nerve} 
Let $\C$ be an $\infty$-category and $I_1,\dots, I_k$ sets of morphisms in $\C$, and $\epsilon_1,\dots,\epsilon_k$ such that each $\epsilon_i$ is either the superscript ``$\op$'' or the empty superscript. Then we denote by  $\C(I_1^{\epsilon_1},\dots,I_k^{\epsilon_k})$ the $k$-simplicial set whose $(n_1,\dots, n_k)$-simplices are functors
$$(\Delta^{n_1})^{\epsilon_1} \times \dots \times (\Delta^{n_k})^{\epsilon_k} \to \C,$$
or equivalently, $n_1$-by-$\dots$-by-$n_k$ $k$-dimensional hypercubes in $\C$ for which the edges in the $i$-th direction are in $I_i^{\epsilon_i}$, such that the 1-dimensional faces are cartesian squares in $\C$. We call such a $k$-simplicial set a \textit{cartesian multisimplicial nerve}.
\end{defn}
For example, let be $\C$ an $\infty$-category and $I_1,I_2$ sets of morphisms in $\C$. Then the 2-simplicial set $\C(I_1^\op,I_2)$ has as 1-simplices squares
\begin{center}
    \begin{tikzcd}
        X_{00} \arrow[d] & X_{10}\arrow[d] \arrow[l]\\
        X_{01}& X_{11}\arrow[l]
    \end{tikzcd}
\end{center}
such that the horizontal morphisms are in $I_1$, the vertical morphisms are in $I_2$, and $X_{10}$ is the pullback of the cospan
$$X_{00}\to X_{01} \leftarrow X_{11}.$$
The faces of such a 1-simplex are the 0-simplices $X_{00}$ and $X_{11}$.

We single out the following ``simple edges" in simplicial sets of the form $\delta_k^*\C(I_1^{\e_1},\dots,I_k^{\e_1})$ for an $\infty$-category $\C$ with sets of edges and superscripts as in Definition \ref{defn:multisimplicial_nerve}. We observe that a 1-simplex of $\delta_k^*\C(I_1^{\e_1},\dots,I_k^{\e_k})$ is a $k$-dimensional hypercube in $\C$ where the faces are cartesian squares and the edges in direction $i$ are maps in $I_i^{\e_i}$ for $i=1,\dots, k$. 
\begin{defn}\label{defn:simple_edge}
A 1-simplex of the simplicial set $\delta_k^*\C(I_1^{\e_1},\dots,I_k^{\e_k})$ is called $i$-\textit{simple} if the maps in all directions except direction $i$ are identities. Therefore such an edge can be identified with a morphism in $I_i^{\e_i}$. A 1-simplex is called \textit{simple} if it is $i$-simple for some $i$.\end{defn}

Let $C(\Delta^n)\subseteq (\Delta^n)^\op \times \Delta^n$ be the full subcategory on objects $(i,j)$ with $i\geq j$; in other words, as a graph this 1-category looks like the above-diagonal part of an $(n\times n)$-grid. For example, $C(\Delta^2)$ looks like
\begin{center}
   
\begin{tikzcd}
    (0,0) &  \arrow[l] (1,0)\arrow[d]\\
    & (1,1) 
\end{tikzcd} 
\end{center}

\begin{defn}\label{defn:diagonal_plus}
Let $\C$ be an $\infty$-category and $I_1$ and $I_2$ sets of morphisms in $\C$. We denote by 
$$\corr(\C,I_1,I_2)$$  the \textit{simplicial set of correspondences} whose $n$-simplices are given by functors
$$C(\Delta^n) \to \C $$
such that horizontal edges are sent to morphisms in $I_1$, vertical edges are sent to morphisms in $I_2$, and such that the little squares are pullback squares.  
\end{defn}
\begin{defn}\label{defn:simple_edge_in_corr_cat}
    A 1-simplex in $\corr(\C,I_1,I_2)$ is called $1$-simple if it is of the form 
    \begin{center}
        \begin{tikzcd}
    X &  \arrow[l, "f"] Y\arrow[d, "="]\\
    & Y 
\end{tikzcd}
    \end{center}
    for $f$ a morphism in $I_1$, and 2-simple if it is of the form 
\begin{center}
        \begin{tikzcd}
    X &  \arrow[l, "="'] X\arrow[d, "g"]\\
    & Y 
\end{tikzcd}
    \end{center}
    for $g$ a morphism in $I_2$. 
\end{defn}
The simplicial set $\corr(\C,I_1,I_2)$ is also denoted by $\C_{\textup{corr:}I_1,I_2}$ in e.g. \cite{gaitsgory_rozenblyum} and \cite{liu_zheng_six_ops}.
If $I_1$ and $I_2$ are both closed under composition, and stable under pullbacks along arbitrary morphisms, then $\corr(\C,I_1,I_2)$ is an $\infty$-category by \cite[Lemma 6.1.2]{liu_zheng_six_ops}. Moreover, there is a natural map of simplicial sets
$$\delta^*_2\C(I_1^\op, I_2) \to \corr(\C,I_1,I_2)$$
that, on 1-simplices, remembers the upper diagonal half of a square. By \cite[Theorem 1.4.26]{liu_zheng_six_ops}, this map is a categorical equivalence (see also Example 1.4.29 in loc. cit.)
\begin{nota}\label{nota:corr}
    As a special case of Definition \ref{defn:diagonal_plus}, for $\C$ an $\infty$-category, we denote $\corr(\C,All,All)$ by $\corr(\C)$. For $E$ a set of morphisms in $\C$ that is closed under composition and pullback along arbitrary morphisms, we denote $\corr(\C,E,All)$ by $\corr(\C,E)$.
\end{nota}

If $\C$ has products, then $\corr(\C^{\op,\sqcup,\op},E_-)$ is a symmetric monoidal $\infty$-category, and its underlying $\infty$-category is $\corr(\C,E)$, by \cite[Lemma 6.1.3]{liu_zheng_six_ops}. More generally, this holds for $\corr(\C^{\op,\sqcup,\op},I_{1,\times},I_{2,-})$ if the subcategory of $\C$ spanned by $I_1$-morphisms admits finite products. This justifies the following notation.
\begin{nota}
    Let $\C$ an $\infty$-category and $I_1$, $I_2$ sets of morphisms that are closed under composition and stable under pullbacks along arbitrary morphisms, and such that the subcategory of $\C$ spanned by $I_1$-morphisms admits finite products. 
    \begin{enumerate}[label=(\arabic*)]
        \item  We denote the symmetric monoidal $\infty$-category $\corr(\C^{\op,\sqcup,\op},I_{1,\times},I_{2,-})$  by $\corr(\C,I_1,I_2)^\otimes$.
        \item  If $\C$ admits all products and $E$ is closed under composition and pullback along arbitrary morphisms, then we denote the symmetric monoidal $\infty$-category $\corr(\C^{\op,\sqcup,\op},All_\times,E_-)$ by $\corr(\C,E)^\otimes$.
        \item   If $\C$ admits all products then we denote the symmetric monoidal $\infty$-category $\corr(\C^{\op, \sqcup,\op})(All_\times, All_-)$ by $\corr(\C)^\otimes$.
    \end{enumerate}
  
\end{nota}

\begin{defn}\label{defn:square}
    For $X$ an arbitrary simplicial set, by a \textit{commutative square} in $X$ we mean a map of simplicial set $\Delta^1\times \Delta^1 \to X$. Via the two natural inclusions
$$\Delta^1 \cong \Delta^1\times \Delta^0 \to \Delta^1\times \Delta^1 $$
and the two natural inclusions
$$\Delta^1 \cong \Delta^0\times \Delta^1 \to \Delta^1\times \Delta^1 $$
induced by the two inclusions of $\Delta^0$ in $\Delta^1$, a square in $X$ determines four 1-simplices in $X$.
\end{defn}

\subsection{The Lurie tensor product}\label{subsect:prereq2}
We denote by $\cat^\otimes \to \fin_\part$ the cocartesian fibration that classifies the cartesian symmetric monoidal structure on the $\infty$-category of $\infty$-categories $\cat$. We recall the following definition from \cite[Section 4.1]{DAGII}.
\begin{defn}
    We define $\Prcat^\otimes \subseteq\cat^\otimes$ to be the subcategory spanned by 
    \begin{enumerate}
        \item[(i)] tuples $(X_i)_I$ where each $X_i$ is a presentable $\infty$-category, 
        \item[(ii)] and morphisms $(X_i)_I \to (Y_j)_J$ over $\alpha:I \dashrightarrow J$ in $\fin_\part$, such that each induced functor 
        $$\prod_{\alpha^{-1}(j)} X_i \to Y_j $$
        preserves colimits in every variable.
    \end{enumerate}
\end{defn}
The $\infty$-category $\Prcat^\otimes$ with the natural map to $\fin_\part$ is a symmetric monoidal $\infty$-category. However, $\Prcat^\otimes$ is not cartesian; the symmetric monoidal product of presentable $\infty$-categories that it classifies is the \textit{Lurie tensor product} $\otimes$. For $\C,\D$ presentable $\infty$-categories, $\C \otimes \D$ is defined by the universal property that there is a map 
$$f:\C\times \D \to \C \otimes \D $$
that preserves colimits in every variable, and any map $F:\C \times \D \to \mathcal{E}$ that preserves colimits in every variable, factors uniquely through $f$ as a colimit-preserving functor $\C\otimes \D \to \E$.
\begin{rmk}
    Concretely, $\C\otimes \D$ is given by the presentable $\infty$-category $\textup{Fun}^R(\C^\op, \D)$ of functors $\C^\op \to \D$ that preserve small limits, see \cite[Proposition 4.8.1.17]{HA}.

\end{rmk} 
The underlying $\infty$-category $\Prcat$ is the subcategory of $\cat$ spanned by presentable $\infty$-categories and functors that preserve colimits. In particular, all functors in $\Prcat$ are left adjoints.

The inclusion $\Prcat^\otimes \subseteq \cat^\otimes$ is an $\infty$-operad map, and therefore induces an embedding of categories of operad maps
$$\alg_{\C}(\Prcat) \subseteq \alg_\C(\cat) $$
for $\C^\otimes$ any symmetric monoidal $\infty$-category. In the following lemma, we characterize the image of this embedding; in other words, we characterize which $\infty$-operad maps $\C^\otimes \to \cat^\otimes$ factor through $\Prcat^\otimes$.

\begin{lem}\label{lem:lift}
    Let $\C^\otimes$ be a symmetric monoidal $\infty$-category. Then $\Alg_\C(\Prcat)$ is equivalent to the subcategory of $ \alg_\C(\cat)$ spanned by $F:\C^\otimes\to \cat^\otimes$ such that
\begin{enumerate}[label=(\roman*)]
    \item for all $X$ in $\C$, $F(X)$ is a presentable $\infty$-category,
    \item for $f:X\to Y$ in $\C$, the map $F(f):F(X)\to F(Y)$ preserves colimits,
    \item for $X_1,\dots,X_n$ in $\C$, the map
    $$F(X_1)\times \dots \times F(X_n) \to F(X_1 \otimes \dots \otimes X_n) $$
    induced by the natural map $(X_1,\dots,X_1)\to X_1\otimes \dots \otimes X_1$, preserves colimits separately in each variable,
    \end{enumerate}
    and by natural transformations $\alpha:F\to Y$ that are component-wise in $\Prcat$. 
\end{lem}
\begin{proof}
    As stated above, the inclusion $\Prcat^\otimes\subseteq \cat^\otimes$ induces an embedding
$$i:\Alg_\C(\Prcat) \to \Alg_\C(\cat).$$
We show that the essential image of $i$ is the subcategory defined in the proposition, which we will denote by $\mathcal{A}$ in this proof. It is clear that the essential image of $i$ is contained in $\mathcal{A}$. On the other hand, let $F$ be in $\mathcal{A}$. Then $ F$ consists of: for each $X$ in $\C$ a presentable $\infty$-category $F(X)$ in $\Prcat$, and for $f:(X_i)_I\to (Y_j)_J$ over $\alpha:I \dashrightarrow J$ consisting of $\C$-morphisms $\otimes_{\alpha^{-1}(j)} X_i \to Y_j$, and each $j\in J$, a functor
\begin{equation}\label{eq:functor}
 \prod_{\alpha^{-1}(j)} F(X_i) \to F(Y_j).   
\end{equation}
We note that $f$ factors as 
$$(X_i)_I \to (\otimes_{\alpha^{-1}(j)} X_i)_J \to (Y_j)_J, $$
and therefore (\ref{eq:functor}) can be written as a composition 
$$ \prod_{\alpha^{-1}(j)} F(X_i) \to F(\otimes_{\alpha^{-1}(j)} X_i) \to F(Y_j).$$
Now by assumption, the first functor preserves colimits separately in each variable, and the second functor preserves colimits. Therefore $F(f)$ can be seen as a morphism in $\Prcat^\otimes$, and $F$ is an $\infty$-operad map 
$$F:\C_\otimes^\op \to \Prcat^\otimes$$
such that $iF = \tilde F$.

Similarly, for $\alpha:F\implies G$ a morphism in $\A$ and $\tilde\alpha:\tilde F \to \tilde G $ the associated morphism in $\Alg_\C(\cat)$, it is clear that this is actually a morphism in $\Alg_\C(\Prcat)$.
\end{proof}

\subsection{Adjoinable squares and recollements}\label{subsect:prereq_recollements}

We recall the following definition from \cite[Section 7.3.1]{htt}.
  Consider a commutative square of categories 
   \begin{center}
        \begin{tikzcd}
            A \arrow[r, "f"] \arrow[d, "g"] & B \arrow[d, "h"] \\
            C \arrow[r, "j" ] & D
        \end{tikzcd}
    \end{center}   
and suppose that $g$ and $h$ have right adjoints $g^R$ and $h^R$ respectively. Then the unit $1\to h^R h$ and the co-unit $g g^R \to 1$ give rise to a natural transformation
$$f g^R \to h^R h f g^R \cong h^R jg g^R \to h^Rj.$$

\begin{defn}
We say that a commutative square 
 \begin{center}
        \begin{tikzcd}
            A \arrow[r, "f"] \arrow[d, "g"] & B \arrow[d, "h"] \\
            C \arrow[r, "j" ] & D
        \end{tikzcd}
    \end{center} 
is \textit{vertically right-adjoinable} if right adjoints $g^R$ and $h^R$ exist, and the natural transformation $fg^R\to h^Rj$ is an isomorphism. Vertically/horizontally left/right-adjoinable squares are defined analogously.

Moreover, we say that the commutative squares 
 \begin{center}
        \begin{tikzcd}
            A \arrow[r, "f"] \arrow[d, "g"] & B \arrow[d, "h"] \\
            C \arrow[r, "j" ] & D
        \end{tikzcd} \hspace{5mm}  and   \hspace{5mm}
           \begin{tikzcd}
            C \arrow[r, "j"] \arrow[d, "g^R"] & D \arrow[d, "h^R"] \\
            A \arrow[r, "f" ] & B
        \end{tikzcd}
    \end{center}  
    (which implicitly commute up to natural equivalences $\eta:hf \to j g $ and $\epsilon: fg^R\to h^Rj$) are \textit{related by adjunction} if there exist adjunctions $g\vdash g^R$ and $h\vdash h^R$, such that $\epsilon$ is the composite 
$$ fg^R\to h^R hfg^R \to h^Rjgg^R \to h^Rj$$
obtained from the adjunctions and $\eta$. 
\end{defn}
\begin{defn}
    We say that a square of $\infty$-categories is vertically/horizontally left/right-adjoinable if the corresponding square of homotopy categories is.
\end{defn}
\begin{defn}
    We say that a square
    \begin{center}
        \begin{tikzcd}
            (\C_i)_I \arrow[d] \arrow[r] & (\D_j)_J \arrow[d] \\
            (\E_k)_K \arrow[r]& (\F_l)_L
        \end{tikzcd}
    \end{center}
in $\cat^\otimes$, over partial maps $\alpha:I\dashrightarrow J$, $\beta:I\dashrightarrow K$, $\gamma:J\dashrightarrow L$ and $\gamma:K \dashrightarrow L$ is \textit{vertically/horizontally left/right-adjoinable} if for every $l\in L$, the induced square of $\infty$-categories
\begin{center}
    \begin{tikzcd}
        \prod_{\epsilon^{-1}(l)} \C_i \arrow[r] \arrow[d] & \prod_{\gamma^{-1}(l)} \D_j \arrow[d]\\
        \prod_{\beta^{-1}(l)} \E_k \arrow[r] & \F_l
    \end{tikzcd}
\end{center}
is a vertically/horizontally left/right-adjoinable square in $\cat$.
\end{defn}

We denote by $\PrSt^\otimes \subseteq \Prcat^\otimes$ the full subcategory spanned by objects $(X_i)_I$ where every $X_i$ is a presentable stable $\infty$-category. This also defines a subcategory $\Prst \subseteq \cat$, spanned by presentable stable $\infty$-categories and left adjoint functors.

Occasionally, we will use the notation $\cat^{LL}$ or $\PrSt^{LL}$, to denote the wide subcategory of (presentable stable) $\infty$-category spanned by functors that are left adjoints, and whose right adjoint also has a right adjoint. We call these functors \textit{left-left} adjoints.

We will denote by $\cat^\ex\subseteq \cat$ the subcategory stable $\infty$-categories and exact functors; that is, functors that preserve finite limits or, equivalently, finite colimits.

\begin{defn}
    We call a sequence 
    $$ \A' \xrightarrow{f} \A \xrightarrow{p} \A''$$
    in $\cat^\ex$ a \textit{Verdier sequence} if it is both a fiber and a cofiber sequence. A Verdier sequence is a \textit{recollement} if $f$ and $p$ both a have a left and a right adjoint.
    \end{defn}
    For a recollement as above, we denote by $f^L$ and $p^L$ the left adjoints, and by $p^R$ and $f^R$ the right adjoints, respectively. By \cite[Lemma A.2.5]{nine}, the sequences 
$$A'' \xrightarrow{p^L} \A \xrightarrow{f^L} \A' $$
and 
$$A'' \xrightarrow{p^R} \A \xrightarrow{f^R} \A' $$
are also Verdier sequences. It follows that in a recollement the functors $f$, $p^L$ and $p^R$ are all fully faithful embeddings, and $f^L,f^R$ and $p$ are full and essentially surjective. We will call the sequence formed by the left adjoints of a recollement a \textit{left-recollement}, and the sequence formed by the right adjoints of a recollement a \textit{right-recollement}. 

For $ \A' \xrightarrow{f} \A \xrightarrow{p} \A''$ a recollement, from results in \cite[Section A.2]{nine} it follows that all the classical conditions for a recollement of triangulated categories are satisfied, such as
\begin{itemize}
    \item for $A$ in $\A$, there are bifiber sequences 
    $$f  f^R A \to A \to p^Rp A $$
    and
    $$p^L g p \to A \to f f^LA,$$
\item the unit $\id \to p p^L$ and co-unit $p p^R\to \id$ are isomorphisms,
\item the unit $\id \to f^Rf$ and co-unit $f^Lf \to \id$ are isomorphisms.
\end{itemize}

We collect a few lemmas about adjoinable squares in $\cat^\ex$. 
\begin{lem}\label{lem:induced_adjoint_0}
       Let 
    \begin{center}
    \begin{tikzcd}
                \A' \arrow[r, "f"] \arrow[d, "{\phi}"] & \A \arrow[r, "p"] \arrow[d, "\psi"] & \A'' \arrow[d, "\chi"] \\
        \B'\arrow[r, "g" ] & \B \arrow[r, "q "] & \B ''
    \end{tikzcd}
    \end{center}
    be a commuting diagram in $\cat^\ex$, where both rows are fiber sequences. If $\psi$ and $\chi$ have right adjoints $\psi^R$ and $\chi^R$ (or left adjoints $\psi^L$ and $\chi^L$), and the square on the right is vertically right (left) adjoinable, then $\phi$ has a right adjoint $\phi^R$ (a left adjoint $\phi^L$), and the square on the left is vertically right (left) adjoinable.
    
\end{lem}
\begin{proof} 
We proof the statement about right adjoints. We can form the diagram
  \begin{center}
    \begin{tikzcd}
                \B'\arrow[r, "g" ] \arrow[d, dashed] & \B  \arrow[r, "q "] \arrow[d, "\psi^R"] & \B ''   \arrow[d, "\chi^R"]     \\
   \A' \arrow[r, "f"]  & \A \arrow[r, "p"]  & \A''
    \end{tikzcd}
    \end{center}
   such that there is an induced functor between the fibers, which we denote $\phi^R:\A'\to \B'$. Then $\phi^R$ is the right adjoint of $\phi$. Indeed, for $a$ in $\A'$ and $b$ in $\B'$, since $f$ and $g$ are fully faithful, we have
   $$\hom_{\B'}(\phi(a),b) = \hom_{\B}(g\phi(a),g(b)) = \hom_\B(\psi f(a),g(b)) = \hom_\A(f(a),\psi^Rg(b)) $$
   $$= \hom_\A(f(a),f \phi^R(b)) = \hom_{\A'}(a,\phi^R(b)).  $$
\end{proof}

\begin{lem}\label{lem:cube}

    Consider a cube
    \begin{center}
\begin{tikzcd}
 & \B \arrow[rr, "b"] \arrow[dd, "h",near start] &                                                & \B' \arrow[dd, equal] \\
\A \arrow[ru, "g"] \arrow[rr, crossing over, "a"', near start ] \arrow[dd, "f"'] &                                                  & \A' \arrow[ru, "g'"] 
& \\
& \D \arrow[rr, "d", near start]                               &                                                & \B'            \\
\C \arrow[rr, "c"'] \arrow[ru, "j"]                                &                                                  & \A' \arrow[ru, "g'"'] \arrow[uu,  crossing over, equal ] &               
\end{tikzcd} \end{center}
    in $\cat^\ex$, where we assume that $a,b,c,d,f$ and $h$ have right adjoints, denoted  $a^R,b^R,c^R,d^R,f^R$ and $h^R$ respectively. Moreover we assume that $c^R$ is full and essentially surjective, and that $(a^R,b^R)$ commutes with $(g,g')$, and $(c^R,d^R)$ commutes with $(j,g')$. Then the leftmost face is vertically right-adjoinable. 
\end{lem}
\begin{proof}
    To show that $gf^R=h^Rj$, it suffices to show that $gf^R c^R=h^Rj c^R$. The left-hand side we can write as $ga^R = b^Rg'$, and the right-hand side as $h^Rd^Rg' = b^Rg'$.
\end{proof}

\begin{lem}\label{lem:another_cube}
    Consider a cube
        \begin{center}
\begin{tikzcd}
 & \B \arrow[rr, "b"] \arrow[dd, "h",near start] &                                                & \B' \arrow[dd, "h'", near start] \\
\A \arrow[ru, "g"] \arrow[rr, crossing over, "a"', near start ] \arrow[dd, "f"', near start] &                                                  & \A' \arrow[ru, "g'"]  
& \\
& \D \arrow[rr, "d", near start]                               &                                                & \D'            \\
\C \arrow[rr, "c"'] \arrow[ru, "j"]                                &                                                  & \C' \arrow[ru, "g'"'] \arrow[uu,  crossing over, leftarrow, "f'"', near end ]  &               
\end{tikzcd} \end{center}
in $\cat^\ex$, assume that $f,g,h, j$ have right adjoints $f^R,g^R, h^R$ and $j^R$ respectively, that the front, back and rightmost face are vertically right-adjoinable, and that $b$ is a fully faithful embedding. Then the leftmost face is vertically right-adjoinable.
\end{lem}
\begin{proof}
    Since $b$ is a fully faithful embedding, it suffices to check that 
    $$bgf^R = bh^Rj.$$
    Now the left-hand side can be written as $g'af^R = g'f'^Rc = h'^R j' c$,
    and the right-hand side as $h'^R dj= h'^R j 'c$.
\end{proof}

The proof of the following lemma is essentially the proof of \cite[Theorem 2.5]{parshall_scott}. 
\begin{lem}\label{lem:parshall_scott}
 Suppose we have the following diagram of adjoint triples and stable $\infty$-categories
 \begin{equation*}
\begin{tikzcd}[sep=large]
\A' \arrow[d, "\phi" description] \arrow[r, "f" description]                                         & \A \arrow[d, "\psi" description] \arrow[r, "g" description] \arrow[l, "f^L"', shift right=2] \arrow[l, "f^R", shift left=2]                                         & \A'' \arrow[d, "\chi" description] \arrow[l, "g^R", shift left=2] \arrow[l, "g^L"', shift right=2]                                         \\
\B' \arrow[r, "h" description] \arrow[u, "\phi^R"', shift right=2] \arrow[u, "\phi^L", shift left=2] & \B \arrow[r, "j" description] \arrow[l, "h^R", shift left=2] \arrow[l, "h^L"', shift right=2] \arrow[u, "\psi^R"', shift right=2] \arrow[u, "\psi^L", shift left=2] & \B'' \arrow[l, "j^L"', shift right=2] \arrow[l, "j^R", shift left=2] \arrow[u, "\chi^R"', shift right=2] \arrow[u, "\chi^L", shift left=2]
\end{tikzcd}
\end{equation*}
 where the rows formed by $f,g$ and by $h,j$ are recollements. We assume that the diagram of maps in $\cat^{LR}$ commutes, in other words, $(\phi,\psi)$ commutes with $(f, h)$ and $(\psi,\chi)$ commutes with $(g,j)$ (or equivalently, the diagram of left adjoints or the diagram of right adjoints commutes). Then 
 \begin{itemize}
     \item[(1)] $(\psi,\chi)$ commutes with $(g^L,j^L)$ if and only if $(\phi, \psi)$ commutes with $(f^L,h^L)$,
     \item[(2)] $(\psi,\chi)$ commutes with $(g^R,j^R)$ if and only if $(\phi, \psi)$ commutes with $(f^R,h^R)$,    
     \item[(3)] $(\psi^L,\chi^L)$ commutes with $(g,l)$ if and only if $(\phi^L,\psi^L)$ commutes with $(f,h)$,  
     \item[(4)] $(\psi^R,\chi^R)$ commutes with $(g,l)$ if and only if $(\phi^R,\psi^R)$ commutes with $(f,h)$. 
 \end{itemize}   
\end{lem}
\begin{proof}
The proofs of all statements are very similar, so we proof the if-direction of (1). Since the upper row is a recollement, for $A$ an object of $\A$ there is a bifiber sequence 
$$g^L g A \to A \to f f^LA.$$
Applying $\psi$ gives a bifiber sequence
$$\psi g^L g A \to \psi A \to \psi f f^LA.$$
By commutativity the last term equals $h\phi f^L A$, and using that $(\phi, \psi)$ commutes with $(f^L,h^L)$, we can rewrite this is $h h^L \psi A$. There is also a bifiber sequence
$$j^L j \psi A \to \psi A \to h h^L \psi  A,$$
which shows that $\psi g^L g A$ must be equivalent to $j^L j \psi A$. By commutativity, the latter is equivalent to $j^L\chi g A$. Now since $g$ is essentially surjective, this shows that the natural transformation $j^L\chi \to \psi g^L$ is an equivalence.
\end{proof}

\subsection{Functorially passing to adjoints}\label{subsect:prereq_passing_to_adjoints}
In this subsection we prove a generalization of \cite[Proposition 2.2.4]{liu_zheng_six_ops}, \textit{Partial Adjoint}. We first state informally what this proposition says, in the case that $R$ is a cartesian multisimplicial nerve $\C(I_1,\dots, I_k)$, i.e., a $k$-simplicial set as defined in Definition \ref{defn:multisimplicial_nerve} for $\C$ an $\infty$-category and $I_1,\dots, I_k$ sets of morphisms.  

Consider a functor 
$$F:\delta_k^*\C(I_1,\dots, I_k)  \to \cat.$$
Suppose that for some set of indices $ J \subseteq \{1, \dots, k \}$, $F$ sends $j$-simple edges to left (or right) adjoints for all $j$ in $J$. Moreover, suppose that $F$ sends squares in $\C(I_1,\dots, I_k)$, whose vertical edges are  $j$-simple edges for some $j\in J$, and whose horizontal edges are $i$-simple edges for some $i \notin J$, to vertically left (or right) adjoinable squares. Note that such a square in $\C(I_1,\dots, I_k)$ can be identified with a pullback square in $\C$, whose vertical morphism are in $I_j$, and whose horizontal morphisms are in $I_i$. Then the Partial Adjoint proposition constructs a functor
$$F':\delta_k^*\C(I_1,\dots, I_j^\op, \dots I_k) \to \cat,$$
with $\C(I_1,\dots, I_j^\op, \dots I_k)$ the cartesian multisimplicial nerve of $\C$ and the classes of morphisms $I_i$ for $i\notin J$, and $I_j^\op$ for $j\in J$. For $j$ in $J$, $F'$ sends a $j$-simple edge in $\C(I_1,\dots, I_j^\op, \dots I_k)$, which can be identified with a morphism $f$ in $I_j^\op$, to the left (or right) adjoint of $F(f)$. Moreover, a square with vertical $j$-simple morphisms, and horizontal $i$-simple morphisms for $i\notin J$, is sent to the left (or right) adjoint of the square that is the $F$-image of the corresponding square in $\C(I_1,\dots, I_k)$. In other words, $F'$ is constructed from $F$ by ``passing to the adjoint'' on simple morphisms in the dimensions that are in $J \subseteq \{1,\dots, k\}$. 

Utilizing the idea of \cite[Remark 2.2.5(5)]{liu_zheng_six_ops}, we want to make this passing to adjoints {functorial}; in other words, we construct an equivalence of categories between the $\infty$-category of functors that satisfy the conditions of Partial Adjoint, and the $\infty$-category of functors that are the result of applying Partial Adjoint.
\begin{defn}
     Let $R$ be an $I$-simplicial set and $J\subseteq I$ a subset. Let $$\fun^{BC}_{J,\textup{left}}(\delta_I^*R,\cat) \subseteq \fun(\delta_I^*R,\cat)$$
    be the subcategory spanned by functors 
    $f:\delta_I^*R \to \cat$ such that
    \begin{itemize}
        \item every $i$-simple edge in $\delta_I^*R $ is sent to a left adjoint functor in $\cat$
        \item every square in $\delta_I^*R $ whose vertical edges are $j$-simple for some $j\in J$, and whose horizontal edges are $i$-simple for some $i\notin J$, is sent to a vertically right-adjoinable square in $\cat$,
    \end{itemize}
    and by natural transformations $\alpha:f\to g$ such that for $j\in J$, and any $j$-simple edge $e$ between vertices $x$ and $y$, the square
    \begin{center}
        \begin{tikzcd}
            f(x)\arrow[d, "f(e)"]\arrow[r, "\alpha_x"] & g(x) \arrow[d, "g(e)"]\\
            f(y) \arrow[r, "\alpha_y"] & g(y)
        \end{tikzcd}
    \end{center}
    is vertically right-adjoinable. Let $$\fun_{J,\textup{right}}(\delta_I^*R,\cat) \subseteq \fun(\delta_I^*R,\cat)$$ is defined analogously, with left adjoint replaced by right adjoint, and right-adjoinable replaced by left-adjoinable.
\end{defn}
For the notation $\op^I_J$ in the following proposition, see \cite[Section 1.3]{liu_zheng_six_ops}. In the case that $R=\C(I_1,\dots, I_k)$ as before, $\op^I_J R$ is the $k$-simplicial set $\C(I_1,\dots, I_j^\op, \dots I_k)$ as before: the cartesian multisimplicial nerve of $\C$ and the classes of morphisms $I_i$ for $i\notin J$, and $I_j^\op$ for $j\in J$.

\begin{lem}[Functorial Partial Adjoint]\label{lem:functorial_PA}
    For $R$ an $I$-simplicial set and $J\subseteq I$, there is an equivalence of categories 
    $$PA: \fun_{J,\textup{left}}(\delta_I^*R,\cat) \xrightarrow{\sim} \fun_{J,\textup{right}}(\delta_I^*\op^I_J R,\cat).$$
    For $f:\delta_I^*R \to \cat$ in $\fun_{J,\textup{left}}(\delta_I^*R,\cat)$, $PA(f)$ has the properties of $f_J$ in \cite[Proposition 2.2.4]{liu_zheng_six_ops}
\end{lem}
\begin{proof}
 We consider the evaluation functor
$$\textup{ev}:\fun_{J,\textup{left}}(\delta_I^*R,\cat) \times \delta_I^*R \to \cat.$$
We can rewrite the source of this functor as
$$\delta^*_{I \sqcup \{0\}}(\Fun_{J,\textup{left}}(\delta_I^*R,\cat) \boxtimes R),$$
in other words, the diagonal of the $I\sqcup \{0\}$-simplicial set that is the exterior product of the simplicial set $\Fun_{J,\textup{left}}(\delta_I^*R,\cat)$ and the $I$-simplicial set $R$  (see also \cite[Definition 1.3.7]{liu_zheng_six_ops}; the  $(n_i)_{i\in I\cup\{0\}}$-simplices of $\Fun_{J,\textup{left}}(\delta_I^*R,\cat) \boxtimes R$ are pairs of an $n_0$-simplex of $\fun_{J,\textup{left}}(\delta_I^*R,\cat)$ and a $(n_i)_{i\in I}$-simplex of $R$). Then the functor $\textup{ev}$ satisfies the conditions of the Partial Adjoint proposition with respect to the subset $J\subseteq I\cup\{0\}$. Applying  Partial Adjoint now gives a functor 
$$\textup{ev}':\delta^*_{I \sqcup \{0\},J}(\Fun_{J,\textup{left}}(\delta_I^*R,\cat) \boxtimes R) \to \cat.$$
We can rewrite $\delta^*_{I \sqcup \{0\},J}(\Fun_{J,\textup{left}}(\delta_I^*R,\cat) \boxtimes R)$ again as $\fun_{J,\textup{left}}(\delta_I^*R,\cat) \times \delta_I^*\op^I_J R $,
and it follows that $\textup{ev}'$ induces a functor
$$PA:\fun_{J,\textup{left}}(\delta_I^*R,\cat)\to \fun(\delta_I^*\op^I_J R, \cat).$$
From the construction it follows that this functor lands in $\fun_{J,\textup{right}}(\delta_I^*\op^I_J R,\cat)$. To see that this gives the desired equivalence, we just need to observe that the process can be reversed, to produce an inverse
 $$PA^{-1}:\fun_{J,\textup{right}}(\delta_I^*\op^I_J R,\cat)  \rightarrow\fun_{J,\textup{left}}(\delta_I^*R,\cat).$$
\end{proof}

\section{Nagata six-functor formalisms}\label{sect:reducingtospan}
In this section we recall the definitions of abstract three- and six-functor formalisms from \cite{mann_thesis}. Then we specialize to a class of six-functor formalisms that we call Nagata. These are six-functor formalisms for which the adjoints $f^*\vdash f_*$ and $f_!\vdash f^!$ form adjoint triples when $f$ is in one of two special classes of morphisms $I$ and $P$. We then prove Proposition \ref{propx:intro0}, showing that a Nagata six-functor formalism is uniquely determined by its restriction to morphisms in $I$ and $P$. 

\subsection{Nagata setups and six-functor formalisms}\label{subsect:nagata_setup_and_6ff}
We start by expanding upon some of the definitions in \cite{mann_thesis}. 
\begin{defn}[{\cite[Definition A.5.1]{mann_thesis}}]
    A \textup{geometric setup} is a pair $(\C,E)$ where $\C$ is an $\infty$-category and $E$ a collection of (homotopy classes of) morphisms, such that
   \begin{itemize}
       \item[(i)] $E$ contains all isomorphisms and is stable under composition,
       \item[(ii)] pullbacks of morphisms in $E$ under arbitrary morphisms in $\C$ exists, and remain in $E$. 
   \end{itemize}
\end{defn}
\begin{defn}\label{defn:Nagata_setup}
    A \textit{Nagata setup} $(\C,E,I,P)$ is a geometric setup $(\C,E)$ with $I, P \subseteq E$ such that 
    \begin{enumerate}[label=(\arabic*)]
        \item $\C$ has finite products,
        \item $(\C,I)$ and $(\C,P)$ are geometric setups, we often denote morphisms in $I$ by $\hookrightarrow $ and morphisms in $P$ by $\twoheadrightarrow$,
        \item for any morphism $f\in P$ or $f\in I$, there exists $n\geq 2$ such that $f$ is $n$-truncated,
        \item for any morphism $f\in E$, there are $i\in I$ and $p\in P$ such that $f=p\circ i$,
         \item given $f:X\to Y$ in $\C$ and $g:Y\hookrightarrow Z$ in $I$, we have $f\in I$ if and only if $g\circ f$ in $I$,
        \item given $f:X\to Y$ in $\C$ and $g:Y\twoheadrightarrow Z$ in $P$, we have $f\in P$ if and only if $g\circ f$ in $P$.
    \end{enumerate}
\end{defn}
We recall that in an $\infty$-category, isomorphisms are $-2$-truncated, and $f$ is $n$-truncated if and only if the diagonal $\Delta_f$ is $(n-1)$-truncated for $n\geq -1$ (see also \cite[Remark 5.5.6.12 and Lemma 5.5.6.14]{htt}). In particular, monomorphisms are $-1$-truncated, and in a 1-category, ever morphism is at most 0-truncated.
\begin{example} 
    A few examples are \begin{enumerate}[label=(\alph*)]
        \item the 1-category $\var_k$ of varieties over a field $k$, with $E$ all morphisms, $I$ the class of open immersions and $P$ the class of proper morphisms,
        \item or more generally the 1-category $\textup{Sch}^{\textup{sf}}_S$ of qcqs schemes over a base scheme $S$ and morphisms between them, with $E$ the separated finite type morphisms, and with $I$ the class of open immersions and $P$ the class of proper morphisms,
        \item the 1-category of locally compact Hausdorff, with $E$ the class of all morphisms, $I$ the class of open immersions and $P$ the class of proper maps (\cite[Section 7]{scholze_notes}),
        \item the $\infty$-category of noetherian derived schemes and morphisms between them, $E$ the separated morphisms of almost finite type, $I$ the class of open immersions and $P$ the class of proper maps (\cite[Section 8.6]{scholze_notes}). We note that open immersions are $-1$-truncated, and proper maps are 0-truncated since their diagonal is a closed immersion and therefore $-1$-truncated.
    \end{enumerate}
    More examples can be found in in \cite{scholze_notes}.
\end{example}
\begin{rmk}
    In the rest of this paper, we will only consider Nagata setups of the form $(\C,All,I,P)$ where $E=All$ contains all morphisms in $\C$. However, we define Nagata setups in this generality, because in \cite{conjecture} we study properties of six-functor formalisms on Nagata setups where $E$ does not need to contain all morphisms in $\C$.
\end{rmk}

\begin{defn}[{\cite[Definition A.5.6]{mann_thesis}}]
    For $(\C,E)$ a geometric setup, a \textit{three-functor formalism} (sometimes called an abstract three-functor formalism) is is a lax cartesian structure (see \cite[Definition 2.4.1.1]{HA})
$$D: \corr({\C}, E)^\otimes \to \cat$$ 
\end{defn}

In other words, this is a functor such that for $(X_i)_I$ in ${\C^{\op,\sqcup,\op}}$, the natural map
$$D((X_i)_I) \to \prod_I D(X_i) $$
is an isomorphism. Such a functor encodes for every $X$ in $\C$ a category $D(X)$, for $f:X\to Y$ in $E$ a functor $f_!:D(X)\to D(Y)$, and for any morphism $g:X \to Y$ in $\C$ a functor $g^*:D(Y)\to D(X)$. Moreover, the natural morphism $X \to (X,X)$ in ${\C^{\op,\sqcup,\op}}$ induces a tensor product
$$-\otimes - :D(X)\times D(X) \to D(X)$$
that makes $D(X)$ into a symmetric monoidal $\infty$-category. For any morphism $g:X \to Y$ in $\C$, the functor $g^*:D(Y)\to D(X)$ is strong symmetric monoidal.

\begin{defn}[{\cite[Definition A.5.7]{mann_thesis}}]
    For $(\C,E)$ a geometric setup, a \textit{six-functor formalism} (sometimes called an abstract six-functor formalism) is a lax cartesian structure $$D: \corr({\C}, E)^\otimes\to \cat$$
such that for any morphism $f$ in $\C$, $f^*$ has a right adjoint $f_*$, $f_!$ has a right adjoint $f^!$; and for any $X$ in $\C$ and $A$ in $D(X)$, the functor
       $$A\otimes -:D(X) \to D(X)$$
     has a right adjoint $\hom(A,-)$.
\end{defn}

For basic properties of three- and six-functor formalisms, we refer to \cite[Proposition A.5.8]{mann_thesis}.
\begin{rmk}
  For $D$ a three-functor formalism and $f:X\to Y$ a morphism in $\C$, we see that $f^*:D(Y) \to D(X)$ is strong symmetric monoidal, essentially because the canonical square
  \begin{center}
      \begin{tikzcd}
          (Y,Y)& Y \arrow[l]\\
          (X,X) \arrow[u, "{(f,f)}" ] & X \arrow[u, "f"'] \arrow[l]
      \end{tikzcd}
  \end{center}
is a commutative square in $\C^{\op,\sqcup} \subseteq \corr(\C,E)$. Moreover, if $f$ is a monomorphism, then the square
\begin{center}
    \begin{tikzcd}
        X \arrow[d, "{f}"'] \arrow[r, "\Delta"] & X \times X \arrow[d, "{(f,f)}"] \\
        Y \arrow[r, "\Delta"] & Y\times Y
    \end{tikzcd}
\end{center}
is a pullback in $\C$, which implies that the square 
\begin{center}
    \begin{tikzcd}
        (X,X) \arrow[d, "{(f,f)}"'] & X \arrow[l] \arrow[d, "f"]\\
        (Y,Y) & Y \arrow[l]
    \end{tikzcd}
\end{center}
  is a commutative square in $\corr(\C,E)$. Therefore if $f$ is a monomorphism, then $f_!$ commutes with the tensor product. 
\end{rmk}

\subsection{Definition of Nagata six-functor formalisms}\label{subsect:nagata_6ff}

In this subsection, we fix a Nagata setup $(\C,All,I,P)$.  

\begin{defn}\label{defn:square_type}
    Let $X$ be a simplicial set let $A, B$ be collections of edges in $X$. Then a \textit{square of type $(A,B)$} in $X$ is a commutative square (see Definition \ref{defn:square})
    \begin{center}
        \begin{tikzcd}
            X \arrow[r, "f"]\arrow[d, "g"] & Y \arrow[d, "h"] \\
            Z \arrow[r, "j"]& V
        \end{tikzcd}
    \end{center}
    where the vertical edges $g$ and $h$ are in $A$, and the horizontal edges $f$ and $j$ are in $B$.  
\end{defn}

\begin{rmk} 
   Consider the collections of morphisms $I_-$ and $P_-$ in $\C^{\op,\sqcup,\op}$, and $I_-^\op$ and $P_\times^\op$ in $\C^{\op,\sqcup}$, see Notation \ref{nota:morphisms}. Recall that $\corr(\C)^\otimes$ contains the morphisms $All_-$ in $\C^{\op,\sqcup,\op}$, and that $\C^{\op,\sqcup}$ includes into $\corr(\C)^\otimes$. Therefore we can see $I_-,P_-,I_-^\op$ and $P_\times^\op$ as collections of morphisms in $\corr(\C)^\otimes$. 
   \begin{itemize}
       \item[(a)] Explicitly, for example, a morphism in $I_-\subseteq \corr(\C)^\otimes$ looks like a correspondence
$$(X_i)_I \xleftarrow{=}(X_i)_I\to (Y_i)_I$$
where the right-pointing morphism lies over $\textup{id}_I$ and the components $X_i\to Y_i$ are in $I$.
\item[(b)]  On the other hand, a morphism in $P_\times ^\op\subseteq \corr(\C)^\otimes$ looks like a correspondence
$$(X_i)_I\leftarrow (Y_j)_J \xrightarrow{=}(Y_j)_J $$
where the left-pointing morphism lies over a morphism $\alpha:I\to J$ in $\fin_\part$ and each map
$$Y_j \to \prod_{\alpha^{-1}(j)}X_i$$ is in $P$. 
   \end{itemize}
   
In the upcoming definition, we speak of squares of types $(I_-^\op, I_-)$, $(I_-^\op, P_\times^\op)$, $(P_-,P_\times^\op)$ and $(P_-,I_-)$ in $\corr(\C)^\otimes$.
\begin{itemize}
    \item[(c)] As an example, a square of type $(P_-,P_\times^\op)$ in $\corr(\C)^\otimes$ is a square of the form 
\begin{center}
    \begin{tikzcd}
        (X_i)_I \arrow[d] & (Y_j)_J\arrow[d] \arrow[l] \\
        (Z_i)_I & (V_j)_J \arrow[l]
    \end{tikzcd}
\end{center}
where vertical maps lie over identities in $\fin_\part$ and the horizontal maps over a map $\alpha:I\to J$ in $\fin_\part$, such that the components $X_i\to Z_i$ and $Y_j\to V_j$, as well as the induced maps 
$Y_j \to \prod_{\alpha^{-1}(j)} X_i$ and $V_j \to \prod_{\alpha^{-1}(j)} Z_i$ are in $P$.  Note that commutativity, as a square starting in $(X_i)_I$ and ending in $(V_j)_J$, implies that $(Y_j)_J$ is the pullback of the cospan
$(X_i)_I \to (Z_i)_I \leftarrow (V_j)_J.$
\end{itemize}
For $D:\corr(\C)^\otimes \to \cat$ a three-functor formalism, $D$ sends a square of type $(P_-,P_\times^\op)$ as above to the square
\begin{center}
    \begin{tikzcd}
        \prod_I D(X_i) \arrow[d] \arrow[r] & \prod_J D(Y_j) \arrow[d]\\
        \prod_I D(Z_i) \arrow[r] & \prod_J D(V_j)
        
    \end{tikzcd}
\end{center}
\end{rmk}

\begin{defn} \label{defn:newpre3ff}
 Given a Nagata setup $(\C,All,I,P)$, we denote by $\mathbf{3FF}\subseteq \fun(\corr(\C)^\otimes,\cat)$ the subcategory spanned by three-functor formalisms 
 $$D:{\corr(\C)^\otimes} \to \cat $$
 such that
  \begin{enumerate}

   \item[(A)]
$D$ sends squares of types  $(I_-^\op,I_-)$, $(I_-^\op,P_\times^\op)$, $(P_-,P_\times^\op)$ and $(P_-,I_-)$ to vertically left-adjoinable squares,
 \end{enumerate}
and by natural natural transformations $\alpha:D\to \tilde D$ such that for $f:X\to Y$ in $P$, the square
\begin{center}
        \begin{tikzcd}
        D(X) \arrow[r, "\alpha_X"] \arrow[d, "D_!(f)"'] & \tilde D(X) \arrow[d, "\tilde D_!(f)"]\\
        D(Y) \arrow[r, "\alpha_Y"'] & \tilde D(Y)
    \end{tikzcd}
\end{center}
is vertically left-adjoinable, and for $f:X\to Y$ in $I$, the square
\begin{center}
    \begin{tikzcd}
        D(Y) \arrow[d, "D^*(f)"' ]\arrow[r, "\alpha_Y"]& \tilde D(Y)\arrow[ d, "\tilde D^*(f)" ]
        \\
        D(X) \arrow[r, "\alpha_X" ] & \tilde D(X)
    \end{tikzcd}
\end{center}
is vertically left-adjoinable. We call objects in this category \textit{Nagata three-functor formalisms}

We denote by $\textup{\textbf{6FF}}\subseteq \mathbf{3FF}$ the full subcategory spanned by functors $$D:{\corr(\C)^\otimes} \to \cat$$
that are a six-functor formalism. We call such functors \textit{Nagata six-functor formalisms}.

\end{defn}

The following proposition shows that in a Nagata three-functor formalism we do indeed have Grothendieck and Wirthm\"uller contexts. 
\begin{prop}\label{prop:cohom_prop_etale}
    Let $(\C,All,I,P)$ be a Nagata setup, and let $D:{\corr(\C)^\otimes}  \to \cat$ in $\mathbf{3FF}$. Then for every $p$ in $P$ there is a canonical adjunction $p^ *\vdash p_!$, and for every $i$ in $I$ there is a canonical adjunction $i_!\vdash i^*$. 
\end{prop}
\begin{proof}
Recall that all morphisms in $P$ and all morphisms in $I$ are $n$-truncated for some $n\geq -2$. We argue by induction over $n$. First, let $f:X\to Y$ be an $n$-truncated morphism in $P$.  
 \begin{itemize}
 \item[(a)] Suppose $n=-2$. In that case, $f$ is an isomorphism, and $f_!$ trivially has a left-adjoint $f^\natural$. The square
        \begin{center}
            \begin{tikzcd}
              {D}(X)\ar[d, "f_!"] \ar[r, "="]& {D}(X) \ar[d, "="] \\
                {D}(Y) \ar[r, "f^*"] &{D}(X)
            \end{tikzcd}
        \end{center}
is the image of a square of type $(P_-, P_\times^\op)$, and therefore by condition (A) of Definition \ref{defn:newpre3ff} it is vertically left-adjoinable, giving a natural isomorphism
$f^*\to f^* f_!f^\natural \cong f^\natural$  and hence an adjunction $f^ *\vdash f_!$.
        \item[(b)] Now suppose that for the $(n-1)$-truncated $P$- morphism $\Delta_f$, there is a natural isomorphism $\epsilon:\Delta_{f}^*\xrightarrow{\cong} \Delta_{f}^\natural$.  The square
\begin{center}
    \begin{tikzcd}
        D(X)\arrow[d, "f_!"] \arrow[r, "p_2^*"] & D(X \times_Y X) \arrow[d, " p_{1,!}"] \\
        D(Y)\arrow[r, "f^*"] & D(X)
    \end{tikzcd}
\end{center}
again the image of a square of type $(P_-,P_\times^\op)$ and therefore vertically left-adjoinable, giving a natural isomorphism 
$$f^* \cong \Delta_f^\natural p_1^\natural f^* \xrightarrow[BC]{\cong} \Delta_f^\natural p_2^* f^\natural \xrightarrow[\epsilon^{-1}]{\cong}\Delta_f^* p_2^* f^\natural \cong f^\natural $$ and hence an adjunction $f^ *\vdash f_!$.
\end{itemize}
Now let $f:X\to Y$ be an $n$-truncated morphism in $I$.
\begin{itemize}
        \item[(a)] Suppose $n=-2$. Again, in this case $f$ is an isomorphism, so
        $f^*$ has a canonical left adjoint $f_\natural$. The square
        \begin{center}
            \begin{tikzcd}
                D(X) \arrow[r, "f_!"] \arrow[d, "="] & D(Y) \arrow[d, "f^* "] \\
                D(X) \arrow[r, "="] & D(X)
            \end{tikzcd}
        \end{center} is the image of a square of type $(I_-^\op,I_-)$, and therefore by condition (A) in Definition \ref{defn:newpre3ff} it
is vertically left-adjoinable, giving a natural isomorphism 
$f_\natural = f_\natural f^* f_! \to f_!$ and hence an adjunction $f_!\vdash f^*$.

        \item[(b)] Suppose that for the $(n-1)$-truncated $I$-morphism $\Delta_f$, there is a natural isomorphism $\epsilon:\Delta_{f,\natural}\to\Delta_{f,!}$. Again, the square
\begin{center}
    \begin{tikzcd}
        D(X)\arrow[r, "f_!"] \arrow[d, "p^*_2"] & D(Y) \arrow[d, "f^*"] \\
        D(X\times_Y X)\arrow[r, "p_{1,!}"] & D(X)
    \end{tikzcd}
\end{center} is the image of a square of type $(I_-^\op, I_-)$ and therefore is vertically left-adjoinable, giving a natural isomorphism
$$f_\natural = f_\natural p_{1,!}\Delta_{f,!} \xrightarrow[BC]{\cong} f_!p_{2,\natural} \Delta_{f,!} \xrightarrow[\epsilon^{-1}]{\cong} f_!p_{2,\natural} \Delta_{f,\natural} = f_!$$
and hence an adjunction $f_!\vdash f^*$.
\end{itemize}
\end{proof}
\begin{rmk}
    For $D$ a Nagata six-functor formalism, the proposition above shows that for $p$ in $P$, there is a natural isomorphism $p_*\simeq p_!$, and for $i$ in $I$, there is a natural isomorphism $i^!\simeq i^*$.
\end{rmk}
\begin{rmk}
    In some sense, the reverse of Proposition \ref{prop:cohom_prop_etale} is true as well. Given any three-functor formalism on a Nagata setup, one can attempt to inductively define adjunctions $f^*\vdash f_!$ for $f$ in $P$, and $f_!\vdash f^!$ for $f$ in $I$, using that morphisms in $P$ and in $I$ are truncated. If such an adjunction can indeed be constructed inductively, then we say that $f$ is \textup{cohomologically proper} or \textit{cohomologically \'etale}, respectively (see also \cite[Lecture VI]{scholze_notes}). In \cite[corollary]{conjecture}, we show the reverse: if, in a three-functor formalism $D$, morphisms in $P$ are cohomologically proper and morphisms in $I$ are cohomologically \'etale, then $D$ is Nagata. This implies that in Definition \ref{defn:newpre3ff}, only half of (A) is necessary: in fact it suffices to ask that squares of types $(I_-^\op,I_-)$ and $(P_-,P_\times^\op)$ are sent to vertically left-adjoinable squares, since this implies the rest of condition (A). 
\end{rmk}

\subsection{Reducing to morphisms in $I$ and $P$}\label{subsect:proper_open}
We will now analyze the $\infty$-categories $\mathbf{3FF}$ and $\textup{\textbf{6FF}}$ further, proving our main result about Nagata six-functor formalisms, Proposition \ref{propx:intro0}. As an intermediate step, we proof the analogous statement, Proposition \ref{prop:3FF}, for Nagata three-functor formalisms.

Consider the embedding 
$$\corr(\C,P,I)^\otimes \to \corr(\C)^\otimes.$$
We note that morphisms in $P_\times^\op$ (including morphisms in $P_-^\op$) and $I_-$ in $\corr(\C)^\otimes$ are in the image of this embedding; we can therefore consider these classes of morphisms in $\corr(\C,P,I)^\otimes$ too. Concretely, a morphism
$$ (X_i)_I \leftarrow (Y_j)_J \xrightarrow{=}(Y_j)_J$$
in $P_\times$ is a 1-simple morphism in $\corr(\C,P,I)^\otimes$, whereas a morphism 
$$(X_i)_I \xleftarrow{=}(X_i)_I \to (Y_i)_I $$
in $I_-$ is a 2-simple morphism in $\corr(\C,P,I)^\otimes$ (see Definition \ref{defn:simple_edge_in_corr_cat}.

Therefore we can speak of squares of types $(I_-,I_-)$, $(I_-,P_\times^\op)$, $(P_-^ \op,P_\times^\op)$ and $(P_-^\op,I_-)$ in $\corr(\C,P,I)^\otimes$. Note that these are not the same as the types of squares considered in Definition \ref{defn:newpre3ff}, as the direction of the vertical morphisms is reversed (the types of squares considered there are not in $\corr(\C,P,I)^\otimes$, since $\corr(\C,P,I)^\otimes$ does not contain morphisms in $I_-^\op$ or $P_-$).

\begin{defn}\label{defn:bc_span}
    Let $$\fun^\lax_{BC}(\corr(\C,P,I)^\otimes,\cat)\subseteq \fun(\corr(\C,P,I)^\otimes,\cat)$$ denote the subcategory spanned by lax cartesian structures
    $$F:\corr(\C,P,I)^\otimes \to \cat$$ such that
    \begin{enumerate}
        \item[(B)] $F$ send all squares in $\corr(\C,P,I)^\otimes$ of type $(I_-,I_-)$, $(I_-,P_\times^\op)$, $(P_-^ \op,P_\times^\op)$ and $(P_-^\op,I_-)$ to vertically right-adjoinable squares
    \end{enumerate}
     and by natural transformations $\alpha:F \to G$ that are compatible with the right adjoints, in other words, such that for $p:X\to Y$ in $P$, the square  
\begin{equation*}
        \begin{tikzcd}
            F(Y) \arrow[r, "\alpha_{Y}"] \arrow[d, "F(p)"'] & G(Y) \arrow[d, "G(p)"] \\
            F(X) \arrow[r, "\alpha_{X}"] & G(X)          
        \end{tikzcd}
    \end{equation*}
    is vertically right-adjoinable, and for $i:U\hookrightarrow X$ in $I$, the square 
\begin{equation*}
        \begin{tikzcd}
            F(U) \arrow[r, "\alpha_{U}"] \arrow[d, "F(i)"'] & G(U) \arrow[d, "G(i)"] \\
            F(X) \arrow[r, "\alpha_{X}"] & G(X)          
        \end{tikzcd}
    \end{equation*}
    is vertically right-adjoinable. We will call this the category of \textit{BC functors}, or \textit{functors with the BC-property}, where BC stands for Beck-Chevalley. 
\end{defn}
A functor $F$ in $\fun^\lax_{BC}(\corr(\C,P,I)^\otimes,\cat)$ can be thought as as assigning to every $X$ in $\C$ an $\infty$-category $F(X)$, and to every $P$-morphism an adjunction $p^*\vdash p_\natural$ and every $I$-morphism an adjunction $i_!\vdash i^\natural$. Proposition \ref{prop:w6FF} will imply that $F(X)$ is in fact a symmetric monoidal $\infty$-category. Moreover, condition (B) ensures that $(-)_!$ and $(-)_\natural$ are compatible and extend to a functor on $\C$, and similarly that $(-)^*$ and $(-)^\natural$ are compatible and extend to a functor on $\C^\op$. The latter uses that any morphism in $\C$ factors as a $I$-morphism followed by a $P$-morphism. 

The claims above are made precise in the following proposition, using the ideas in the proof of \cite[Proposition A.5.10]{mann_thesis}, and machinery developed by Liu and Zheng in \cite{liu_zheng_six_ops}.

\begin{prop}\label{prop:3FF}
Restricting along the embedding $e:\corr(\C,P,I)^\otimes \hookrightarrow {\corr(\C)^\otimes}$ induces an equivalence of $\infty$-categories 
$$e^*:  \mathbf{3FF} \xrightarrow{\sim} \fun^\lax_{BC}(\corr(\C,P,I)^\otimes,\cat).$$
\end{prop}

We will prove Proposition \ref{prop:3FF} by means of a chain of three equivalences, which are treated in Lemma \ref{lem:w6FF_step_3}, Lemma \ref{lem:w6FF_step_2.1} and Lemma \ref{lem:w6FF_step_1}. 

In the first Lemma, we consider the simplicial set $\delta_4^*\C^{\op,\sqcup,\op}(P_\times^\op,I_-^\op, I_-, P_-)$. Recall that its 1-simplices are 4-dimensional hypercubes in $\C^{\op,\sqcup,\op}$ with cartesian faces, where the morphisms in the first direction are in $P_\times^\op$, the morphisms in the second direction in $I_-^\op$, the third direction in $I_-$ and the fourth direction in $P_-$. Simple edges are hypercubes with identities in all but one of the directions. Therefore any morphism in $P_\times^\op$ has a corresponding 1-simple edge in $\delta_4^*\C^{\op,\sqcup,\op}(P_\times^\op,I_-^\op, I_-, P_-)$, a morphism in $I_-^\op$ corresponds to a 2-simple edge, and so on. Note that a morphism in $P_\times^\op \cap I_-^\op$ can be 1-simple or 2-simple as an edge in $\delta_4^*\C^{\op,\sqcup,\op}(P_\times^\op,I_-^\op, I_-, P_-)$, and a morphism in $I_-\cap P_-$ can be either 3-simple or 4-simple.

\begin{nota}
Let $\delta_k^*\C(I_1^{\epsilon_1},\dots,I_k^{\epsilon_k})$ be as in Definition \ref{defn:simple_edge}. Given a set of morphisms $A$ in $\C$, we denote by $A^i$ the set of $i$-simple morphisms in $\delta_k^*\C(I_1^{\epsilon_1},\dots,I_k^{\epsilon_k})$ for which the morphisms in the $i$th dimension are in $A$. By abuse of notation, we denote by $A$ also the set of simple morphisms $A_1\cup\dots \cup A_k$. 
\end{nota}
Thus, we have sets of morphisms $P_\times^{\op,1}$, $I_-^{\op,2}$, $I_-^3$ and $P_-^4$ in $\delta_4^*\C^{\op,\sqcup,\op}(P_\times^\op,I_-^\op, I_-,P_-).$ 

\begin{defn}\label{defn:last_BC}
    Let $$\fun_{BC}^\lax(\delta_4^*\C^{\op,\sqcup,\op}(P_\times^\op,I_-^\op, I_-,P_-), \cat)$$
    denote the subcategory of $\fun(\delta_4^*\C^{\op,\sqcup,\op}(P_\times^\op,I_-^\op, I_-,P_-), \cat)$ spanned by functors $F$ that send squares of types
   $(I_-^{\op,2},I_-^3)$, $(I_-^{\op,2},P_\times^{\op,1})$, $(P_-^4,P_\times^{\op,1})$ and $(P_-^4,I_-^3)$ to vertically left-adjoinable squares, and for which the natural morphisms
   \begin{equation} \label{eq:lax_cartesian_structure}
      F((X_i)_I)\to \prod_I F(X_i) 
   \end{equation}
   for $(X_i)_I$ in $\C^{\op,\sqcup,\op}$ are isomorphisms; and by natural transformations $\alpha:F \to G$ that are compatible with the left adjoints of images of $P_\times^{\op,1}$- and $I_-^{\op,2}$-morphisms, cf. Definition \ref{defn:newpre3ff}.
\end{defn}

\begin{lem}\label{lem:w6FF_step_3}
    There is an equivalence of categories
    $$ \psi^*: \mathbf{3FF} \xlongrightarrow{\sim} \fun^\lax_{BC}(\delta_4^*\C^{\op,\sqcup,\op}(P_\times^\op,I_-^\op, I_-,P_-), \cat).$$ 
\end{lem}
\begin{proof}
There is a categorical equivalence
    $$\psi_0: \delta_4^*\C^{\op,\sqcup,\op}(P_\times^\op,I_-^\op, I_-,P_-) \xlongrightarrow{\sim} \delta_2^*\C^{\op,\sqcup,\op}(All_\times^\op, All_-)$$
 This can be seen by applying \cite[Theorem 1.5.4]{liu_zheng_six_ops} twice, once with respect to the sets of morphisms $P_\times^{\op,1}, I_-^{\op,2}$ and  once with respect to the sets of morphisms $I_-^3,P_-^4$.  
 To check condition 1 in \cite[Theorem 1.5.4]{liu_zheng_six_ops}, let $f:(X_i)_I\to (Y_j)_J $ be a morphism in $All_\times$ over a partial map $\alpha:J \dashrightarrow I$, given by maps $f:X_i \to \prod_{\alpha^{-1}(i)} Y_i$. By assumption every $f_i$ can be factored as a map in $I$ followed by a map in $P$  
 $$X_i \hookrightarrow Z_i \twoheadrightarrow \prod_{\alpha^{-1}(i)} Y_i,$$
 and therefore $f$ factors as a map $(X_i)_I\hookrightarrow (Z_i)_I$  in $I_-$ and a map $(Z_i)_I\twoheadrightarrow (Y_j)_J$ in $P_\times$. Similarly, every map in $All_-$ factors as a map in $I_-$ followed by a map in $P_-$. The other conditions follow directly from the assumptions on $I$ and $P$.

Using \cite[Theorem 1.4.26]{liu_zheng_six_ops}, we see that there is a categorical equivalence of simplicial sets 
$$\psi_1:\delta_2^*\C^{\op,\sqcup,\op}(All_\times^\op, All_-) \to \corr(\C)^\otimes$$
It follows that $\psi:=\psi_1\circ\psi_0$ induces an equivalence of $\infty$-categories 
\begin{equation}\label{eq:equivalence_pre6ff_step_3}
    \psi^*:\fun({\corr(\C)^\otimes}, \cat) \xlongrightarrow{\sim} \fun(\delta_4^*\C^{\op,\sqcup,\op}(P_\times^\op,I_-^\op, I_-,P_-), \cat).
\end{equation}
It is clear that  $\psi^*$ preserves and reflects the existence of equivalences (\ref{eq:lax_cartesian_structure}). We observe that $\psi$ takes squares of types  $(I_-^{\op,2},I_-^3)$, $(I_-^{\op,2},P_\times^{\op,1})$, $(P_-^4,P_\times^{\op,1})$ and $(P_-^4,I_-^3)$ to squares of types  $(I_-^\op,I_-)$, $(I_-^\op,P_\times^\op)$, $(P_-,P_\times^\op)$ and $(P_-,I_-)$, and vice versa, every square of e.g. type $(I_-^\op,I_-)$ is the image of a unique square of type $(I_-^{\op,2},I_-^3)$, etc. Therefore $D$ in $\fun({\corr(\C)^{\op,\sqcup,\op}}, \cat)$ sends squares of types  $(I_-^\op,I_-)$, $(I_-^\op,P_\times^\op)$, $(P_-,P_\times^\op)$ and $(P_-,I_-)$ to vertically left-adjoinable squares if and only if $\psi^*D$ sends squares of types $(I_-^{\op,2},I_-^3)$, $(I_-^{\op,2},P_\times^{\op,1})$, $(P_-^4,P_\times^{\op,1})$ and $(P_-^4,I_-^3)$ to vertically left-adjointable squares. Similarly, since every morphism in $I_-^{\op, 2}$ (respectively $P_-^4$) in $\corr(\C)^\otimes$ is the image of an $I_-^\op$-morphism (respectively a $P_-$-morphism) in  $\delta_4^*\C^{\op,\sqcup,\op}(P_\times^\op,I_-^\op, I_-,P_-)$ , and therefore a morphism of three-functor formalisms is compatible with left adjoints if and only if $\psi^*D$ is. This shows that $\psi^*$ restricts to an equivalence  
$$\psi^ *:  \mathbf{3FF} \xlongrightarrow{\sim} \fun^\lax_{BC}(\delta_4^*\C^{\op,\sqcup,\op}(P_\times^\op,I_-^\op, I_-,P_-), \cat).$$ \end{proof}

In the next step we change the variance of some of the inputs of the functors, by passing to adjoints. First we introduce the following auxiliary $\infty$-category.

\begin{defn}\label{defn:BC_functor_cat}
   Let $$\fun^\lax_{BC}(\delta_4^*\C^{\op,\sqcup,\op}(P_\times^\op,I_-, I_-, P_-^\op),\cat)$$ denote the $\infty$-category spanned by functors $F:\delta_4^*\C^{\op,\sqcup,\op}(P_\times^\op,I_-, I_-, P_-^\op)\to \cat$ that send all squares of types $(I_-^2,I_-^3)$, $(I_-^2,P_\times^{\op,1})$, $(P_-^{\op,4},P_\times^{\op,1})$ and $(P_-^{\op,4},I_-^3)$ to vertically right-adjoinable squares, and for which the morphisms (\ref{eq:lax_cartesian_structure}) are equivalences; and by natural equivalences between them that are compatible with the right adjoints, cf. Definition \ref{defn:bc_span}. 
   \end{defn}

\begin{lem}\label{lem:w6FF_step_2.1}
There is an equivalence of $\infty$-categories
$$PA:  \fun^\lax_{BC}(\delta_4^*\C^{\op,\sqcup,\op}(P_\times^\op,I_-^\op, I_-,P_-), \cat) \xlongrightarrow{\sim} \fun^\lax_{BC}(\delta_4^*\C^{\op,\sqcup,\op}(P_\times^\op,I_-, I_-,P_-^\op), \cat).$$ 
\end{lem}
\begin{proof}
This follows directly from Lemma \ref{lem:functorial_PA}, where we observe that both $PA$ and $PA^{-1}$ preserve the existence of isomorphisms (\ref{eq:lax_cartesian_structure}).
\end{proof}

For the final equivalence, we apply the same results from \cite{liu_zheng_six_ops} as before. 

\begin{lem}\label{lem:w6FF_step_1} 
    There is an equivalence of $\infty$-categories 
\begin{equation*}
 \phi^*:\fun_{BC}^\lax(\corr(\C,P,I)^\otimes ,\cat) \xlongrightarrow{\sim} \fun^\lax_{BC}(\delta_4^*\C^{\op,\sqcup,\op}(P_\times^\op,I_-, I_-, P_-^\op) ,\cat)
\end{equation*}     
\end{lem}
\begin{proof}
 By \cite[Theorem 1.4.26]{liu_zheng_six_ops}, the natural map
$$\phi_0:\delta_2^*\C^{\op,\sqcup,\op}(P_\times^\op, I_-) \to \corr(\C,P,I)^\otimes $$
is a categorical equivalence. 

Let
 $$\phi_1:\delta_4^*\C^{\op,\sqcup,\op}(P_\times^\op,P_-^\op, I_-,I_-) \longrightarrow \delta_2^*\C^{\op,\sqcup,\op}(P_\times^\op, I_-)$$ be the map given in degree $n$ by precomposition with the product of diagonals   $$(\Delta^n)^\op \times \Delta^n \longrightarrow (\Delta^n)^\op \times (\Delta^n)^\op  \times \Delta^n \times \Delta^n.$$
The assumptions on $I$ and $P$ allow us to apply \cite[Theorem 1.5.4]{liu_zheng_six_ops} twice, with respect to the sets of morphisms $P_\times^\op, P_-^\op$ and $I_-,I_-$, showing that $\phi_1$ is a categorical equivalence.
 
Lastly, there is an isomorphism of simplicial sets 
$$ \phi_2: \delta_4^*\C^{\op,\sqcup,\op}(P_\times^\op,I_-, I_-,P_-^\op) \xlongrightarrow{\cong} \delta_4^*\C^{\op,\sqcup,\op}(P_\times^\op, P_-^\op, I_-,I_-) $$ given in degree $n$ by precomposition with the map 
     $$(\Delta^n)^\op \times (\Delta^n)^\op \times  \Delta^n \times \Delta^n \longrightarrow (\Delta^n)^\op  \times  \Delta^n \times \Delta^n \times (\Delta^n)^\op$$
     that simply permutes the factors in the product of categories in the obvious way (i.e. interchanging the second and the fourth).

Composing these maps gives a categorical equivalence 
$$\phi:=\phi_0\circ\phi_1\circ\phi_2:  \delta_4^*\C^{\op,\sqcup,\op}(P_\times^\op,I_-, I_-,P_-^\op) \to \corr(\C,P,I)^\otimes.$$
and precomposing with this map gives an equivalence of $\infty$-categories 
$$\phi^*:\fun(\corr(\C,P,I)^\otimes,\cat) \xlongrightarrow{\sim} \fun(\delta_4^*\C^{\op,\sqcup,\op}(P_\times^\op,I_-, I_-,P_-^\op),\cat).$$
Similar to the proof of Lemma \ref{lem:w6FF_step_3}, we observe that any $(I_-,I_-)$-square in $\corr(\C,P,I)^\otimes$ is the image of a $(I_-^2,I_-^3)$-square in $\delta_4^*\C^{\op,\sqcup,\op}(P_\times^\op,I_-, I_-,P_-^\op)$, etc. Therefore, for $F$ in $\fun(\corr(\C,P,I)^\otimes,\cat)$ and $G=\phi^*F$, we have that $F$ sends squares of types$(I_-,I_-)$, $(I_-,P_\times^\op)$, $(P_-^{\op},P_\times^{\op})$ and $(P_-^{\op},I_-)$ to vertically right-adjoinable squares if and only if $G$ sends squares of types $(I_-^2,I_-^3)$, $(I_-^2,P_\times^{\op,1})$, $(P_-^{\op,4},P_\times^{\op,1})$ and $(P_-^{\op,4},I_-^3)$ to vertically right-adjoinable squares.  
Similarly, for $\alpha$ a natural transformation in $\fun(\corr(\C,P,I)^\otimes,\cat)$, we have that $\alpha$ is in $ \fun_{BC}^\lax(\corr(\C,P,I)^\otimes ,\cat) $ if and only if $\psi^*\alpha$ is in $\fun^\lax_{BC}(\delta_4^*\C^{\op,\sqcup,\op}(P_\times^\op,I_-, I_-, P_-^\op) ,\cat)$.
\end{proof}
The existence of this final equivalence implies the following.
\begin{lem}\label{lem:meer_adjoinable_squares}
    For $F$ in $\fun^\lax_{BC}(\delta_4^*\C^{\op,\sqcup,\op}(P_\times^\op,I_-, I_-, P_-^\op),\cat)$, \textit{all} squares of types $(I_-,I_-)$, $(I_-,P_\times^\op)$, $(P_-^ \op,P_\times^\op)$ and $(P_-^\op,I_-)$ are sent to vertically right-adjoinable squares, including for example squares of type $(I_-^3, I_-^3)$ or of type $(P_-^{\op,4}, P_\times^{\op,1})$. 
\end{lem}
\begin{proof}
    For $F$ in  $\fun^\lax_{BC}(\delta_4^*\C^{\op,\sqcup,\op}(P_\times^\op,I_-, I_-, P_-^\op),\cat)$, we have that $F$ is naturally equivalent to a functor of the form $\phi^*H$ for some  $H$ in $\fun^\lax_{BC}(\corr(\C,P,I)^\otimes,\cat)$, by Lemma \ref{lem:w6FF_step_1}. Let us denote this natural equivalence by $\eta:F\to \phi^*H$. For $\phi^*H$ it is certainly the case that every square of one of the types $(I_-,I_-)$, $(I_-,P_\times^\op)$, $(P_-^ \op,P_\times^\op)$ and $(P_-^\op,I_-)$ is sent to a vertically right-adjoinable square, regardless of whether the type is e.g. $(I_-^2,I_-^3)$ or $(I_-^3, I_-^3)$. Indeed, for example for a morphism $i:U\to X$ in $I$, denoting the associated 2- and 3-simple morphisms by $(\id,i,\id,\id)$ and $(\id,\id,i,\id)$ respectively, we have $\phi^*H(\id,i,\id,\id)= \phi^*H(\id,\id,i,\id)$, since $\phi(\id,i,\id,\id) = \phi(\id,\id,i,\id)$. Since being adjoinable is a property stable under natural isomorphism, the same now holds for $F$. 
\end{proof}
To show how the equivalences proven in Lemma \ref{lem:w6FF_step_3}, Lemma \ref{lem:w6FF_step_2.1} and Lemma \ref{lem:w6FF_step_1} combine to an equivalence induced by restriction along $e:\corr(\C,P,I)^\otimes\hookrightarrow {\corr(\C)^\otimes}$, we need to show that the restriction of a Nagata three-functor formalism along $e$ is in $\fun_{BC}^\lax(\corr(\C,P,I)^\otimes ,\cat)$.
\begin{lem}\label{lem:nog_meer_adjoinable_squares}
    For $F$ in $\fun^\lax_{BC}(\delta_4^*\C^{\op,\sqcup,\op}(P_\times^\op,I_-^\op, I_-,P_-), \cat)$, in addition to sending squares of types 
   $(I_-^{\op,2},I_-^3)$, $(I_-^{\op,2},P_\times^{\op,1})$, $(P_-^4,P_\times^{\op,1})$ and $(P_-^4,I_-^3)$ to vertically left-adjoinable squares, $F$ sends squares of types  $(I_-^3,I_-^3)$, $(I_-^1,P_\times^{\op,1})$, $(P_-^ {\op,3},P_\times^{\op,1})$ and $(P_-^{\op,1},I_-^3)$ to vertically right-adjoinable squares. Moreover, morphisms in $\fun^\lax_{BC}(\delta_4^*\C^{\op,\sqcup,\op}(P_\times^\op,I_-^\op, I_-,P_-), \cat)$ are compatible with not only left adjoints, but also with these right adjoints, as in Definition \ref{defn:bc_span}.  
\end{lem}
\begin{proof}

Let $F$ be in $\fun^\lax_{BC}(\delta_4^*\C^{\op,\sqcup,\op}(P_\times^\op,I_-^\op, I_-,P_-), \cat)$.
By Lemma \ref{lem:meer_adjoinable_squares}, $PA(F)$ sends squares of the types $(I_-^3,I_-^3)$, $(I_-^3,P_\times^{\op,1})$, $(P_-^ {\op,3},P_\times^{\op,1})$ and $(P_-^{\op,1},I_-^3)$ to vertically right-adjoinable squares. But $F$ and $PA(F)$ agree on 1-simple edges and 3-simple edges, and therefore on squares of these types in $\delta_4^*\C^{\op,\sqcup,\op}(P_\times^\op,I_-^\op, I_-,P_-)$, $F$ agrees with $PA(F)$; so $F$ sends such squares to vertically right-adjoinable squares. 

For $\alpha$ a morphism in $\fun^\lax_{BC}(\delta_4^*\C^{\op,\sqcup,\op}(P_\times^\op,I_-^\op, I_-,P_-), \cat)$, we have that $PA(\alpha)$ is a morphism in $\fun^\lax_{BC}(\delta_4^*\C^{\op,\sqcup,\op}(P_\times^\op,I_-, I_-, P_-^\op),\cat)$ and therefore compatible with right adjoints. Since the components of $PA(\alpha)$ coincide with those of $\alpha$, it follows that $\alpha$ is compatible with right adjoints.
\end{proof}
\begin{cor}\label{cor:nog_meer_adjoinable_squares}
    For $D$ in $  \mathbf{3FF}$, $D$ sends squares of types $(I_-,I_-)$, $(I_-,P_\times^\op)$, $(P_-^ \op,P_\times^\op)$ and $(P_-^\op,I_-)$ to vertically right-adjoinable squares, and morphisms in $\mathbf{3FF}$ are compatible with left adjoints as well as right adjoints.
\end{cor}
\begin{proof}
Let $D$ be in $\mathbf{3FF}$. By Lemma \ref{lem:w6FF_step_3}, $\psi^*D$ is in $\fun^\lax_{BC}(\delta_4^*\C^{\op,\sqcup,\op}(P_\times^\op,I_-^\op, I_-,P_-), \cat)$, and therefore by Lemma \ref{lem:nog_meer_adjoinable_squares} sends squares of types $(I_-,I_-)$, $(I_-,P_\times^\op)$, $(P_-^ \op,P_\times^\op)$ and $(P_-^\op,I_-)$ to vertically right-adjoinable squares. By the same arguments as in the proof of Lemma \ref{lem:meer_adjoinable_squares}, this also applies to $D$. The same reasoning applies to morphisms between three-functor formalisms.  
\end{proof}

\begin{proof}[Proof of Proposition \ref{prop:3FF}]
Lemma \ref{lem:w6FF_step_3}, Lemma \ref{lem:w6FF_step_2.1} and Lemma  \ref{lem:w6FF_step_1} together give a zig-zag of equivalences
\begin{multline*}
    \mathbf{3FF}\xlongrightarrow[\sim]{\psi^*} \fun^\lax_{BC}(\delta_4^*\C^{\op,\sqcup,\op}(P_\times^\op, I_-^\op,I_-,P_-),\cat) \xlongrightarrow[\sim]{{PA}} \\ \fun^\lax_{BC}(\delta_4^*\C^{\op,\sqcup,\op}(P_\times^\op, I_-,I_-,P_-^\op),\cat) \xlongleftarrow[\sim]{\phi^*} \fun^\lax_{BC}(\corr(\C,P,I)^\otimes,\cat)   .
\end{multline*}
On the other hand, Corollary \ref{cor:nog_meer_adjoinable_squares} shows that restriction along $e$ indeed defines a functor 
$$e^*:\mathbf{3FF} \to \fun^\lax_{BC}(\corr(\C,P,I)^\otimes,\cat) .$$
Let us denote $$\Phi: = (\phi^*)^{-1}\circ PA \circ \psi^*: \mathbf{3FF} \to \fun^\lax_{BC}(\corr(\C,P,I)^\otimes,\cat) .$$ We want to show that $\Phi$ is naturally equivalent to $e^*$, for this it suffices to show that $PA\circ \psi^*$ is equivalent to $\phi^* \circ e^*$. Tho this end, we consider the following diagram of inclusions of simplicial sets.
\begin{center}
\begin{tikzcd}[column sep=tiny]
                                                                          & {\delta_4^*\C(P_\times^\op, I_-,I_-,P_-^\op)} \arrow[rr, "\phi"] &                                                                       & {\corr(\C,P,I)^\otimes} \arrow[rd, "e"] &                   \\
{\delta_2^*\C(P_\times^\op,I_-)} \arrow[ru, "\alpha"] \arrow[rrd, "\beta"'] &                                                                 &                                                                       &                                                   & \corr(\C)^\otimes \\
                                                                          &                                                                 & {\delta_4^*\C(P_\times^\op, I_-^\op, I_-, P_-)} \arrow[rru, "\psi"'] &                                                   &                  
\end{tikzcd}
\end{center}
Here $\alpha$ and $\beta$ are the natural inclusions of simplicial set that, on 1-simplices, make a square into a 4-dimensional hypercube by inserting identities in the second and fourth direction. We note that $\phi\circ \alpha = \phi_0\circ \phi_1\circ \phi_2 \circ \alpha = \phi_0$, since $\phi_1\circ \phi_2\circ \alpha = \id_{\delta^*\C(P_\times^\op, I_-)}$. Since both $\phi_0$ and $\phi$ are categorical equivalences, it follows that $\alpha$ is a categorical equivalence.

Therefore, to show that $PA\circ \psi^* \cong \phi^*\circ e^*$, it suffices to show that $\alpha^*\circ PA\circ \psi^* \cong \alpha^*\circ \phi^* \circ e^*$. Now, the commutative diagram above also shows that $\alpha^*\circ \psi^*\circ e^* = \beta^*\circ \psi^*$. Therefore we can reduce the problem to showing that $\beta^*\circ \psi^* \cong \alpha^*\circ PA \circ \psi^*$. 

 We claim that in fact, for $D$ in $\mathbf{3FF}$, there is an equality $\beta^*\circ \psi^*(D) = \alpha^*\circ PA \circ \psi^*(D)$. Indeed, we note that for $x$ an $n$-simplex in $\delta_4^*\C(P_\times^\op, I_-, I_-,P_-^\op)$, by definition 
 $$PA\circ \psi^*(D)(x) = \textup{ev}'(\psi^*(D),x) $$
 for $$\textup{ev}':\textup{Fun}_{\{2,4\},\textup{left}}(\delta_4^*\C^{\op, \sqcup,\op}(P_\times^\op, I_-^\op,I_-,P_-))\times \delta_4^*\C^{\op,\sqcup,\op}(P_\times^\op,I_-,I_-,P_-^\op) \to \cat $$ is the functor that results from applying Partial Adjoints (\cite[Proposition 2.2.4]{liu_zheng_six_ops}) to the standard evaluation functor
$$\textup{ev}:\textup{Fun}_{\{2,4\},\textup{left}}(\delta_4^*\C^{\op, \sqcup,\op}(P_\times^\op, I_-^\op,I_-,P_-))\times \delta_4^*\C^{\op,\sqcup,\op}(P_\times^\op,I_-^\op,I_-,P_-) \to \cat $$
 as in the proof of Lemma \ref{lem:functorial_PA}. Now, conclusion (1) in \cite[Proposition 2.2.4]{liu_zheng_six_ops} implies that $\textup{ev}'$, when restricted to $\textup{Fun}_{\{2,4\},\textup{left}}(\delta_4^*\C^{\op, \sqcup,\op}(P_\times^\op, I_-^\op,I_-,P_-))\times \delta_2^*\C^{\op,\sqcup,\op}(P_\times^\op,I_-)$ via $\alpha$, coincides with the  $\textup{ev}$ when restricted to $\textup{Fun}_{\{2,4\},\textup{left}}(\delta_4^*\C^{\op, \sqcup,\op}(P_\times^\op, I_-^\op,I_-,P_-))\times \delta_2^*\C^{\op,\sqcup,\op}(P_\times^\op,I_-)$ via $\beta$. This implies that $\alpha^*(PA(\psi^*(D)) =\beta^*(\psi^*(D))$, as desired.

\end{proof}

Now we specialize Proposition \ref{prop:3FF} to the category of Nagata six-functor formalisms. First, we introduce some more notation.
\begin{defn}
    For $X$ a simplicial set, $X'\subseteq X$ a sub-simplicial set, $\C$ an $\infty$-category and $\C'\subseteq \C$ a sub-$\infty$-category, let $\fun(X\mid X',\C\mid\C')$ denote the full sub-$\infty$-category of $\fun(X,\C)$ on functors $f:X\to \C$ that restrict to a functor $X' \to \C'$.
\end{defn}
The following definition build upon Definition \ref{defn:bc_span}.
\begin{defn}\label{defn:BC_span_compleet}
        We denote by 
    $$\fun^{\lax,\cl}_{BC}(\corr(\C,P,I)^\otimes\mid \corr(\C,P,I),\cat\mid\cat^{LL}) \subseteq \fun^\lax_{BC}(\corr(\C,P,I)^\otimes,\cat)$$
    the full subcategory spanned by lax cartesian structures $F:\corr(\C,P,I)^\otimes\to \cat$ such that in addition to (B) in definition \ref{defn:bc_span},
    \begin{enumerate}
        \item[(i)] for $p:X \to Y$ in $P$, the image $p^*:F(Y)\to F(X)$ is a left-left adjoint, and we denote the adjoints by $p^*\vdash p_\natural\vdash p^!$
        \item[(ii)] for $i:U\hookleftarrow X$ in $I$, the image $i^!:F(U)\to  F(X)$ is a left-left adjoint, and we denote the adjoints by $i_!\vdash i^\natural\vdash i_*$,
        \item[(iii)] for every $X$ in $\C$, $F(X)$ is a closed symmetric monoidal stable $\infty$-category.
    \end{enumerate}
\end{defn}

\begin{prop}\label{prop:w6FF}
The equivalence of $\infty$-categories in Proposition \ref{prop:3FF} restricts to an equivalence of $\infty$-categories 
$$e^*: \textup{\textbf{6FF}} \xrightarrow{\sim} \fun^{\lax,\cl}_{BC}{(\corr(\C,P,I)^\otimes \mid \corr(\C,P,I),\cat\mid\cat^{LL})}. $$
\end{prop}
\begin{proof}
For $D$ in  $  \mathbf{3FF}$, we need to show that $D$ is in  $  \textup{\textbf{6FF}}$ if and only if $F:=e^*D$ is in $$\fun^{\lax,\cl}_{BC}{(\corr(\C,P,I)^\otimes \mid \corr(\C,P,I),\cat\mid\cat^{LL})}.$$ 
We recall that by Proposition \ref{prop:cohom_prop_etale}, for $p$ in $P$, $p_*$ has a right adjoint isomorphic to $p_!$, and for $i$ in $I$, $i_!$ has a right adjoint isomorphic to $i^*$. 

First, assume that $F$ is in $\fun^{\lax,\cl}_{BC}{(\corr(\C,P,I)^\otimes \mid \corr(\C,P,I),\cat\mid\cat^{LL})}$. Let $f$ be in $\C$. By assumption we can factor $f$ as $f=p\circ i$ where $i$ in $I$ and $p$ in $P$. Since $F$ is in $$\fun^{\lax,\cl}_{BC}{(\corr(\C,P,I)^\otimes \mid \corr(\C,P,I),\cat\mid\cat^{LL})},$$ we have that $i_!$ and $p^*$ are left-left adjoints. By the observation above, it follows that $i^*$ and $p_!$ are left adjoints. This shows that $f_!=p_!\circ i_!$ and $f^*=i^* \circ p^*$ are both left adjoints. Moreover, for $X$ in $\C$, we know that $D(X)=F(X)$ is a closed symmetric monoidal $\infty$-category. This shows that $D$ is a pre-six-functor formalism.

On the other hand, suppose $D$ is in  $  \textup{\textbf{6FF}}$. In particular, for $i$ in $I$ and $p$ in $P$, we have that $p_!$ and $i^*$ are left adjoints. Again by Proposition \ref{prop:cohom_prop_etale}, it follows that $p^*$ and $i_!$ are left-left adjoints. Moreover, for $X$ in $\C$, we know that $F(X)=D(X)$ is a closed symmetric monoidal $\infty$-category. This shows that $F$ is in $\fun^{\lax,\cl}_{BC}{(\corr(\C,P,I)^\otimes \mid \corr(\C,P,I),\cat\mid\cat^{LL})}$. 
\end{proof}
\section{Noetherian Nagata setups and local six-functor formalisms}\label{sect:reducing_to_comp}
In this section we define \textit{local six-functor formalisms}, which we can think of as ``Nagata six-functor formalisms that take values in presentable stable $\infty$-categories, and have recollements''. For this, we need to specialize from a Nagata setup to a \textit{noetherian Nagata setup}. This is a Nagata setup $(\C, I, P)$ where $\C$ is a 1-category, and where the class $J\subseteq P$ of monomorphisms in $P$ plays a special role dual to that of $I$; one can think of morphisms in $I$ and $J$ as open immersions and closed immersions respectively. 
The upshot is that a noetherian Nagata setup has exactly the required properties to prove one the equivalence (\ref{eq:intro1}), we show this in Subsection \ref{subsect:span_and_comp}. Then we show that a local six-functor formalism is uniquely determined by its restriction to a subcategory of ``complete objects''. 

Subsections \ref{subsect:noeth_nagata_setup} and \ref{subsect:span_and_comp} can be mostly skipped by readers only interested in six-functor formalisms on the category of varieties; in that case, the relevant results in these subsections are already proven in \cite{kuij_descent}, apart from Lemma \ref{lem:spaniscorr}.

\subsection{Noetherian Nagata setups}\label{subsect:noeth_nagata_setup}
We recall that a 1-category is \textit{balanced} if every morphism that is both an epimorphism and a monomorphism, is an isomorphism. We also recall that a family of morphisms $\{a_i:A_i\to X \}_{i\in I}$ is called \textit{jointly epic} if for $f,g:X\to Y$, whenever $f\circ a_i=g\circ a_i$ for all $i\in I$, then $f=g$. Lastly, we recall that in a 1-category $\C$, an initial object $\emptyset$ is called a \textit{strict initial object} if every morphism $X \to \emptyset$ is an isomorphism.

In the following definition we use the convention that for $A\to X$ and $B\to X$ monomorphisms, the pullback $A\times_X B$ can be denoted $A\cap B$ when $X$ is clear from context.
\begin{defn}\label{defn:app_top_Nagata_setup}
Let $(\C,All,I,P)$ be a Nagata setup. We denote by $J$ the class of monomorphisms that are in $P$, and we denote arrows in $J$ by $\closedrightarrow$.  Then $(\C,I,P)$ is a \textit{topological Nagata setup} if
    \begin{enumerate}[label=(\roman*)] 
    \item $\C$ is a small and balanced 1-category,
    \item $\C$ has a terminal object $*$ and a strict initial object $\emptyset$,
    \item for any object $X$ in $\C$, the arrow $\emptyset \to X$ is in $I$ and in $P$ (and hence in $J$),
    \item all morphisms in $I$ are monomorphisms,
    \item for $X$ in $\C$, the 1-category which we denote by $(I/X)_{\cong}$, with as objects morphisms $i:U\hookrightarrow$ in $I$ and morphisms given by isomorphisms $f:U\xrightarrow{\cong} U'$ that make the diagram
 \begin{center}
     \begin{tikzcd}
         U \arrow[d, "f"']\arrow[r,hook,"i"] & X\\
         U'\arrow[ru, hook,"i'"']&
     \end{tikzcd}
 \end{center}
 commute, is equivalent to a 0-category,
    \item \textit{closed complements:}   for $i:U \hookrightarrow X$ in $I$ there is unique maximal morphism $j:X \setminus U \closedrightarrow X$ in $J$ such that 
    \begin{center}
        \begin{tikzcd}
            \emptyset \arrow[d] \arrow[r] & U \arrow[d, hook, "i"]\\
            X \setminus U \arrow[r, two heads, tail, "j"] & X
        \end{tikzcd}
    \end{center}
    is a pullback square, in the sense that for every $k:C \closedrightarrow X$ for which 
    \begin{equation*}
          \begin{tikzcd}
            \emptyset \arrow[d] \arrow[r] & U \arrow[d, hook, "i"]\\
            C \arrow[r, two heads, tail, "k"] & X
        \end{tikzcd}
    \end{equation*}
    is a pullback, there is a unique morphism $C\closedrightarrow X\setminus U$ in $J$ such that 
     \begin{center}
        \begin{tikzcd}
            C \arrow[dr, two heads, tail, "k"'] \arrow[r, two heads,tail] & X \setminus U \arrow[d, two heads,tail, "j"]\\
             & X
        \end{tikzcd}
    \end{center}
    commutes,
    \item \textit{open complements}: vice versa, for $j:C\closedrightarrow X$ in $J$ there is a unique maximal morphism $i:X \setminus C \hookrightarrow X$ in $I$ such that \begin{center}
        \begin{tikzcd}
            \emptyset \arrow[d] \arrow[r] & X \setminus C \arrow[d, hook, "i"]\\
            C \arrow[r, two heads,tail, "j"] & X
        \end{tikzcd}
    \end{center} is a pullback square,
    \item  \textit{open and closed complements are stable under pullback}: given $X\setminus U \closedrightarrow X \hookleftarrow U$ a pair of open and closed complements and $f:Y\to X$ any morphism, the top row in 
    \begin{center}
        \begin{tikzcd}
            (X\setminus U) \times_X Y \arrow[r, two heads,tail] \arrow[d] & Y\arrow[d] & U \times_X Y \arrow[d] \arrow[l, hook']  \\
              X\setminus U \arrow[r, two heads,tail] & X & U \arrow[l, hook']
        \end{tikzcd}
        \end{center}
            is also a pair of open and closed complements,
    \item if $i:U\hookrightarrow X$ in $I$ is also in $P$ (and hence in $J$), then the open complement is the same as the open complement, so $X\setminus U$ is well-defined and  $X\setminus U\to X$ is in both $I$ and $J$,
    \item the morphisms $U\hookrightarrow X$ and $X\setminus U \closedrightarrow X$ are jointly epic,
    \item \textit{closures}: 
    for $e:A\to X$ a monomorphism, there is a unique minimal $j:\overline{A}\closedrightarrow X$ in $J$ such that there exists a monomorphism $f:A\to \overline{A}$ and $j\circ f=e$, in the sense that for $j':\overline{A}'\closedrightarrow X$ in $J$ and $f':A\to \overline{A}'$ such that $j'\circ f'=e$, there is a unique $g:\overline{A}\closedrightarrow \overline{A}'$ in $J$ that makes
    \begin{center}
        \begin{tikzcd}
            A \arrow[r,  "f"] \arrow[dr, "f' " ']& \overline{A}\arrow[d, two heads,tail, "g"] \arrow[r, two heads,tail, "j"] & X\\
            & \overline{A}' \arrow[ur, two heads,tail, "j'"'] &
        \end{tikzcd}
    \end{center}
    commute,  
    \item  \textit{``subspace topology''}: for $A\to X$ a monomorphism and $U\hookrightarrow A$ in $I$, there is $V\hookrightarrow X$ in $I$ such that $V\cap A = U$. For $C \closedrightarrow A$ in $J$, there is $D\closedrightarrow X$ in $J$ such that $D \cap A=C$,
    \item \textit{gluing along opens}: for $U\hookrightarrow X$ and $U\hookrightarrow Y$, there is a pushout, which we denote
    \begin{center}
        \begin{tikzcd}
            U \arrow[r, hook] \arrow[d, hook] & X \arrow[d, hook] \\
            Y \arrow[r, hook] & X\cup_U Y,
        \end{tikzcd}
    \end{center}
   this square is also a pullback square and the morphisms $X\hookrightarrow X\cup_U Y$ and $Y \hookrightarrow X\cup_U Y$ are in $I$,
    \item \textit{closed unions}: for $C \closedrightarrow X$ and $D \closedrightarrow X$ in $J$, the pushout over $C\cap D$ which we denote
    \begin{center}
        \begin{tikzcd}
            C\cap D \arrow[r, two heads, tail] \arrow[d, two heads, tail] & C \arrow[d, two heads, tail] \\
            D \arrow[r, two heads, tail] & C\cup D
        \end{tikzcd}
    \end{center}
    exists, and the map $C\cup D \to X$ is in $J$,
    \item \textit{complement of a union}: for $C \closedrightarrow X$ and $D \closedrightarrow$ in $J$, the natural map $X\setminus (C\cup B)\to X\setminus C \cap X \setminus D$ is an isomorphism,
    \item \textit{separated}: for $X$ in $\C$, the diagonal $X \to X \times X$ is in $J$. 
    \end{enumerate}
\end{defn}
\begin{example}\label{ex:varieties}
    The motivating example is the Nagata setup $(\var_k,I,P)$ where $\var_k$ is the category of varieties over a field $k$, $I$ are the open immersions and $P$ the proper morphisms. In that case, $J$ is the class of closed immersions. 
\end{example}

The following definition generalizes the notion of denseness for topological spaces.
\begin{defn}
     Let $(\C,I,P)$ be a topological Nagata setup. We say that a monomorphism $A\to X$ in $\C$ is \textit{dense} if there is no $V\hookrightarrow X$ in $I$ with $V\neq\emptyset$, such that $U \cap V=\emptyset$. 
\end{defn}
In the example of varieties, this coincides with the familiar definition of denseness. We record a number of standard results involving denseness, that follow from the axioms in Definition \ref{defn:app_top_Nagata_setup}.

\begin{lem}\label{lemm:app_dense}
   Let $(\C,I,P)$ be a topological Nagata setup. Then for any monomorphism $e:A\to X$ in $\C$, the morphism $ f:A \to \overline{A}$ to the closure of $A$ in $X$ is dense.   
\end{lem}
\begin{proof}
    Suppose there is a morphism $V\hookrightarrow \overline{A}$ in $I$ such that $A \cap V = \emptyset$. By (viii) in Definition \ref{defn:app_top_Nagata_setup}, pulling back the open-closed decomposition along $A\to \overline{A}$ gives a diagram
    \begin{center}
        \begin{tikzcd}
            \emptyset \arrow[r] \arrow[d]& A \arrow[d] & A \arrow[d] \arrow[l] \\
            V \arrow[r, hook] & \overline{A} & \overline{A}\setminus V. \arrow[l, two heads,tail]
            \end{tikzcd}
    \end{center}
    Since $A\to \overline{A}\setminus V \twoheadrightarrow \overline{A} \twoheadrightarrow X$ is $e$, and $\overline{A}$ is minimal, there must be a morphism $g:\overline{A}\closedrightarrow \overline{A}\setminus V$ in $J$ that makes
    \begin{center}
        \begin{tikzcd}
        & \overline{A} \arrow[d, two heads,tail, "f"] \arrow[drr] & & \\
            A \arrow[r] \arrow[ur] & \overline{A}\setminus V \arrow[r, two heads, tail] & \overline{A}\arrow[r, two heads,tail] & X
        \end{tikzcd}
    \end{center}
    commute. This implies that $\overline{A}\setminus V \closedrightarrow \overline{A}$ is an epimorphism in addition to being a monomorphism. So since $\C$ is balanced, we have $\overline{A}\setminus V = \overline{A}$ and hence $V=\emptyset$.
\end{proof}
In particular, when $e:A\to X$ is a monomorphism such that $\overline A = X$, then $e$ is dense. The reverse is true too, by the following lemma.
\begin{lem}\label{lem:app_dense_reverse}
   Let $(\C,I,P)$ be a topological Nagata setup. If $e:A\to X$ is a dense monomorphism, then $\overline A \closedrightarrow X$ is an isomorphism.
\end{lem}
\begin{proof}
We can form the commutative diagram
\begin{center}
    \begin{tikzcd}
        \emptyset \arrow[d] \arrow[r] & \emptyset \arrow[d]\arrow[r]& X\setminus \overline{A}\arrow[d, hook]\\
        A\arrow[r] & \overline{A} \arrow[r, two heads,tail] &X
    \end{tikzcd}
\end{center}
where both the small squares, and therefore the outer square, are pullbacks. This contradicts $e:A \to X$ being dense, unless $\overline{A}\closedrightarrow X$ is an isomorphism (and therefore $X\setminus \overline{A} = \emptyset$).
\end{proof}
\begin{lem}\label{lem:app_open_in_closure}
    For $i:U\hookrightarrow X$ in $I$, the morphism $f:U\to \overline{U}$ is in $I$.
\end{lem}
\begin{proof}
    Since $U \hookrightarrow X$ factors through $\overline{U}\closedrightarrow X$, we have a pullback square
    \begin{center}
        \begin{tikzcd}
            U \arrow[r] \arrow[d] & U \arrow[d,hook] \\
            \overline{U}\arrow[r, two heads,tail]& X.
        \end{tikzcd}
    \end{center}
 Since $I$ is closed under pullbacks, this implies that $U \to \overline{U}$ is in $I$.
\end{proof}
\begin{lem}
     Let $(\C,I,P)$ be a topological Nagata setup. Then any morphism $f:X\to Y$ in $\C$ can be written as $p\circ i$ where $p$ is in $P$, and $i$ is in $I$ and dense.
\end{lem}
\begin{proof}
 With (1) in Definition \ref{defn:Nagata_setup} we can write $f$ as
$$X \hookrightarrow Z \rrightarrow Y$$
with $i$ in $I$ and $p$ in $P$, and by Lemma \ref{lem:app_open_in_closure} we can rewrite this as
$$X \hookrightarrow \overline{X}\closedrightarrow Z\rrightarrow Y.  $$   
\end{proof}
\begin{rmk}
    In any Nagata setup with a terminal object $*$, for $X$ any object, the morphism $X\to *$ factors as an $I$-morphism followed by $P$ morphism
    $$X\hookrightarrow Y \to *.$$
    We cal $Y$ a \textit{compactification} of $X$, and by definition, $Y$ is in $\comp$. By the lemma above, there exists a compactification $Y$ such that $X$ is dense in $Y$. 
\end{rmk}
\begin{lem}[{\cite[Lemma 6.15]{kuij_descent}}] \label{lem:app_pullback_lem}
 Let $(\C,I,P)$ be a topological Nagata setup, and let
 \begin{center}
     \begin{tikzcd}
         V \arrow[d, hook, "i"']\arrow[r, two heads]&U\arrow[d, hook]\\
         Y\arrow[r, two heads]&X
     \end{tikzcd}
 \end{center}
 be a commutative square with the vertical morphisms in $I$ and the horizontal morphisms in $P$, and with $i$ dense. Then it is a pullback square.
\end{lem}
\begin{proof}
    The proof is essentially the same as that of \cite[Lemma 6.15]{kuij_descent}, but we repeat it for the reader's convenience. The canonical morphism $V \to Y\times_X U$ is both in $I$ and in $P$, by (4) and (5) in Definition \ref{defn:Nagata_setup}. Therefore it is also in $J$, and by (ix) in Definition \ref{defn:app_top_Nagata_setup} the complement $(Y\times_X U )\setminus V \to Y\times_X U$ is both in $I$ and in $J$. We now consider the following diagram 
    \begin{center}
     \begin{tikzcd}
         \emptyset \arrow[d] \arrow[r] & (Y\times_X U ) \setminus V \arrow[r, equal] \arrow[d, hook] & (Y \times_X U) \setminus V \arrow[d, hook] \\
         V \arrow[r, hook] \arrow[d, equal] & Y\times_X U \arrow[d, equal] \arrow[r, equal] & Y\times_X U \arrow[d, hook] \\
         V \arrow[r, hook]  \arrow[rr, hook, bend right=15, shift right, "i"']& Y \times_X U \arrow[r, hook] & Y
     \end{tikzcd}
    \end{center}
    where all squares, and therefore the outer square, are pullbacks. This contradicts $i$ being dense, unless $(Y\times_X U)\setminus V = \emptyset$ and therefore $V \to Y\times_X U$ is an isomorphism.
\end{proof}

Another notion that generalizes easily to this setting, is (ir)reducibility.
\begin{defn}
  Let $(\C,I,P)$ be a topological Nagata setup. An object in $\C$ is called \textit{reducible} if there are  $j:C \closedrightarrow X$ and $k:D \closedrightarrow X$ in $J$ such that $C\cup D =X$, and neither $j$ nor $k$ is an isomorphism. An object that is not reducible, is called \textit{irreducible}. 
\end{defn} 
Note that, in the light of (xv) in Definition \ref{defn:app_top_Nagata_setup}, the condition $C\cup D= X$ is equivalent to the condition that $X\setminus C \cap X \setminus D = \emptyset$. 

\begin{defn}
    Let $(\C,I,P)$ be a topological Nagata setup and $X$ an object of $\C$. An \textit{irreducible component} of $X$ is an object $Z$ with a morphism $Z \closedrightarrow X$ in $J$, such that $Z$ is irreducible, and there is no $Y\closedrightarrow Z$ in $J$ with $Y$ irreducible and a non-identity map $Z\to Y$ that makes
    \begin{center} 
        \begin{tikzcd}
            Z \arrow[rd, two heads, tail] \arrow[d] \\
            Y \arrow[r, two heads, tail ]  & X
        \end{tikzcd}
    \end{center}
    commute.
\end{defn}
In the example of $(\var_k, I,P)$, the definitions above recover the familiar notions of (ir)reducibility and irreducible components.

\begin{lem}\label{lem:app_standard}
  Let $(\C,I,P)$ be a topological Nagata setup. Let $i:U\hookrightarrow X$ be in $I$ such that $X$ is irreducible, and $U\neq \emptyset$. Then $i$ is dense, and $U$ is irreducible. 
\end{lem}
\begin{proof}
 We have $J$-morphisms $X \setminus U\closedrightarrow X$ and $\overline{U} \closedrightarrow X$. As in the proof of Lemma \ref{lem:app_dense_reverse}, we have that $U \cap X \setminus \overline{U} = \emptyset$. Since $U = X \setminus (X \setminus U)$, this shows that $(X\setminus U) \cup \overline{U} = X$. Since $X$ is irreducible, it must be that $\overline{U} = X$, so $i$ is dense.

Now suppose that $U$ is reducible. Then there are non-isomorphisms $C\closedrightarrow U$ and $D \closedrightarrow U$ with $U\setminus C \cap U \setminus D = \emptyset$. We can form the diagram
\begin{center}
\begin{tikzcd}
    \emptyset \arrow[r] \arrow[d] & U \setminus D \arrow[r, equal] \arrow[d, hook] & U \setminus D \arrow[d, hook] \\
    U\setminus C \arrow[d, equal] \arrow[r, hook] & U \arrow[r, equal] \arrow[d, equal] &U \arrow[d, hook]\\
    U\setminus C \arrow[r, hook] & U \arrow[r, hook] & X
\end{tikzcd}
\end{center}
where all the small squares are pullbacks. Therefore the outer square is a pullback, contradicting that $X$ is irreducible. This shows that $U$ is irreducible.
\end{proof}

Lastly, we can generalize the usual (Krull) dimension function. 
\begin{defn}\label{defn:appendix_dimension}
     Let $(\C,I,P)$ be a topological Nagata setup. For $X$ in $\C$, we set $\dim(X)$ to be the largest integer $d$ such that there exists a chain 
     $$ C_0 \closedrightarrow C_1 \closedrightarrow \dots \closedrightarrow C_d \closedrightarrow X $$
    of morphisms in $J$, with $C_0,\dots, C_d$ non-empty, distinct and irreducible (possibly $ \dim(X)=\infty$). We set $\dim(\emptyset)=-1$.
\end{defn}
From this definition it follows immediately that for $C\closedrightarrow X$ in $J$, we have $\dim(C)\leq \dim(X)$. However, to make the dimension function generally behave as we want in a topological Nagata setup, we need to assume some extra conditions. We introduce the following terminology. 

\begin{defn}\label{defn:app_noeth}
    A \textit{noetherian Nagata setup} is a topological Nagata setup $(\C,I,P)$, such that
    \begin{enumerate}[label=(\roman*)] 
    \item for every $X$ in $\C$, $\dim(X)$ is finite,
    \item every $X$ in $\C$ can be written as a finite union $X_1\cup \dots \cup X_n$ of irreducible components $X_i \closedrightarrow X$, 
    \item for $U\hookrightarrow X$ in $I$, we have $\dim(U)\leq \dim(X)$,
    \item if $U\hookrightarrow X$ in $I$ is dense, then $\dim(U)=\dim(X)$ and $\dim(X\setminus U)<\dim(X)$.
    \end{enumerate}
\end{defn}
\begin{example}\label{ex:app_noeth_setup}
  The topological Nagata setup of varieties in Example \ref{ex:varieties} is noetherian.  Another example of a noetherian Nagata setup is the category of separated, finite type, reduced schemes over a noetherian base scheme $S$, with $I$ the open immersions, and $P$ the proper morphisms. It is necessary for the morphisms to be separated and of finite type, to get a Nagata setup. The objects need to be separated to satisfy (xvi) in Definition \ref{defn:app_top_Nagata_setup}, and reduced to satisfy (vi).  \end{example}

\subsection{The categories $\span(\C)$ and $\comp(\C)$}\label{subsect:span_and_comp}

In this subsection we study the categories $\span(\C)$ and $\comp(C)$ that generalize the categories $\span$ and $\comp$ from the Nagata setup of varieties to arbitrary noetherian Nagata setups.

In any Nagata setup, we can define the following category. 
 
 \begin{defn}
   Let $(\C,All,I,P)$ be a Nagata setup. We denote by $\span(\C)$ the $\infty$-category $\corr(\C,I,P)$. In other words, $\span(\C)$ is the subcategory of $\corr(\C)$ spanned by edges $X \hookleftarrow U \twoheadrightarrow X$ where $U \hookrightarrow X$ is in $I$, and $U \twoheadrightarrow X$ is in $P$.
\end{defn} 
Note that $\span(\C)$ is the \textit{opposite} category to the $\infty$-category $\corr(\C,P,I)$ which we considered in Subsection \ref{subsect:proper_open}. The reason is that we will define a Grothendieck topology on $\span(\C)$, and consider functors 
$$\corr(\C,P,I) = \span(\C)^\op \to \cat $$
that satisfy descent with respect to this topology. 

    In general, $\corr(\C,I_1,I_2)$ need not be a 1-category, even if $\C$ is. However, for a topological Nagata setup we can prove the following.
    \begin{prop}\label{prop:1-cats}
        Let $(\C,I,P)$ be a topological Nagata setup. Then the $\infty$-categories $\corr(P,I)$ and $\corr(P, I)^\otimes$ are 1-categories.
    \end{prop}
    \begin{proof}
    This follows from (v) in Definition \ref{defn:app_top_Nagata_setup}. First, we show that $\corr(P,I)$ is a 1-category. For every $X$ in $\C$, we can choose a skeleton $\textup{Open}(X)$ of $(I/X)_{\cong}$, which is just a set. Then every edge in $\corr(P,I)$ is equivalent to an edge
    $$X \twoheadleftarrow U \hookrightarrow Y$$
    where $U\hookrightarrow Y$ is in $\textup{Open}(X)$. In the sub-simplicial set spanned by these edges, composition is now well defined, since for two composable spans 
$$X \xtwoheadleftarrow{p}U \xhookrightarrow{i} Y \xtwoheadleftarrow{q} V \xhookrightarrow{j} Z $$
there is exactly one $q^{-1}(U)\hookrightarrow Z$ in $\textup{Open}(Z)$, and the composition becomes
$X \twoheadleftarrow q^{-1}(U) \hookrightarrow Z.$ 

Now for $\corr(\C,P,I)^\otimes$, we note that an edge 
$$(X_i)_I \leftarrow (U_j)_J \rightarrow (Y_j)_J  $$
is isomorphic to an edge
$$(X_i)_I \leftarrow (U_j')_J \rightarrow (Y_j)_J  $$
where each $U_j'\to Y_j$ is in $\textup{Open}(Y_j)$. Then sub-simplicial set spanned by edges of this again form is a 1-category that is equivalent to $\corr(\C,P,I)^\otimes$.
    \end{proof}
Of course a direct consequence of the above is that $\span(\C)$ is a $1$-category. 
  
We note that $\span(\C)$ does not have products; however, $(\span(\C),\times)$ with $\times$ the cartesian products in $\C$, is still a symmetric monoidal 1-category. 

\begin{defn} Let $(\C,All,I,P)$ be a Nagata setup. As in Notation \ref{nota:C_times}, we denote by 
    $\span(\C)_\times^\op\to \fin_\part$
     the cocartisian fibration that classifies the symmetric monoidal 1-category $(\span(\C)^\op,\times)$. Explicitly, $\span(\C)_\times$ is a 1-category with as objects tuples $(X_i)_I$ indexed by a finite set, with $X_i$ in $\C$, and a morphism
    $$ (X_i)_I \to (Y_j)_J $$
    is given by a partial map $\alpha:J \dashrightarrow I$, and for all $i\in I$ a map 
    $$ X_i \hookleftarrow U_i \twoheadrightarrow  \prod_{j\in \alpha^{-1}(i)} Y_j.$$ 
\end{defn}
\begin{lem}\label{lem:spaniscorr}
For $(C,I,P)$ a topological Nagata setup, there is an equivalence of symmetric monoidal $\infty$-categories $\span(\C)_\times^\op\simeq \corr(\C,P,I)^\otimes$. 
\end{lem}
\begin{proof}
By the definition above, and by Proposition \ref{prop:1-cats}, both $\span(\C)_\times^\op$ and $\corr(\C,P,I)^\otimes$ are 1-categories, so we just need to compare their objects and morphisms.

 For both $\span(\C)_\times^\op$ and $\corr(\C,P,I)^\otimes$, the objects are tuples $(X_i)_I$ of objects in $\C$, indexed by a finite set. 
 We give a functor 
 $$\span(\C)_\times^\op\to \corr(\C,P,I)^\otimes $$
 which is the identity on objects, as follows. A morphism $(X_i)_I\to (Y_j)_J$ in $\span(\C)_\times^\op$ is given by a partial map $\alpha:I\dashrightarrow J$ and for $j\in J$ a map
 $$ Y_j \hookleftarrow U_j \twoheadrightarrow \prod_{\alpha^{-1}(j)} X_i$$
 in $\span(\C)$.  This is equivalent to giving the tuple $(U_j)_J$, 
 a map $(U_j)_J\to (Y_j)_J$ in $I_-$ and a map $(U_j)_J\to(X_i)_I$ in $P_\times$. This gives a map in $\corr(\C,P,I)^\otimes$, and it is clear that this assignment is a bijection on the morphisms in both categories. 
\end{proof}
The following notation is inspired by the fact that, for the Nagata setup of varieties from Example \ref{ex:varieties}, this defines the category of complete varieties.
\begin{defn}
    For $(\C, All, I,P)$ a Nagata setup, we define $\comp(\C)$ to be the full subcategory of $\C$ spanned by objects $X$ such that the unique morphism $X\to *$ to the terminal object is in $P$.
\end{defn}
If $\C$ is a 1-category, then it is clear that $\comp(\C)$ is a 1-category too. By (v) in Definition \ref{defn:Nagata_setup}, all morphisms in $\comp(\C)$ are in $P$. Therefore $\comp(\C)$ embeds into $\span(\C)$.

As in \cite[Section 6]{kuij_descent}, we can define the following sets of squares. 
\begin{defn}\label{def:abs_blowup_square}
     Let $(\C,I,P)$ be a topological Nagata setup. An \textit{abstract blowup square} in $\C$ is a pullback square of $P$-morphisms
     \begin{equation}\label{eq:app_abs_blowup}
      \begin{tikzcd}
             E \arrow[d, two heads] \arrow[r, two heads,tail]&Y\arrow[d, two heads, "p"]\\
             C \arrow[r, two heads,tail, "j"'] & X
         \end{tikzcd}     
     \end{equation}       
     with $j$ in $J$, such that the induced morphism on the open complements $Y\setminus E \to X \setminus C$ is an isomorphism.

     A \textit{localization square} in $\span(\C)$ is a commutative square of the form 
     \begin{equation}\label{eq:app_localization}
         \begin{tikzcd}
             X \setminus U \arrow[r, two heads,tail,"j" ] & X \\
             \emptyset \arrow[u, hook, "i"'] \arrow[r, two heads,tail] & U \arrow[u, hook].
         \end{tikzcd}
     \end{equation}
     with $i$ in $I$ and $j$ in $J$.
\end{defn}
This allows us to define a cd-structure $AC$ on $\comp(\C)$, given by abstract blowup squares in $\comp(\C)$. The abstract blowup squares also define a cd-structure on $\comp(\C)_0$, as defined in \cite[Definition 5.1]{kuij_descent}. Lastly, we consider the set of abstract blowup squares and localization squares $A\cup L$ as a cd-structure on $\span(\C)$. 

    \begin{rmk}
  In the upcoming lemmas and in the proof of Proposition \ref{prop:excision_is_descent}, we heavily use the theory of cd-structures as developed by Voevodsky in \cite{voevodsky}, and an adaptation of this theory to categories without a strict initial object, as in \cite[Sections 2 and 3]{kuij_descent}. We recall the following for a cd structure $P$ on a category $\D$ with an initial object $\emptyset$.
    \begin{itemize}
    \item[(a)] We denote by $\tau_P$ denotes the topology generated by $P$ (\cite[Section 2]{voevodsky}), and by $\tau^c_P$ the coarse topology generated by $P$ (\cite[Definition 2.3]{kuij_descent}).
    \item[(b)] If $P$ is \textit{c-complete} (respectively \textit{complete}) and $F$ is a presheaf (of sets) on $\D$ that sends distinguished squares to pullback squares (respectively, $F$ sends distinguished squares to pullback squares and $F(\emptyset)\cong *$) then $F$ is a $\tau^c_P$-sheaf (respectively, a $\tau_P$-sheaf). 
     \item[(c)] If $P$ is \textit{c-regular} (respectively regular), and $F$ is a $\tau^c_P$-sheaf (respectively, a $\tau_P$-sheaf), then $F$ sends distinguished squares to pullback squares (respectively, $F$ sends distinguished squares to pullback squares and $F(\emptyset)\cong *$).
     \item[(d)] If $P$ is \textit{compatible with a dimension function} in addition to being c-complete and c-regular (respectively, complete and regular), then a simplicial presheaf on $\D$ is a $\tau^c_P$-hypersheaf (respectively, a $\tau_P$-hypersheaf) if and only if $F$ sends distinguished squares to homotopy pullback squares (respectively, $F$ sends distinguished squares to homotopy pullback squares and $F(\emptyset)$ is contractible).
    \end{itemize}
\end{rmk}

\begin{lem}\label{lem:app_thm_1}
    Let $(\C,I,P)$ be a noetherian Nagata setup. Then the cd-structures $AC$ on $\comp(\C)$, $AC$ on $\comp(\C)_0$ and $A\cup L$ on $\span(\C)$ are compatible with the dimension function as defined in Definition \ref{defn:appendix_dimension}.
\end{lem}
\begin{proof}
 As in the proof of \cite[Lemma 6.10]{kuij_descent}, we define $A'$ to be the set of abstract blowup squares in $\C$ where $\dim(Y)\leq \dim(X)$ and $\dim(E)<\dim(X)$. As in the cited proof, we see that for $X$ in $\C$, the irreducible components form a simple $A'$-cover $\{X_i\closedrightarrow X\}$. Now we need to show that for any abstract blowup square (\ref{eq:app_abs_blowup}) in $\C$, the sieve $\langle j,p \rangle$ contains a simple $A'$-cover, where we can assume $X$ is irreducible. Let $\overline{Y \setminus E} \closedrightarrow Y$ be the denote the closure of $Y \setminus E \hookrightarrow Y$. Then 
\begin{equation}\label{eq:dimension_abs_blowup}
    \begin{tikzcd}
        \overline{Y \setminus E} \setminus (Y \setminus E) \arrow[d, two heads] \arrow[r, two heads, tail] & \overline{Y \setminus E} \arrow[d, two heads]\\
        C \arrow[r, two heads, tail] & X
    \end{tikzcd}
\end{equation}
is a pullback square by Lemma \ref{lem:app_pullback_lem}, and by construction an abstract blowup. Since $Y\setminus E \hookrightarrow \overline{Y\setminus E}$ is dense, we have  $\dim( \overline{Y\setminus E}) = \dim({Y\setminus E})$ and $\dim (\overline{Y \setminus E}\setminus (Y \setminus E))<\dim(Y \setminus E)$. This, combined with the fact that $\dim(Y \setminus E) = \dim(X \setminus C) \leq \dim(X)$, shows that (\ref{eq:dimension_abs_blowup}) is in $A'$, and the arrows $C \closedrightarrow X$ and $\overline{Y \setminus E}\twoheadrightarrow X$ are in $\langle j,p \rangle$. This shows that the cd-structure $A$ on $\C$ is compatible with the dimension function, and with the same proof we see that $AC$ is compatible with the dimension function as a cd-structure on $\comp(\C)$ or $\comp(\C)_0$. 

For $A\cup L$, as in the proof of \cite[Lemma 6.14]{kuij_descent} $L'\subseteq L$ denote the set of localization squares where $U\hookrightarrow X$ is dense. Then $\dim(X)=\dim(U)$ and $\dim(X\setminus U)<\dim(X)$. As in the cited proof, for any localization square (\ref{eq:app_localization}) we can show that the arrow $\overline{U}\hookleftarrow U$ is in the sieve $\langle \emptyset \to U, X \hookleftarrow U \rangle$. This, together whit the fact that $A$ is compatible with the dimension function, shows that $A\cup L$ is compatible with the dimension function.   
\end{proof}

\begin{lem}\label{lem:app_thm_2}
    Let $(\C,I,P)$ be a topological Nagata setup. Then the cd-structure $AC$ on $\comp(\C)$ is complete, and the cd-structures $AC$ on $\comp(\C)_0$ and $A\cup L$ on $\span(\C)$ are c-complete.
\end{lem}
\begin{proof}
 We follow the relevant parts of the proof of \cite[Lemma 6.16]{kuij_descent}. To show that $AC$ is complete as cd-structure on $\comp(\C)$, it suffices to observe that also in this general case, abstract blowup squares are stable under pullback. Pullbacks in $\comp(\C)$ are also pullbacks in $\comp(\C)_0$. With the additional observation that the pullback of a sieve along a zero morphism is the maximal sieve, and therefore a covering sieve, it follows that $AC$ is c-complete as cd-structure on $\comp(\C)_0$.

Now we show that $A$ is a c-complete cd-structure on $\span(\C)$. The non-trivial step here is to show that for $f:Z \hookleftarrow X$ an $I$-morphism in $\span(\C)$, and an abstract blowup square (\ref{eq:app_abs_blowup}), the sieve $f^*\langle j, p \rangle $ contains a simple $A$-cover. 

As in the cited proof, we factor $Y \xtwoheadrightarrow{p} X \hookrightarrow Z$ as a dense $I$-morphism followed by a $P$-morphism $Y \hookrightarrow Y' \xtwoheadrightarrow{p'} Z$. Let $\overline{C}\closedrightarrow Z$ be the closure of $C \closedrightarrow X \hookrightarrow Z$, then we can form the union $C':=\overline{C} \cup (Z \setminus X) \closedrightarrow Z$. Using (xv) in Definition \ref{defn:app_top_Nagata_setup}, we see that $Z \setminus C' = (Z \setminus \overline{C})\cap X$. We form the following diagram
\begin{center}
    \begin{tikzcd}
         X \cap \overline{C}  \arrow[r, tail, two heads] \arrow[d, hook]           & X \arrow[d, hook] &    Z \setminus C'    \arrow[d, hook] \arrow[l, hook']      \\
        \overline{C} \arrow[r, tail, two heads] & Z & Z \setminus \overline{C} \arrow[l, hook']
    \end{tikzcd}
\end{center}
by pulling back the open-closed decomposition $\overline{C} \closedrightarrow Z \hookleftarrow Z \setminus \overline{C}$ along $X \hookrightarrow Z$. This shows that $Z \setminus C' = X \setminus (X \cap \overline{C})$. Now, by (xii) in Definition \ref{defn:app_top_Nagata_setup}, there is a $J$-morphism $D\closedrightarrow Z$ such that $C = D\cap X$. Therefore there is a $J$-morphism $\overline{C} \closedrightarrow D$, and this implies that $X \cap \overline{C} \subseteq X \cap D = C$. Hence, we conclude $Z \setminus C' = X \setminus C$. 

Now, we have a square 
\begin{equation}\label{eq:completeness}
    \begin{tikzcd}
       Y \arrow[r,hook ]\arrow[d, two heads, "p"'] & Y' \arrow[d, two heads, " p'" ] \\
       X \arrow[r, hook] & Z
    \end{tikzcd}
\end{equation}
which is a pullback square by Lemma \ref{lem:app_pullback_lem}, and this implies that the inverse image of $Z \setminus C'= X\setminus C$ under $p'$ is contained in $Y$ and is the same as $p^{-1}(X \setminus C) = Y \setminus E$. We conclude that the pullback square
\begin{center}
    \begin{tikzcd}
        C' \times_Z Y' \arrow[d, two heads] \arrow[r, tail, two heads] &  Y' \arrow[d, two heads, "p'"] \\
        C' \arrow[r, tail, two heads] & Z
    \end{tikzcd}
\end{center}
is an abstract blowup square, since $(p')^{-1}(Z \setminus C') = Y \setminus E$ is isomorphic to $Z \setminus C'=X \setminus C$. Now show that the simple $A$-cover generated by it, is contained in the $f^*\langle j, p \rangle $. The morphism $p'$ is in $f^*\langle j, p \rangle $ since (\ref{eq:completeness}) is a pullback square. The diagram of pullback squares 
\begin{center}
    \begin{tikzcd}
        C'\cap X \arrow[d, hook] \arrow[r, tail, two heads] & X \arrow[d, hook] & X \setminus C\arrow[d, equal] \arrow[l, hook'] \\
        C' \arrow[r, tail, two heads] & Z & Z \setminus C' = X \setminus C \arrow[l, hook']
    \end{tikzcd}
\end{center}
demonstrates that $C' \cap X = C$, since the top row is an open-closed decomposition by (viii) in Definition \ref{defn:app_top_Nagata_setup}. Therefore $C'\closedrightarrow Z$ is in  $f^*\langle j, p \rangle $ as well.

Lastly, to show that $L$ is a c-complete cd-structure on $\span(\C)$, we observe that the proof of Lemma \cite[Lemma 6.19]{kuij_descent} still works, where we use (xiii) in Definition \ref{defn:app_top_Nagata_setup} to construct $Y \cup_V Z$.
\end{proof}

\begin{lem}\label{lem:app_thm_3}
    Let $(\C,I,P)$ be a topological Nagata setup. Then the cd-structure $AC$ on $\comp(\C)$ is regular, and the cd-structures $AC$ on $\comp(\C)_0$ and $A\cup L$ on $\span(\C)$ are c-regular.
\end{lem}
\begin{proof}
  We start by showing that $A$ is a regular cd-structure on $\C$. It suffices to show that for abstract blowup square (\ref{eq:app_abs_blowup}), the square
\begin{center}
    \begin{tikzcd}
       E \arrow[d, two heads] \arrow[r, two heads, tail] & Y \arrow[d, two heads]\\
        E \times_C E \arrow[r, two heads, tail] & Y \times_X Y
    \end{tikzcd}
\end{center}
is an abstract blowup square. This is done for algebraic varieties in \cite[Lemma 2.14]{voevodskycdh}, but we can generalize this proof as follows. Since diagonals are monomorphisms, the vertical maps in the square above are in $J$ as well. Therefore it suffices to show that $Y\times_X Y = Y \cup (E \times_C E)$. Now we consider the diagram
\begin{center}
    \begin{tikzcd}
        E\times_C E \arrow[r, two heads,tail] \arrow[d, two heads] & Y\times_X Y \arrow[d, two heads] & Y\setminus E \times_{X\setminus C} Y \setminus E \arrow[d, two heads] \arrow[l, hook] \\
        E \arrow[r, tail, two heads] & Y & Y \setminus E \arrow[l, hook]
    \end{tikzcd} 
\end{center}
where the vertical maps are all projections onto the first factor of the fiber product. Both the squares are pullbacks, showing that the top row is an open-closed decomposition. This shows that $$Y\times_X Y \setminus (E \times_C E) \cong Y\setminus E \times_{X\setminus C} Y \setminus E\cong Y\setminus E.$$
Now we consider the diagram
\begin{center}
    \begin{tikzcd}
      (  Y \times_X Y \setminus Y) \cap (Y \times_X Y \setminus E \times_C E) \arrow[r, hook] \arrow[d, hook] &  Y \times_X Y \setminus E \times_C E \arrow[d, hook] & Y \setminus E \arrow[l, two heads, tail, "\cong"'] \arrow[d, hook] \\
      Y \times_X Y \setminus Y \arrow[r, hook] & Y \times_X Y & Y \arrow[l, two heads, tail]
    \end{tikzcd}
\end{center}
where again the squares are pullbacks and therefore the top row an open-closed decomposition. This shows that $ (  Y \times_X Y \setminus Y) \cap (Y \times_X Y \setminus E \times_C E) =\emptyset$, and hence by (xv) in Definition \ref{defn:app_top_Nagata_setup} also $Y\times_X Y \setminus (Y \cup E\times_C E) = \emptyset$, in other words $Y\times_X Y = Y \cup E \times_C E$, as desired. By the same arguments, we see that $AC$ is a regular cd-structure on $\comp(\C)$.

To show that $A\cup L$ is a c-regular cd-structure, we follow the proof of \cite[Lemma 6.23]{kuij_descent}. It suffices to show that for a localization square (\ref{eq:app_localization}), there is an epimorphism of $\tau^c_{A\cup L}$-separated presheaves
$$\phi:y_{X} \amalg y_{X\setminus U} \times_{y_{\emptyset}} y_{X \setminus U} \to Y_{X}\times_{y_{U}} y_X.$$
Generalizing the proof of \cite[Corollary 6.27]{kuij_descent}, we see that for any Grothendieck topology with jointly epic covers, representable presheaves are separated. By the same arguments as in \cite[Lemma 6.26]{kuij_descent}, we see that $\tau^c_{A\cup L}$-covers are jointly epic (where we use that $X \hookleftarrow U$ is an epimorphism in $\span(\C))$. To show that $\phi$ is an epimorphism of $\tau^c_{A\cup L}$-separated presheaves, by \cite[Lemma 6.25]{kuij_descent} it suffices to show that $\phi$ is locally surjective. Lastly, to show that $\phi$ is locally surjective, we can almost follow the proof of \cite[Lemma 6.28]{kuij_descent} line by line, using Lemma \ref{lem:app_standard}. One thing we do need to show, is that when $p,p':V\cap V' \to Y$ agree after restricting to a dense morphism $V'' \hookrightarrow V \cap V'$, then $p=p'$. Indeed, let $P$ be the pullback 
\begin{center}
    \begin{tikzcd}
        P: = (V\cap V') \times_{Y \times Y} Y \arrow[d, tail, two heads] \arrow[r] & Y \arrow[d, two heads, tail, "\Delta"]\\
        V \cap V' \arrow[r, "{\langle p, p' \rangle}"']& Y \times Y
    \end{tikzcd}
\end{center}
where the diagonal map on the right is in $J$ by (xvi) in Definition \ref{defn:app_top_Nagata_setup}. Since the square
\begin{center}
    \begin{tikzcd}
        V'' \arrow[d, hook] \arrow[r, "p=p' "] & Y \arrow[d, two heads, tail, "\Delta"]\\
        V \cap V' \arrow[r, "{\langle p, p' \rangle}"']& Y \times Y
    \end{tikzcd}
\end{center}
commutes, there is a map $V''\to P$ making
\begin{center}
    \begin{tikzcd}
        V'' \arrow[r] \arrow[dr,hook] & P \arrow[d, two heads, tail]\\
        & V \cap V'
    \end{tikzcd}
\end{center}
commute. Therefore the closure $\overline{V''}\closedrightarrow V\cap V'$ must factors through $P \closedrightarrow V \cap V'$. Since $V''$ is dense in $V \cap V'$, this implies $P = V \cap V'$ and $p=p'$ on all of $V \cap V'$.   
\end{proof}
We can now re-prove the relevant results of \cite{kuij_descent} for $\span(\C)$ and $\comp(\C)$ where $\C$ is an arbitrary noetherian Nagata setup.
\begin{prop}\label{prop:excision_is_descent}
    Let $(\C,I,P)$ be a noetherian Nagata setup and let $\D$ be a complete $\infty$-category. Then the $\infty$-categories $$\hsh_{\tau^c_{A\cup L}}(\span(\C),\D)\ \ \textup{ and }\ \ \hsh_{\tau^c_{AC}}(\comp(\C)_0;\D)$$ are equivalent to the categories of presheaves that send squares in $A\cup L$, respectively squares in $AC$, to pullback squares. The $\infty$-category 
    $$\hsh_{\tau_{AC}}(\comp(\C);\D)$$
    is equivalent to the category of presheaves $F$ that satisfy $F(\emptyset)\simeq *$ and that send squares in $AC$ to pullback squares.
\end{prop}
\begin{proof}
    This follows from \cite[Proposition 3.8]{voevodsky}, \cite[Proposition 3.6 and Proposition 4.7]{kuij_descent}.
\end{proof}
\begin{rmk}
   In fact, the category of presheaves on $\comp(\C)$ that send $\emptyset$ to the terminal object and $AC$ squares to pullback squares, also coincides with the category of sheaves for $\tau_{AC}$, as the topos $\sh(\comp(\C);\S$ is hypercomplete.  For the given topologies on $\span(\C)$ and $\comp(\C)_0$ this is not known; the standard arguments showing that excision is equivalent to \v{C}ech descent (\cite[Corollary 5.10]{voevodsky}, \cite[Theorem 3.2.5]{AHW}) do not apply since pullbacks in these categories are not well-behaved. 
\end{rmk}
In the following proposition and its proof, $\hsh_{\tau^c_{A\cup L}}(\span(\C);\D)_\emptyset$ and $\hsh_{\tau_{AC^c}}(\comp(\C)_0;\D)$ denote the categories of $\D$-valued hypersheaves $F$ such that $F(\emptyset\simeq *)$ is equivalent to the terminal object of $\D$. 
\begin{prop}\label{prop:descent_principle_general}
  Let $(\C,I,P)$ be a noetherian Nagata setup, and let $\D$ be a cocomplete, pointed $\infty$-category. Then the inclusion $\comp(\C) \to \span(\C)$ induces an equivalence of $\infty$-categories
  $$\hsh_{\tau^c_{A\cup L}}(\span(\C);\D)_\emptyset\simeq \hsh_{\tau_{AC}}(\comp(\C);\D) $$
\end{prop}
\begin{proof}
   To show that $$\hsh_{\tau^c_{A\cup L}}(\span(\C);\D)_\emptyset\simeq \hsh_{\tau^c_{AC}}(\comp(\C)_0;\D)_\emptyset$$  we need to show that the functor
 $$\comp(\C)_0\to \span(\C) $$
 satisfies the conditions of the Comparison Lemma (\cite{comparison}) with respect to the topologies $\tau^c_{AC}$ and $\tau^c_{A\cup L}$. For this we can use the exact same proof as in \cite[Lemma 7.8]{kuij_descent}. We need property (ix) of Definition \ref{defn:app_top_Nagata_setup}, and also the fact that for any $X$ in $\C$, the map $X\to *$ can be written as an $I$-morphism followed by a $P$-morphism $X \hookrightarrow \overline{X}\twoheadrightarrow *$.

 To see that $\hsh_{\tau^c_{AC}}(\comp(\C)_0;\D) \simeq \hsh_{\tau_{AC}}(\comp(\C);\D) $ we can use \cite[Proposition 5.2]{kuij_descent}. This shows that there is an equivalence of categories 
    $$\fun(\comp(\C);\D)_\emptyset \to \fun(\comp(\C)_0,\D)_\emptyset$$
   and it is clear that this restricts to an equivalence of categories of presheaves that send squares in $AC$ to pullback squares.
\end{proof}

\subsection{Local six-functor formalisms}\label{subsect:local_6ff}
In this subsection, we fix a noetherian Nagata setup $(\C,I,P)$. On this setup, we define a subcategory $\textup{\textbf{6FF}}^{\loc}\subseteq   \textup{\textbf{6FF}}$ whose objects we call \textit{local six-functor formalisms}. Their distinguishing features are that they take values in presentable, stable $\infty$-categories, and that for $i:U\hookrightarrow X$ in $I$, the diagram 
\begin{equation}\label{eq:recollement}
    \begin{tikzcd}[sep = large]
D(X\setminus U) \arrow[r, "j_* \cong j_!" description] &  D(X) \arrow[l, "j^ !", shift left=2] \arrow[l, "j^ *"', shift right=2] \arrow[r, "i^!\cong i^*" description] &  D(U) \arrow[l, "i_!"', shift right=2] \arrow[l, "i_*", shift left=2]
\end{tikzcd}
\end{equation}
is a recollement. 

In this subsection, we identify the category of local six-functor formalisms with a subcategory of symmetric monoidal hypersheaves on $\span(\C)$, using the equivalence proven in Proposition \ref{prop:w6FF}.

\begin{defn}\label{defn:6ff_unstable}
    Let $\textup{\textbf{6FF}}^\loc\subseteq   \textup{\textbf{6FF}}$ be the subcategory spanned by Nagata six-functor formalisms 
    $$D:{\corr(\C)}^\otimes \to \cat $$
    such that
    \begin{enumerate}[label=(\roman*)]
    \item for every $X$ in $\C$, $D(X)$ is a presentable stable $\infty$-category, 
        \item$D(\emptyset)$ is the trivial one-object category $\mathbbm{1}$,
        \item for every localization square (\ref{eq:app_localization}), seen as a square in $\corr(\C)$, the diagram 
        \begin{center}
            \begin{tikzcd}
                D(U) \arrow[r, "i^!"] \arrow[d, "e^*"] & D(X) \arrow[d, "j^*"]\\
                D(\emptyset) \arrow[r, "f_!"] & D(X \setminus U)
            \end{tikzcd}
        \end{center}    
         is a pullback square in $\cat$,
    \end{enumerate}
    and by morphisms that are component-wise colimit-preserving. We call objects in this category \textit{local six-functor formalisms}.
    \end{defn}
The following justifies why we think of local six-functor formalisms as ``Nagata six-functor formalisms with recollements''.
\begin{lem}\label{lem:fiber_wordt_recollement}
  Let $D:{\corr(\C)^{\otimes}}\to \cat$ be a local six-functor formalism. Then for $i:U\hookrightarrow X$ an in $I$, the diagram (\ref{eq:recollement}) is a recollement
\end{lem}
\begin{proof}
Let $j:X\setminus U \to X$ be the complement of $i$. Since $i_!$ and $j^*$ are both left adjoints, in particular they are exact functors between stable $\infty$-categories. Since limits in $\cat^\ex$ are computed in $\cat$, we have that
\begin{equation}\label{eq:fiber}
 D(U)\xrightarrow{i_!}D(X) \xrightarrow{j^*} D(X\setminus U)   
\end{equation}
is a fiber sequence not only in $\cat$, but also in $\cat^\ex$. Now we consider the sequence
\begin{equation}
\begin{tikzcd}
    X\setminus U \arrow[r, tail, two heads, "j"] & X & X \setminus U \arrow[l, tail, two heads, "j"']
\end{tikzcd}
\end{equation}
in the $\infty$-category $\corr(\C)$, which composes to
$$X\setminus U \xlongrightarrow{\id} X\setminus U.$$
This implies that the composition 
$$D(X\setminus U) \xlongrightarrow{j_\natural} D(X)\xlongrightarrow{j^*} D(X\setminus U)$$
is the identity, so the co-unit $j^*j_\natural \Rightarrow \id$ is an isomorphism by \cite[Lemma A.1.1.1]{elephant}, which implies that $j_\natural$ is fully faithful. Therefore, by \cite[Lemma A.2.5]{nine}, we have that the fiber sequence (\ref{eq:fiber}) in $\cat^\ex$
is a Verdier sequence, and that moreover the the sequence formed by the right adjoints is a Verdier sequence. This implies that the diagram (\ref{eq:recollement}) is a recollement. \end{proof}

By classical arguments, we can show that local six-functor formalisms satisfy proper cdh-descent. Recall that we defined abstract blowup squares in arbitrary topological Nagata setups in Definition \ref{def:abs_blowup_square}.

\begin{lem}\label{lem:proper_cdh_descent}
    For $D$ a local six-functor formalism, $D$ sends abstract blowup squares to pullbacks.
\end{lem}
\begin{proof}
This is a standard argument that can be found for example in the proof of \cite[Proposition 6.24]{HOYOIS2017197}. The crux is that we can use the criteria of \cite[Proposition 5.17]{Dagvii} to check whether for an abstract blowup square
\begin{center}
  \begin{tikzcd}
             E \arrow[d, two heads] \arrow[r, two heads,tail, "k"]&Y\arrow[d, two heads, "p"]\\
             C \arrow[r, two heads,tail, "i"'] & X
         \end{tikzcd}  
\end{center}
the square 
\begin{center}
    \begin{tikzcd}
        D(X) \arrow[d, "p^*"] \arrow[r, "i^*"] & D(C) \arrow[d] \\
        D(Y) \arrow[r, "k^*"]& D(E)
    \end{tikzcd}
\end{center}
is a pullback square in $\cat$. One of the criteria is that $i^*$ and $p^*$ are jointly conservative. The other criteria is stated purely in terms of objects and morphisms in the involved categories. As in the proof of \cite[Proposition 6.24]{HOYOIS2017197}, both follow from the existence of recollements and base change.
\end{proof}

 The following definition expands on Definition \ref{defn:BC_span_compleet}.
\begin{defn}\label{defn:bc_span_sheaf}
We denote by 
$$\fun^{\lax,\cl}_{BC,A\cup L}(\span(\C)_\times^\op\mid \span(\C)^\op,\cat\mid\PrSt^{LL})_\emptyset \subseteq \fun^{\lax,\cl}_{BC}(\span(\C)_\times^\op\mid \span(\C)^\op,\cat\mid\cat^{LL})_\emptyset $$
the subcategory spanned by lax cartesian structures $F:\span(\C)_\times^\op\to \cat$  such that in addition to the conditions in Definition \ref{defn:BC_span_compleet}, 
 \begin{enumerate}
        \item[(iv)]  for every $X$ in $\C$, $F(X)$ is a presentable stable $\infty$-category,
        \item[(v)]
        for the initial object $\emptyset$ we have $F(\emptyset) = \mathbbm{1}$, the trivial one-object stable $\infty$-category,
        \item[(vi)] the restriction $F|_\span:\span(\C)^\op\to \cat$ is a hypersheaf for the topology $\tau^c_{A\cup L}$,
    \end{enumerate}
        and by natural transformations $\alpha:F \to F'$ that are compatible with the right adjoints, and objectwise colimit-preserving.
\end{defn}
By Proposition \ref{prop:excision_is_descent}, a presheaf $F:\span(\C)^\op \to \D$ with values in any complete $\infty$-category is a $\tau^c_{A\cup L}$-hypersheaf if and only if it sends localization squares and abstract blowup squares to pullback squares. With this, and with Lemma \ref{lem:fiber_wordt_recollement} and Lemma \ref{lem:proper_cdh_descent}, it is easy to see the following.
\begin{lem}\label{lem:6ff}
The equivalence $$e^*: \textup{\textbf{6FF}} \xlongrightarrow{\sim}\fun^{\lax,\cl}_{BC}(\span(\C)_\times^\op\mid \span(\C)^\op,\cat\mid\cat^{LL}) $$ in Proposition \ref{prop:w6FF} restricts to an equivalence 
$$ e^*:\textup{\textbf{6FF}}^{\loc}\xlongrightarrow{\sim }\fun^{\lax,\cl}_{BC,A\cup L}(\span(\C)_\times^\op\mid \span(\C)^\op,\cat\mid\PrSt^{LL})_\emptyset .$$
\end{lem}
By the following lemma, we can identify $\textup{\textbf{6FF}}^\loc$ with a subcategory of $\infty$-operad maps $\span(\C)_\times^\op\to \PrSt^\otimes$ instead.
\begin{lem}\label{lem:lift_applied}
For $F:\span(\C)_\times^\op \to \cat$ a lax cartesian structure in $$\fun^{\lax,\cl}_{BC,A\cup L}(\span(\C)_\times^\op\mid \span(\C)^\op,\cat\mid\PrSt^{LL})_\emptyset,$$ the corresponding $\infty$-operad map $\tilde F:\span(\C)_\times^\op \to \cat^\otimes$ factors through $\PrSt^\otimes$. 
\end{lem}
\begin{proof}
For $F:\span(\C)_\times^\op \to \cat$ in $\fun^{\lax,\cl}_{BC,A\cup L}(\span(\C)_\times^\op\mid \span(\C)^\op,\cat^{\ex}\mid\PrSt^{LL})_\emptyset$, it is already clear that $F(X)$ is a stable presentable $\infty$-category. We show that for the associated $\infty$-operad map $\tilde F:\span(\C)_\times \to \cat^\otimes$, the conditions of Lemma \ref{lem:lift} are satisfied. It is clear that conditions (i) and (ii) are satisfied. 
By Proposition \ref{prop:3FF}, we know that $F$ can be extended to a Nagata three-functor formalism (in fact, to a Nagata six-functor formalism) $D:{\corr(\C)^{\otimes}}\to \cat$. To see that $F$ satisfies condition (iii), it suffices to show this for $D$. Let $X,Y$ be in $\C$. We have a commutative diagram \begin{center}
    \begin{tikzcd}
        (X,Y)_{\underline 2} &   &  \\
        (X\times Y, X\times Y)_{\underline{2}} \arrow[u, "{(p_X,p_Y)}"] & (X\times Y \times X \times Y)_{\underline 1} \arrow[l] & (X\times Y)_{\underline 1} \arrow[l, "\Delta"] \arrow[llu]
    \end{tikzcd}
\end{center}
 in ${\corr(\C)^{\otimes}}$. Using the equivalences $D((X,Y)_{\underline 2}) \simeq D(X)\times D(Y)$ and $D((X\times Y,X\times Y)_{\underline 2}) \simeq D(X\times Y)\times D(X\times Y)$, the diagram evaluates to
\begin{center}
\begin{tikzcd}
D(X)\times D(Y) \arrow[d, "p_X^*\times p_Y^*"'] \arrow[rrrd]        &  &  &                                        \\
D(X\times Y)\times D(X\times Y) \arrow[rr] \arrow[rrr, bend right, "\otimes"]&  & D(X\times Y\times X \times Y) \arrow[r, "\Delta^*"'] &  D(X\times Y) 
\end{tikzcd}
\end{center}
where the composite of the lower horizontal arrows gives the symmetric monoidal structure on $D(X\times Y)$. By assumption, $D(X\times Y)$ is a closed symmetric monoidal $\infty$-category, so this composite preserves colimits in both variables. Since $p_X^*$ and $p_Y^*$ are both left adjoints and therefore colimit-preserving, it follows that the map 
$$D(X)\times D(Y)\to D(X \times Y) $$ preserves colimits in each variable, which is what we needed to show. 
\end{proof}
We define the following subcategory of the category of $\infty$-operad maps $\alg_{\span(\C)^\op}(\PrSt)$.
\begin{defn}\label{defn:alg_span_BC}
    Let $\alg_{\span(\C)^\op}^{BC,A\cup L}(\PrSt) \subseteq \alg_{\span(\C)^\op}(\PrSt)$ be the subcategory spanned by  $\infty$-operad maps $F:\span(\C)_\times^\op \to \PrSt^\otimes$ such that  \begin{enumerate}[label=(\alph*)]
        \item for $p:X\to Y$ in $P$, the image $p^*:F(Y)\to F(X)$ is a left-left adjoint, and we denote the adjoints by $p^*\vdash p_\natural \vdash p^!$,
        \item for $i:U\hookrightarrow X$ in $I$, the image $i_!:F(U)\to F(X)$ is a left-left adjoint, and we denote the adjoints by $i_!\vdash i^\natural\vdash i_*$,
       \item squares of types  $(I_-,I_-)$, $(I_-,P_\times^\op)$, $(P_-^ \op,P_\times^\op)$ and $(P_-^\op,I_-)$ are sent to vertically right-adjoinable squares in $\PrSt^\otimes$,
       \item  we have $F(\emptyset) = \mathbbm{1}$, the trivial one-object category,
       \item the underlying functor $F:\span(\C)^\op \to \PrSt$ is a hypersheaf for the topology $\tau^c_{A\cup L}$,
    \end{enumerate}
    and by natural transformations $\alpha:F \to F'$ that are compatible with the right adjoints $p_\natural$ and $i^\natural$.
\end{defn}
\begin{prop}\label{prop:6ff_to_span}
    There is an equivalence of $\infty$-categories 
$$\textup{\textbf{6FF}}^\loc\simeq\alg_{\span(\C)^\op}^{BC,A\cup L}(\PrSt)_\emptyset.$$
\end{prop}
\begin{proof}
    By Lemma \ref{lem:6ff}, suffices to show that $\alg_{\span(\C)^\op}^{BC,A\cup L}(\PrSt)_\emptyset$ is equivalent to $\fun^{\lax,\cl}_{BC,A\cup L}(\span(\C)_\times^\op\mid \span(\C)^\op,\cat\mid\PrSt^{LL})_\emptyset$. To be more precise, we want to show that the equivalence of $\infty$-categories
    $$\pi_*:\alg_{\span(\C)^\op}(\cat) \to \fun^\lax(\span(\C)_\times^\op,\cat)$$
    from \cite[Proposition 2.4.1.7]{HA}, restricts to an equivalence of $\infty$-categories
    $$\pi_*:\alg_{\span(\C)^\op}^{BC,A\cup L}(\PrSt)_\emptyset \xrightarrow{\sim} \fun^{\lax,\cl}_{BC}(\span(\C)_\times^\op\mid \span(\C)^\op,\cat\mid\cat^{LL}).$$
  
    Let $F$ be in $\alg_{\span(\C)^\op}(\cat)$, and and let $G$ be the lax cartesian structure $\pi_* F$. First, suppose that $F$ is in $\alg_{\span(\C)^\op}^{BC,A\cup L}(\PrSt)_\emptyset$. Then (c) implies that condition (B) in Definition \ref{defn:bc_span} holds for $G$, and (a) and (b) immediately imply conditions (i) and (ii) in Definition \ref{defn:BC_span_compleet}. To see that (iii) holds, we recall that for $X$ in $\C$, the symmetric monoidal structure on $F(X) = G(X)$ is given by 
    $$F(X) \times F(X) \xrightarrow{\otimes} F(X \times X) \xrightarrow{\delta^*} F(X).$$
By assumption, $\otimes$ preserves colimits in every variable, and $\Delta^*$ is a left adjoint and therefore colimit-preserving. Since $F(X)$ is presentable, it follows that for $A$ in $F(X)$, $A\otimes - $ is a left adjoint. Lastly, it is immediate that conditions (iv) in Definition \ref{defn:bc_span_sheaf} hold, and (v) and (vi) follow from (d) and (e).

On the other hand, suppose $G$ is in $\fun^{\lax,\cl}_{BC}(\span(\C)_\times^\op\mid \span(\C)^\op,\cat\mid\cat^{LL})$. Then Lemma \ref{lem:lift_applied} shows that $F$ is in $\alg_{\span(\C)^\op}(\PrSt)$. Condition (B) in Definition \ref{defn:bc_span} implies (c), and that (a) and (b) hold for $F$, follows from conditions (i) and (ii) in Definition \ref{defn:BC_span_compleet}. Condition (d) and (e) follow from (v) and (vi) in Definition \ref{defn:bc_span_sheaf}.
\end{proof}

\subsection{A symmetric monoidal comparison lemma}\label{subsect:monoidal_compact_support}
The most important ingredient that we need to proof further results about local six-functor formalisms, is a version of the comparison lemma that deals with ``lax symmetric monoidal sheaves''. For $f:(\C',\tau') \to (\C,\tau)$ a morphism between sites, the classical comparison lemma gives sufficient conditions such that $f$ induces an equivalence of categories
$$f^*:\sh_\tau(\C;\set) \xrightarrow{\sim} \sh_{\tau'}(\C';\set).$$
One can ask the following question. If in this setup $\C'$ and $\C$ are symmetric monoidal 1-categories, and is $f$ a strong symmetric monoidal functor, then we wonder when $f$ induces an equivalence between categories of sheaves that are also lax symmetric monoidal functors.

\begin{defn}\label{defn:otimes_stable}
  Let $(\C,\otimes)$ be a symmetric monoidal 1-category and $P$ a cd-structure on $\C$. Then $P$ is called a $\otimes$-\textit{stable} cd-structure if
for a distinguished square
    \begin{equation}\label{eq:gendistsquare}
        \begin{tikzcd}
            B \arrow[r]\arrow[d]&A\arrow[d] \\
            Y \arrow[r] & X
        \end{tikzcd}
    \end{equation}
    and an object $Z$ of $\C$, the square
       \begin{center}
        \begin{tikzcd}
            B \otimes Z \arrow[r]\arrow[d]&A \otimes Z\arrow[d] \\
            Y\otimes Z \arrow[r] & X \otimes Z
        \end{tikzcd}
    \end{center}
    is also distinguished.  
 \end{defn}

\begin{lem}\label{lem:pre_comparison}
    let $\C$ be a symmetric monoidal 1-category equipped with a cd-structure $P$ that is $\otimes$-stable, (c-)complete, (c-)regular and compatible with a dimension function. Let $\D^\otimes$ be a complete presentable symmetric monoidal $\infty$-category. Then the localization functor
$$L:\psh(\C;\D) \to \hsh_\tau(\C;\D)$$
is a symmetric monoidal localization (\cite[Definition 3.3]{universality}) for the Day convolution product on $\psh(\C;\D)$, where $\tau$ is the (coarse) topology generated by $P$. 
\end{lem}
\begin{proof}
    We show this for the case that $P$ is a c-complete and c-regular cd-structure. First, we assume that $\D = \S$.
    
    By \cite[Proposition 4.7]{kuij_descent}, $\hsh_\tau(\C;\S)\subseteq \psh(\C;\S)$ is the full subcategory spanned by presheaves $F:\C^\op\to \S$ that send all distinguished squares to pullback squares. By the Yoneda lemma, this is the case if for every distinguished square (\ref{eq:gendistsquare}) the canonical morphism 
  \begin{equation}\label{eq:local_morph}
      y_B \sqcup_{y_A} y_C \to y_D
  \end{equation}
induces an equivalence
    $ \fun(y_D,F)\to  \fun(y_B \sqcup_{y_A} y_C).$
    In other words, $\hsh_\tau(\C;\S)$ is the localization of $\psh(\C;\S)$ at the class $S$ of all morphisms (\ref{eq:local_morph}) associated to distinguished squares. The Day convolution product preserves colimits in each variable, see \cite[Corollary 4.8.1.12]{HA}. Therefore, to demonstrate that $L$ is a symmetric monoidal localization, it suffices to show that $S$ is closed under Day convolution with representable presheaves. The Day convolution product of (\ref{eq:localization}) with a representable presheaf $y_X$ is 
    $$y_{B\otimes X} \sqcup_{y_{A\otimes X}} y_{C\otimes X} \to y_{D\otimes X},$$
    since the Yoneda embedding is strong symmetric monoidal with respect to Day convolution. By the assumption on $P$, this is again a morphism in $S$. 

    Now let $\D^\otimes$ be a complete presentable symmetric monoidal $\infty$-category. Then there is an equivalence $\psh(\C;\D)\simeq\psh(\C;\S)\otimes \D$ where $\otimes$ is the Lurie tensor product of presentable $\infty$-categories. We note that $\psh(\C;\S)\otimes \D$ gets an induced symmetric monoidal structure from the Day convolution product on $\psh(\C;\S)$ and the symmetric monoidal structure on $\D$, whereas $\psh(\C;\D)$ has a symmetric monoidal structure given by Day convolution. By \cite[Proposition 3.10]{BenMoshe2023} however, the equivalence $\psh(\C;\D)\simeq\psh(\C;\S)\otimes \D$ is symmetric monoidal.
    
    Similarly, $\hsh_\tau(\C;\D)\simeq\hsh_\tau(\C;\S)\otimes \D$, see also \cite[Proposition 4.5]{kuij_descent}. By \cite[Lemma 4.21]{BenMoshe2023}, the localization 
    $$L:\psh(\C;\D) \to \hsh_\tau(\C;\D)$$
    is therefore also a symmetric monoidal localization for the Day convolution product on $\psh(\C;\D)$. 
\end{proof}
Let $\psh(\C,\D)^\otimes$ denote the symmetric monoidal $\infty$-category encoding the Day convolution product (\cite[Definition 2.8]{Glasman}). By \cite[Lemma 3.4]{universality}, there is a symmetric monoidal structure $\hsh(\C;\D)^\otimes$ on $\hsh_\tau(\C;\D)$, and
$L:\psh(\C;\D)\to \hsh_\tau(\C;\D)$ extends to a strong symmetric monoidal functor 
$$ L^\otimes:\psh(\C;\D)^\otimes \to \hsh(\C;\D)^\otimes$$
that is left adjoint to the inclusion $\hsh(\C;\D)^\otimes\subseteq \psh(\C;\D)$, and this inclusion is lax symmetric monoidal.
\begin{nota}
    Let $(\C,\otimes)$ be a symmetric monoidal 1-category equipped with a Grothendieck topology $\tau$, and let $\D^\otimes$ be a symmetric monoidal $\infty$-category. Then we denote by 
    $$\alg^{\tau}_{\C^\op}(\D) \subseteq \alg_{\C^\op}(\D)$$
    the category of operad maps $(\C^\op)^\otimes\to \D^\otimes$ such that the underlying functor $\C^\op \to \D$ is a $\tau$-hypersheaf. We call such operad maps \textit{lax symmetric monoidal hypersheaves}.
\end{nota}
\begin{prop}
    Let $\C$ be a symmetric monoidal 1-category equipped with a cd-structure $P$ that is $\otimes$-stable, (c-)complete, (c-)regular and compatible with a dimension function, and $\D^\otimes$ a complete presentable symmetric monoidal $\infty$-category. Then there is an equivalence of $\infty$-categories
    $$\alg_{\C^\op}^\tau(\S)\simeq \textup{CAlg}(\hsh_\tau(\C;\D)),$$
    where $\tau$ is the (coarse) topology generated by $P$.
\end{prop}
\begin{proof}
    By \cite[Proposition 2.12]{Glasman}, the equivalence 
    $$\textup{Fun}_{\fin_\part}(\fin_\part,\psh(\C;\D)^\otimes) \simeq \textup{Fun}_{\fin_\part}(\C^\otimes,\D^\otimes)$$
arising from the defining universal property of $\psh(\C;\D)^\otimes$, restricts to an equivalence between the category of commutative algebra objects $ \textup{CAlg}(\psh_\tau(\C;\D))$ and the category of $\infty$-operad maps $\Alg_{\C^\op}(\D)$.  
Under this equivalence, $\textup{CAlg}(\hsh_\tau(\C;\D) \subseteq \textup{CAlg}(\psh(\C;\D))$ is equivalent to the category of lax symmetric monoidal hypersheaves $\alg_{\C^\op}^\tau(\D)$.
\end{proof}
\begin{lem}[The comparison lemma for lax symmetric monoidal hypersheaves]\label{lem:comparison_spaces}
 Let $\C$ and $\C'$ be symmetric monoidal 1-categories equipped with cd-structures $P$ and $P'$ that satisfy the conditions of Lemma \ref{lem:pre_comparison}, and let $\D^\otimes$ be a complete, presentable symmetric monoidal $\infty$-category.  Let $f:\C \to \C'$ be a strong symmetric monoidal functor, such that the induced functor on sheaf categories
 $$f_!:\sh_{\tau}(\C;\set) \to \sh_{\tau'}(\C';\set) $$
 is an equivalence, where $\tau$ and $\tau'$ are the (coarse) topologies generated by $P$ and $P'$ respectively. Then $f$ induces an equivalence of categories 
 $$f_!:\alg^\tau_{\C^\op}(\D) \to \alg^{\tau'}_{(\C')^\op}(\D).$$
\end{lem}
\begin{proof}
 The assumptions imply that $f$ induces an equivalence of categories of hypersheaves $f_!:\hsh_{\tau}(\C;\D) \to \hsh_{\tau'}(\C';\D),$ see for example \cite[Remark 4.3 and Corollary 4.6]{kuij_descent}. Now consider the square
  \begin{center}
      \begin{tikzcd}
          \psh(\C;\D)^\otimes \arrow[r, "L^\otimes"] \arrow[d, "f_!"] & \hsh_{\tau}(\C;\D)^\otimes \arrow[d, "f_!"]\\
          \psh(\C';\D)^\otimes \arrow[r, "(L')^\otimes"] & \hsh_{\tau'}(\C';\D)^\otimes
      \end{tikzcd}
  \end{center}
where the vertical functors are strong symmetric monoidal by Lemma \ref{lem:pre_comparison}. This square commutes, since the square of right adjoints, given by inclusions and restrictions along $f$ commutes. The vertical functor $f_!$ on the left is strong symmetric monoidal, by \cite[Proposition 3.6]{BenMoshe2023}. Since $L$ is surjective, it follows that $f_!$ on the right is strong symmetric monoidal. Since the underlying functor of $\infty$-categories is an equivalence, it follows that $f_!$ induces an equivalence between categories of commutative algebra objects.
\end{proof}

    \begin{rmk}[Comparison to {\cite[Proposition 6.6]{cirici_horel}}]\label{rmk:more_comparison} 
In the case that the symmetric monoidal 1-categories $\C$ and $\C'$ are both cartesian, we expect Lemma \ref{lem:comparison_spaces} is true when $\tau$, $\tau'$ are arbitrary topologies (not necessarily generated by cd-structures). This follows from a generalisation of the proof of \cite[Proposition 6.6]{cirici_horel} by Cirici and Horel about ``strong symmetric monoidal hypersheaves''. Their proof uses Hinich' theory of strict localizations of symmetric monoidal $\infty$-categories \cite[Section 3.2]{hinich}. We expect that by the same arguments, using Hinich' notion of \textit{right symmetric monoidal localization} (\cite[Section 3.3]{hinich}) instead, one can obtain an equivalence of $\infty$-categories of lax symmetric monoidal hypersheaves too. However, these arguments rely on the fact that $\C$ and $\C'$ are \textit{cartesian} symmetric monoidal $\infty$-categories. This is not the case for the categories $\span(\C)$ and $\comp(\C)_0$, which is why we need Lemma \ref{lem:comparison_spaces}. 
\end{rmk}

\subsection{Reducing to $\comp(\C)$}\label{subsect:comp}
In this subsection we show that the category of local six-functor formalisms $\textup{6FF}^\loc$ is equivalent to a subcategory of symmetric monoidal hypersheaves on $\comp(\C)$. Moreover, in the special case of the Nagata setup of varieties over a field of characteristic zero, $\textup{6FF}^\loc$ embeds into the category of symmetric monoidal hypersheaves on the category of smooth and complete varieties $\smcomp$. 

Again, $(\C,I,P)$ is a noetherian Nagata setup.
\begin{prop}\label{prop:cohom}
 Let $\D$ be a pointed, complete, presentable symmetric monoidal $\infty$-category. Then there is an equivalence of $\infty$-categories
\begin{equation}\label{eq:compactly_supportes_sheaf_theories}
    \alg_{\span(\C)^\op}^{\tau^c_{A\cup L}}(\D)\simeq \alg_{\comp(\C)^\op}^{\tau_{AC}}(\D).
\end{equation}

For $(\C,I,P)$ the noetherian Nagata setup of Example \ref{ex:varieties}, over a field of characteristic zero, there is an equivalence of $\infty$-categories 
 \begin{equation}\label{eq:compactly_supportes_sheaf_theories_2}
    \alg_{\comp^\op}^{\tau_{AC}}(\D)\simeq \alg_{\smcomp^\op}^{\tau_{B}}(\D).   
    \end{equation}
\end{prop}
\begin{proof} This is the symmetric monoidal version of Proposition \ref{prop:descent_principle_general} and \cite[Lemma 7.4]{kuij_descent}.

It is easy to check that all the cd-structures involved, are $\otimes$-stable, and in Subsection \ref{subsect:span_and_comp} we showed that they are (c-)complete, (c-)regular and compatible with a dimension function. Therefore we can use Lemma \ref{lem:comparison_spaces} to show that 
   $$\alg_{\span(\C)^\op}^{\tau^c_{A\cup L}}(\D)\simeq \alg_{\comp(\C)_0^\op}^{\tau^c_{AC}}(\D),$$
    Similarly, in the special case of varieties over a field of characteristic zero,
 $$\alg_{\comp^\op}^{\tau_{AC}}(\D)\simeq \alg_{\smcomp^\op}^{\tau_{B}}(\D).$$

To see that 
$$\alg_{\comp(\C)_0^\op}^{\tau^c_{AC}}(\D) \simeq \alg_{\comp^\op}^{\tau_{AC}}(\D),$$
we note that the inclusion $i:\comp(\C) \to \comp(\C)_0$ is strong symmetric monoidal, and therefore by by \cite[Proposition 3.6]{BenMoshe2023} the functor  $$i_!:\fun(\comp(\C);\D) \to \fun(\comp(\C)_0,\D)$$
  given by left Kan extension along $i$ is strong symmetric monoidal  with respect to the Day convolution product on both sides. Now, by \cite[Proposition 5.2]{kuij_descent}, this functor restricts to a strong symmetric monoidal functor
    $$i_!:\fun(\comp(\C);\D)_\emptyset \to \fun(\comp(\C)_0,\D)_\emptyset$$
    that is an equivalence on the underlying $\infty$-categories.
   This implies an equivalence between the categories of commutative monoids in both categories, which are exactly the lax symmetric monoidal functors. One sees that this equivalence moreover restricts to an equivalence between the categories of functors that send abstract blowups to pullback squares, giving the desired equivalence between categories of lax symmetric monoidal hypersheaves.
\end{proof} 

 Proposition \ref{prop:6ff_to_span} shows that $\textup{\textbf{6FF}}^\loc$ is equivalent to a certain subcategory of the $\infty$-category $\alg_{\span(\C)^\op}^{\tau^c_{A\cup L}}(\PrSt)$. Now we will try to chase $\textup{\textbf{6FF}}^\loc$ through the equivalences (\ref{eq:compactly_supportes_sheaf_theories}), and (\ref{eq:compactly_supportes_sheaf_theories_2}) in the case of varieties over a field of characteristic zero, for $\D$ the category $\prst$ of presentable stable $\infty$-categories. We start by showing what $\textup{\textbf{6FF}}^\loc$ corresponds to under (\ref{eq:compactly_supportes_sheaf_theories}).
\begin{defn}\label{defn:bc_comp}
    Let $\alg_{\comp(\C)^\op}^{BC,AC}(\PrSt)\subseteq \alg_{\comp(\C)^\op}(\PrSt)$ denote the subcategory spanned by $\infty$-operad maps 
    $$F:\comp(\C)_\times^\op \to \PrSt^\otimes $$
    such that
    \begin{enumerate}[label=(\arabic*)]
        \item for every $p:X\to Y$ in $\comp(\C)$, the image $p^*:F(Y)\to F(X)$ is a left-left adjoint, and we denote the adjoints by $p^*\vdash p_\natural \vdash p^!$,
        \item squares of type $(P_-^\op, P_\times^\op)$ are sent to vertically right-adjoinable squares in $\PrSt^\otimes$,
        \item the underlying functor $\comp(\C)^\op \to \PrSt$ is a sheaf for $\tau_{AC}$,
    \end{enumerate}
    and by natural transformations $\alpha:F \to F'$ that are compatible with the right adjoints $p_\natural$.
\end{defn}

\begin{rmk}
    By Proposition \ref{prop:descent_principle_general}, a presheaf  $F:\comp(\C)^\op\to \PrSt$ is a $\tau_{AC}$-hypersheaf if and only if $F$ sends abstract blowup squares to pullbacks and $\emptyset$ to the zero object of $\PrSt$. Therefore we can rephrase the definition above as follows: $$\alg_{\comp(\C)^\op}^{BC,AC}(\PrSt)\subseteq \alg_{\comp(\C)^\op}(\PrSt)$$
   is the full subcategory spanned by $\infty$-operad maps 
    $F:\comp(\C)_{0,\times}^\op\to \PrSt^\otimes$ such that (1)-(2) hold, and
    \begin{enumerate}
      \item[(3')]  for the initial object $\emptyset$ we have $F(\emptyset) = \mathbbm{1}$, the trivial one-object $\infty$-category,
          \item[(4')] $F$ sends abstract blowup squares in $\comp(\C)$ to pullback squares,
    \end{enumerate}
      and by natural transformations $\alpha:F \to F'$ that are compatible with the right adjoints $p_\natural$. 
\end{rmk}

\begin{prop}\label{prop:cs6ff}
The equivalence
$\alg_{\span(\C)^\op}^{\tau^c_{A\cup L}}(\Prst)_\emptyset \simeq  \alg_{\comp(\C)^\op}^{\tau_{AC}}(\Prst)$ restricts to an equivalence 
\begin{equation}\label{eq:prop_cs6ff}
    \alg_{\span(\C)^\op}^{BC,A\cup L}(\PrSt)_\emptyset \xrightarrow{\sim} \alg_{\comp(\C)^\op}^{BC,AC}(\PrSt)
\end{equation}
\end{prop}

The proof of the proposition will follow from Lemma \ref{lem:cs6ff_step0} and Lemma \ref{lem:cs6ff_step_0.5}. It is clear that the restriction of (\ref{eq:compactly_supportes_sheaf_theories}) to $\alg_{\span(\C)^\op}^{BC,A\cup L}(\PrSt) $ lands in $\alg_{\comp(\C)^\op}^{BC,AC}(\PrSt)$, so we need to show that the functor (\ref{eq:prop_cs6ff}) is surjective.

\begin{lem}\label{lem:cs6ff_step0}
For $G$ in $\alg_{\span(\C)^\op}^{\tau^c_{A\cup L}}(\Prst)_\emptyset$, if $G|_{\comp(\C)_\times}$ is in $ \alg_{\comp(\C)^\op}^{BC,AC}(\PrSt)$, then
$G$ sends all $P$-morphisms to left-left adjoints, and squares of type $(P_-^\op,P_\times^\op)$ to vertically right-adjoinable squares.    
\end{lem}
\begin{proof}
We start with the following observation. For $U$ in $\C$ arbitrary, we consider a sequence 
    $$X\setminus U \xtwoheadrightarrow{j} X \xhookleftarrow{i} U$$
    where $U \xhookrightarrow{i}X$ is in $I$ and $X$ is in $\comp(\C)$ (and therefore $X \setminus U$ as well).
     Since $G$ is in $\alg_{\span(\C)^\op}^{\tau^c_{A\cup L}}(\Prst)_\emptyset$, we know that
    \begin{equation}\label{eq:fiber_seq_G}
         G(U)\xlongrightarrow{i_!} G(X)\xlongrightarrow{j^*} G(X\setminus U) 
    \end{equation}
    is a fiber sequence in $\Prst$, and therefore also a fiber sequence in $\cat^\ex$. Since the square 
    \begin{center}
        \begin{tikzcd}
            G(X) \arrow[d, "j^ *"] \arrow[r, "j^ *"] & G(X\setminus U) \arrow[d, "="]\\
            G(X\setminus U) \arrow[r, " =" ] & G(X\setminus U)
        \end{tikzcd}
    \end{center}
    is vertically right-adjoinable by the assumption on $G|_{\comp(\C)_\times}$, we have that $j^ *\circ j_\natural =\id$. This implies that $j_\natural$ is fully faithful. By \cite[Lemma A.2.5]{nine} we see that $i_!$ is fully faithful and has a right adjoint, which we denote by $i^\natural$. Moreover, the sequence (\ref{eq:fiber_seq_G}) is a Verdier sequence, and so is the sequence formed by the right adjoints
 $$G(X\setminus U) \xrightarrow{j_\natural} G(X) \xrightarrow{i^\natural} G(U).$$
 Since $j_\natural$ has a further right adjoint $j^!$ by the assumption on $G|_{\comp(\C)_\times}$, it follows again from \cite[Lemma A.2.5]{nine} that $i^\natural$ has a right adjoint $i_*$, and (\ref{eq:fiber_seq_G}) is in fact a left-recollement. 
 
\textbf{Morphisms in $P$.} Now we show that $G$ sends morphisms in $P$ to left-left adjoints. Let $q:U\twoheadrightarrow V$ be an arbitrary $P$-morphism in $\span(\C)$. Let $Y$ in $\comp(\C)$ such that there is an $I$-morphism $j:V\hookrightarrow Y$. By one of the basic properties of a Nagata setup, we can factor $U \xtwoheadrightarrow{q}V \xhookrightarrow{j} Y$ as an $I$-morphism followed by a $P$-morphism $$U \xhookrightarrow{i} X \xtwoheadrightarrow{p} Y,$$ where we may assume that $U$ is dense in $Y$. Then by \cite[Lemma 6.15]{kuij_descent} we see that the square 
    \begin{center}
        \begin{tikzcd}
            U \arrow[d, two heads, "q" ] \arrow[r, hook, "i" ] & X \arrow[d, two heads, "p"] \\
            V \arrow[r, hook, "j" ] & Y
        \end{tikzcd}
    \end{center}
    is a pullback square. Let $r:Y\setminus V \twoheadrightarrow X \setminus U$ be the induced morphisms on the complements. We can form a commuting diagram
\begin{equation}\label{eq:cs6ff_step_0_diag}
     \begin{tikzcd}
         G(V) \arrow[r, "j_!"] \arrow[d, "q^*"] & G(Y) \arrow[d, "p^*"] \arrow[r, "k^*"] & G(Y \setminus V) \arrow[d, "r^*"] \\
         G(U)\arrow[r, "i_!" '] & G(X) \arrow[r, "l^*" '] & G(X \setminus U)
     \end{tikzcd}
 \end{equation}
 in $\cat^\ex$, where the rows are left-recollements by the observations above. From the the assumption on $G|_{\comp(\C)_\times}$ the square on the right is vertically and horizontally right-adjoinable, and that the vertical morphisms in the square on the right are left-left adjoints. It follows from \cite[Theorem 2.5]{parshall_scott} that $q^*$ is a left-left adjoint.

\textbf{Squares of type $(P_-^\op,P_\times^\op)$.} We consider a square of type $(P_-^\op,P_\times^\op)$ in $\span(\C)^\op_\times$ (which we depict as a square in $\span(\C)_\times$)
\begin{equation*}
    \begin{tikzcd}
        (X_i)_I \arrow[r, two heads ] \arrow[d, two heads  ] & (Y_j)_J \arrow[d, two heads  ] \\
        (Z_i)_I \arrow[r, two heads] & (V_j)_J.
    \end{tikzcd}
\end{equation*}
where the horizontal morphisms lie over a partial map $\alpha:J \dashrightarrow I$. Let $(\overline{V}_j)_J$ be a compactification of $(V_j)_J$. 
Then we can factor each $Y_j \twoheadrightarrow V_j \hookrightarrow \overline{V_j}$ as a dense $I$-morphism followed by a $P$-morphism $Y_j \hookrightarrow \overline{Y_j} \twoheadrightarrow \overline{V_j}$. This gives a dense compactification $(\overline{Y_i})_J$ of $(Y_j)_J$ that admits a $P_-$-morphism $(\overline{Y_j})_J\twoheadrightarrow (\overline{V_j})_J$.

Similarly, for each $i$ we can factor $Z_i \rightarrow \prod_{\alpha^{-1}(i)} V_j \hookrightarrow \prod_{\alpha^{-1}(i)} \overline{V_j} $ as a dense $I$-morphism followed by a $P$-morphism 
$Z_i \hookrightarrow \overline{Z_i} \rightarrow \prod_{\alpha^{-1}(i)} \overline{V_j}  $. This gives a dense compactification $(\overline{Z_i})_I$ of $(Z_i)_I$ that admits a $P_\times$-morphism $(\overline{Z_i})_I \rightarrow (\overline{V_j})_J$. 

Now for each $i$ let $\overline{X_i}$ be the pullback
\begin{equation*}
\begin{tikzcd}
\overline{X_i} \arrow[r,  ] \arrow[d,  ] & \prod_{\alpha^{-1}(i)} \overline{Y_i} \arrow[d,  two heads] \\
\overline{Z_i} \arrow[r, two heads  ]           & \prod_{\alpha^{-1}(i)} \overline{V_i}  .                   
\end{tikzcd}
\end{equation*}
Then $(\overline{X_i})$ is the pullback of $(\overline{Y_j})_J\twoheadrightarrow (\overline{V_j})_J$ along $(\overline{Z_i})_I \twoheadrightarrow (\overline{V_j})_J$. In the cube 
\begin{equation*}
\begin{tikzcd}
 & (Y_j)_J \arrow[rr, hook] \arrow[dd,  ] &                                                                & (\overline{Y_j})_J \arrow[dd, two heads ] \\
(X_i)_I \arrow[rr] \arrow[dd, two heads  ] \arrow[ru, two heads ] &                                                & (\overline{X_i})_I \arrow[ru, two heads ] \arrow[dd, two heads ] &                                          \\
  & (V_j)_J \arrow[rr, hook]                       &                                                                & (\overline{V_j})_J                       \\
(Z_i)_I \arrow[rr, hook] \arrow[ru,two heads  ]                 &                                                & (\overline{Z_i})_I \arrow[ru, two heads ]                       &                                         
\end{tikzcd}
\end{equation*}
the left and right faces are pullbacks. By \cite[Lemma 6.15]{kuij_descent} the bottom and back faces are pullbacks, since $(\overline{Z_i})_I$ and $(\overline{Y_j})_J$ are dense compactifications of $(Z_i)_I$ and $(Y_j)_J$. By the pasting law for pullbacks, it follows that the front and top face are pullbacks as well. This implies that $(X_i)_I \to (\overline{X_i})_I$ is an $(I_-)$-morphism. Now by taking complements, we can extend the cube to the left by
\begin{equation}\label{eq:monster}
\begin{tikzcd}
& (Y_j)_J \arrow[rr, hook] \arrow[dd, two heads ] &                                                                & (\overline{Y_j})_J \arrow[dd, two heads  ] &                                                                                        & (\overline{Y_j}\setminus Y_j)_J \arrow[ll, tail, two heads  ] \arrow[dd, two heads  ] \\
(X_i)_I \arrow[rr,hook] \arrow[dd, two heads ] \arrow[ru, two heads ] &                                                & (\overline{X_i})_I \arrow[ru, two heads  ] \arrow[dd, two heads ] &                                          & (\overline{X_i}\setminus X_i)_I \arrow[ll, tail, two heads ]  \arrow[dd,  two heads] &                                                                             \\
& (V_j)_J \arrow[rr, hook]                       &                                                                & (\overline{V_j})_J                       &                                                                                        & (\overline{V_j}\setminus V_j)_J \arrow[ll, tail, two heads  ]                         \\
(Z_i)_I \arrow[rr, hook] \arrow[ru, two heads  ]                 &                                                & (\overline{Z_i})_I \arrow[ru,two heads  ]                       &                                          & (\overline{Z_i}\setminus Z_i)_I. \arrow[ll, tail, two heads  ]             &                                                                            
\end{tikzcd}
\end{equation}
where all the vertical squares are now pullbacks. The front and back squares on the right are sent to to vertically right-adjoinable squares in $\PrSt^\otimes$, since they are squares of type $(P_-^\op, P_\times^\op)$ in $(\comp_0)_\times$, and the rows are sent to left-recollements. Recall that by definition of a vertically right-adjoinable square in $\cat^\otimes$, this just means that for every $i\in I$, the induced square
\begin{center}
    \begin{tikzcd}
        G(\overline{Z_i}) \arrow[d] \arrow[r] & G(\overline{Z_i}\setminus Z_i) \arrow[d] \\
        G(\overline{X_i}) \arrow[r] & G(\overline{X_i}\setminus X_i)
    \end{tikzcd}
\end{center}
is vertically right-adjoinable square in $\cat$; and similar for the back right square in (\ref{eq:monster}). By Lemma \ref{lem:parshall_scott}, this implies that the front and back square on the left are sent to vertically right-adjoinable squares. The vertical square in the middle is sent to a vertically right-adjoinable square in $\PrSt^\otimes$ as well, since it is a pullback square in $\comp(\C)_\times$. This means that for all $i$ in $I$, the square 
\begin{center}
    \begin{tikzcd}
       \prod_{\alpha^{-1}(i)} G(\overline{V_j})\arrow[d] \arrow[r] & G(\overline{Z_i}) \arrow[d] \\
       \prod_{\alpha^{-1}(i)} G(\overline{Y_j}) \arrow[r]& G(\overline{X_i})
    \end{tikzcd}
\end{center}
is vertically right-adjoinable in $\cat$. By the first part of this proof, the horizontal maps on the left in (\ref{eq:monster}) are sent to fully faithful embeddings. Note that products of fully faithful embeddings are fully faithful embeddings, and products of adjoinable squares are adjoinable. It follows from Lemma \ref{lem:another_cube} that the vertical square on the left is sent to a vertically right-adjoinable square in $\Prcat^\otimes$, as desired.
\end{proof}

\begin{lem}\label{lem:cs6ff_step_0.5}
 For $G$ in $\alg_{\span(\C)^\op}^{\tau^c_{A\cup L}}(\PrSt)_\emptyset$, if $G|_{\comp_\times}$ is in $ \alg_{\comp(\C)^\op}^{BC,AC}(\PrSt)$, then $G$ sends morphisms in $I$ to left-left adjoints and squares of types $(P_-^\op, I_-)$, $(I_-, P_\times^\op)$ and $(I_-, I_-)$ to vertically right-adjoinable squares.
\end{lem}

\begin{proof}     
\textbf{Morphisms in $I$.} Let $i:U\hookrightarrow V$ be in $I$. Then by assumption, $G$ sends the sequence 
\begin{center}
    \begin{tikzcd}
        V\setminus U \arrow[r, two heads, tail, "j"]& V & U \arrow[l, hook', "i"']
    \end{tikzcd}
\end{center}
to a fiber sequence 
$$G(U) \xrightarrow{i_!} G(V) \xrightarrow{j^*} G(V \setminus U).$$
Now by Lemma \ref{lem:cs6ff_step0}, we know that $j^*$ has a right adjoint $j_\natural$, and that $G$ sends the pullback square
\begin{center}
    \begin{tikzcd}
        V \setminus U \arrow[d, "= "] \arrow[r, "="] & V \setminus U \arrow[d, tail, two heads, "j"] \\
        V \setminus U \arrow[r, tail, two heads, "j"] & V
    \end{tikzcd}
\end{center}
to a vertically right-adjoinable square, showing that $j_\natural$ is fully faithful. By \cite[Lemma A.2.5]{nine}, it follows that the fiber sequence is a Verdier sequence and $i_!$ has a right adjoint $i^\natural$ which is full and essentially surjective. Moreover, by Lemma \ref{lem:cs6ff_step0}, $j_\natural$ has a further right adjoint $j^!$. Again, \cite[Lemma A.2.5]{nine} implies that $i^\natural$ has a fully faithful right adjoint, and therefore $i_!$ is a left-left adjoint.

\textbf{Squares of type $(I_-,P_\times^\op)$.}
First, we show that squares of type $(I_-,P^\op_-)$ are sent to vertically right-adjoinable squares. Indeed, suppose that 
\begin{equation*}
\begin{tikzcd}
(X_i)_I \arrow[r, two heads ,"p" ]                 & (Y_i)_I                 \\
(U_i)_I \arrow[u, hook, "i"] \arrow[r, two heads, "q" ] & (V_i)_I \arrow[u, hook, "j"]
\end{tikzcd}
\end{equation*}
is such a square. Now consider the diagram
\begin{equation*}
\begin{tikzcd}
(U_i)_I \arrow[r, hook,"i"] \arrow[d,  two heads, "q"] & (X_i)_I \arrow[d,two heads,"p"  ] & (X_i\setminus U_i)_I \arrow[l,tail, two heads, "e" ] \arrow[d, "r" ]  \\
(V_i)_I \arrow[r, hook,"j"]                      & (Y_i)_I                      & (Y_i\setminus V_i)_I \arrow[l,tail,two heads, two heads, "f"  ]                     
\end{tikzcd}
\end{equation*}
which is sent by $G$ to 
\begin{center}
    \begin{tikzcd}
        G((V_i)_I)\arrow[r, "j_!" ] \arrow[d, " q^*"] & G((Y_i)_I) \arrow[r, "f^*"] \arrow[d, "p^*"]& G((Y_i\setminus V_i)_i) \arrow[d, "r^*"] \\
        G((U_i)_I)\arrow[r, "i_!" ] & G((X_i)_I) \arrow[r, "e^*"] & G((X_i\setminus U_i)_I). 
    \end{tikzcd}
\end{center}
By Lemma \ref{lem:cs6ff_step0} it follows that the square on the right is sent to a horizontally right-adjoinable square. Therefore by Lemma \ref{lem:parshall_scott}, the square on the left is sent to a horizontally right adjoinable square, as desired.

Now let  
\begin{equation*}
\begin{tikzcd}
(X_i)_I \arrow[r, two heads ,"p" ]                 & (Y_j)_J                 \\
(U_i)_I \arrow[u, hook, "i"] \arrow[r, two heads, "q" ] & (V_j)_J \arrow[u, hook, "j"]
\end{tikzcd}
\end{equation*}
be a square of type $(I_-,P_\times^\op)$ , where the horizontal morphisms lie over a partial map $\alpha:J \dashrightarrow I$. Note that this square can be written as a composition
\begin{equation*}
    \begin{tikzcd}
(X_i)_I \arrow[r, two heads ,"p" ]                 & (\prod_{\alpha^{-1}(i)} Y_j)_I     \arrow[r] & (Y_j)_J             \\
(U_i)_I \arrow[u, hook, "i"] \arrow[r, two heads, "q" ] & (\prod_{\alpha^{-1}(i)} V_j)_I \arrow[u, hook, "j"] \arrow[r] & (V_j)_J. \arrow[u, hook, "j"]
\end{tikzcd}
\end{equation*}
We know that the square on the left is sent to a vertically right-adjoinable square, since it is of type $(I_-,P_-^\op)$, so we focus on the square on the right. We need to show that for all $i \in I$, the square
\begin{center}
    \begin{tikzcd}
        \prod_{\alpha^{-1}(i)}G(V_j) \arrow[r, "\boxtimes"] \arrow[d, "j_!"] & G(\prod_{\alpha^{-1}(i)} V_j) \arrow[d, "j_!"] \\
         \prod_{\alpha^{-1}(i)}G(Y_j) \arrow[r, "\boxtimes"] &G(\prod_{\alpha^{-1}(i)} Y_j) 
    \end{tikzcd}
\end{center}
is vertically right-adjoinable. For each $j\in \alpha^{-1}(i)$, let $Y_j \hookrightarrow\overline{Y_j}$ be a compactification; then $\overline{Y_j}$ is also a compactification of $V_j$. We can form the cube
\begin{center}
\begin{tikzcd}
   & G(\prod V_j) \arrow[rr] \arrow[dd, "j_!"', near start] &                                                            & G(\prod \overline{Y_j}) \arrow[dd, equal] \\
\prod G(V_j) \arrow[ru, "\boxtimes"] \arrow[rr, crossing over] \arrow[dd, "j_!"', near start] &                                                   & \prod G(\overline{Y_j}) \arrow[ru, "\boxtimes"]  &                                \\
      & G(\prod Y_j) \arrow[rr]                         &                                                            & G(\prod \overline{Y_j})            \\
\prod G(Y_j) \arrow[ru, "\boxtimes"] \arrow[rr]                         &                                                   & \prod G(\overline{Y_j}) \arrow[ru, "\boxtimes"] \arrow[uu, equal, crossing over]           &                               
\end{tikzcd}
\end{center}
where we denote the arrows going into the paper by $\boxtimes$ since they can be though of as the external tensor product. We $G$ sends $I_-$-morphisms to fully faithful embeddings, whose right adjoint is full and essentially surjective. 
Therefore by Lemma \ref{lem:cube}, we know that the face on the left is vertically right-adjoinable if the top and bottom faces, i.e., the squares
\begin{center}
    \begin{tikzcd}
        \prod G(V_j) \arrow[r, "\boxtimes"] \arrow[d] & G(\prod V_j) \arrow[d]\\
        \prod G(\overline{Y_j}) \arrow[r, "\boxtimes"] & G(\prod \overline{Y_j})
    \end{tikzcd}
    and 
    \begin{tikzcd}
           \prod G(\prod Y_j) \arrow[r, "\boxtimes"] \arrow[d,] & G(\prod Y_j) \arrow[d]\\
        \prod G(\overline{Y_j}) \arrow[r, "\boxtimes"] & G(\prod \overline{Y_j})  
    \end{tikzcd}
\end{center}
are vertically right-adjoinable. For both of the squares, this will follow from the following argument. Let $W_1,\dots, W_n$ be a finite set of objects in $\C$, and for each $1\leq k \leq n$, let $W_k\hookrightarrow \overline{W_k}$ be some compactification. We want to show that
\begin{center}
       \begin{tikzcd}
            G(W_1)\times \dots \times G(W_n) \arrow[r, "\boxtimes"] \arrow[d] & G(W_1\times \dots \times W_n) \arrow[d]\\
       G(\overline{W_1})\times \dots \times G(\overline{W_n}) \arrow[r, "\boxtimes"] & G(\overline{W_1}\times \dots \times \overline{W_n})  
    \end{tikzcd}
\end{center}
is vertically right-adjoinable. We can decompose this square as
\begin{center}
       \begin{tikzcd}
            G(W_1)\times \dots \times G(W_n) \arrow[r, "\boxtimes"] \arrow[d] & G(W_1 \times \dots\times W_n) \arrow[d]\\
       G(\overline{W_1})\times G(W_2) \times \dots \times G(W_n) \arrow[d] \arrow[r, "\boxtimes"] & G(\overline{W_1}\times  W_2 \times \dots \times W_n) \arrow[d] \\
       \dots \arrow[d, ] & \dots \arrow[d, ] \\
              G(\overline{W_1})\times \dots \times G(\overline{W_n}) \arrow[r, "\boxtimes"] & G(\overline{W_1}\times \dots \times\overline{W_n})
    \end{tikzcd}
\end{center}
i.e, only replacing one more $W_k$ by its compactification in each row. Now it suffices that each of the squares in this vertically right-adjoinable. We do this for the topmost square, the argument is the same for all other squares. Consider the diagram 
\begin{center}
    \begin{tikzcd}
        W_1\times \dots \times W_n \arrow[d,two heads] \arrow[r, hook] & \overline{W_1}\times W_2 \times \dots \times W_n \arrow[d, two heads] & (\overline{W_1}\setminus W_1) \times W_2 \times \dots \times W_n \arrow[d, two heads] \arrow[l, tail,two heads]\\
         (W_1,\dots , W_n)  \arrow[r, hook] & (\overline{W_1}, W_2 , \dots , W_n ) & (\overline{W_1}\setminus W_1,W_2 , \dots , W_n )\arrow[l, tail, two heads]
    \end{tikzcd}
\end{center}
in $\span(\C)_\times$. Then $G$ sends the rows to recollements and the square on the right to a horizontally right-adjoinable square, since it is of type $(P_\times^\op, P_-^\op)$. It follows that the square on the left is sent to a horizontally right-adjoinable square, as desired.

\textbf{Squares of type $(P_-^\op, I_-)$.} Let
\begin{center}
    \begin{tikzcd}
        (U_i)_I \arrow[r, hook] \arrow[d,two heads  ] & (X_i)_I \arrow[d, two heads  ] \\
        (V_i)_I \arrow[r, hook] & (Y_i)_I
    \end{tikzcd}
\end{center}
be a square of type $(P_-^\op, I_-)$. As above, we can extend this with a pullback square of type $(P_-^\op,P_-^\op )$ as follows.
\begin{equation*}
\begin{tikzcd}
(U_i)_I \arrow[r, hook] \arrow[d, two heads ] & (X_i)_I \arrow[d, two heads  ] & (X_i\setminus U_i)_I \arrow[l,tail, two heads  ] \arrow[d, two heads ] \\
(V_i)_I \arrow[r, hook]                      & (Y_i)_I                      & (Y_i\setminus V_i)_I. \arrow[l, tail, two heads ]                     
\end{tikzcd}
\end{equation*}
The rows in the diagram are sent to left-recollements, and by Lemma \ref{lem:cs6ff_step0} the square on the right is sent to a vertically right-adjoinable square. Lemma \ref{lem:parshall_scott} implies that image of the square of the left is  vertically right-adjoinable.

\textbf{Squares of type $(I_-,I_-)$.} Lastly we consider a pullback square
\begin{equation*}
\begin{tikzcd}
(X_i)_I                 & (V_i)_I \arrow[l, hook']                 \\
(U_i)_I \arrow[u, hook] & (W_i)_I \arrow[u, hook] \arrow[l, hook']
\end{tikzcd}
\end{equation*}
of type $(I_-,I_-)$ in $\span(\C)_\times^\op$. We can extend this with another a pullback square of type $(I_-,P_-)$: 
\begin{equation*}
\begin{tikzcd}
(V_i)_I \arrow[r, hook]                 & (X_i)_I                 & (X_i\setminus V_i)_I \arrow[l, tail, two heads ]                 \\
(W_i)_I \arrow[u, hook] \arrow[r, hook] & (U_i)_I \arrow[u, hook] & (U_i\setminus W_i)_I \arrow[l,tail, two heads  ] \arrow[u, hook]
\end{tikzcd}
\end{equation*}
The rows are sent to left-recollements, and by an earlier step, the square on the right is sent to a vertically right-adjoinable square square. By Lemma \ref{lem:parshall_scott}, this implies that the square on the left is sent to a vertically right-adjoinable square. 
\end{proof}

\begin{proof}[Proof of Proposition \ref{prop:cs6ff}]
Lemma \ref{lem:cs6ff_step0} and Lemma \ref{lem:cs6ff_step_0.5} provide most of the proof; together they show that for $G$ in $\alg^{\tau^c_{A\cup L}}_{\span(\C)^\op}(\PrSt)_\emptyset $, if the restriction $G|_{\comp(\C)_\times}$ is in $\alg^{BC,AC}_{\comp(\C)^\op}(\PrSt)$, then $G$ is in $\alg^{BC,A\cup L}_{\span(\C)^\op}(\PrSt)_\emptyset $. What we still need to show is: for a natural transformation $\alpha:G\to G'$ in $\alg^{\tau^c_{A\cup L}}_{\span(\C)^\op}(\PrSt)_\emptyset $, if $\alpha|_{\comp(\C)_\times}$ is in $\alg^{BC,AC}_{\comp(\C)^\op}(\PrSt)$, then $\alpha$ is in $\alg^{BC,A\cup L}_{\span(\C)^\op}(\PrSt)_\emptyset $, which means it has to be compatible with the adjoints $p_\natural$ and $i^\natural$. 

Let $\alpha:G\to G'$ be such a natural transformation. For $q:U \twoheadrightarrow V$ a $P$-morphism in $\span(\C)$, as in the proof of Lemma \ref{lem:cs6ff_step0} we find a pullback square
\begin{center}
    \begin{tikzcd}
        U \arrow[r,hook, "i " ] \arrow[d, two heads, "q"] & X \arrow[d, two heads, "p"] \\
        V \arrow[r, hook, "j"] & Y
    \end{tikzcd}
\end{center}
of $I$- and $P$-morphisms, with $X$ and $Y$ in $\comp(\C)$. Applying the natural transformation $\alpha$, we obtain a commutative diagram 
\begin{center}
\begin{tikzcd}
G(V) \arrow[dd] \arrow[rd, "\alpha_V"] \arrow[rr] &                             & G(Y) \arrow[rd, "\alpha_Y"] \arrow[rr] \arrow[dd] &                             & G(Y\setminus V) \arrow[rd, "\alpha_{Y \setminus V}"] \arrow[dd] &                             \\
                                      & G'(V) \arrow[dd] \arrow[rr] &                                       & G'(Y) \arrow[dd] \arrow[rr] &                                       & G'(Y\setminus V) \arrow[dd] \\
G(U) \arrow[rd, "\alpha_U"] \arrow[rr]            &                             & G(X) \arrow[rd, "\alpha_X"] \arrow[rr]            &                             & G(X\setminus U) \arrow[rd, "\alpha_{X \setminus U}"]            &                             \\
                                      & G'(U) \arrow[rr]            &                                       & G'(X) \arrow[rr]            &                                       & G'(X\setminus U)           
\end{tikzcd}
\end{center}
where the rows are left-recollements. By the condition on $\alpha|_{\comp_\times}$, in the cube on the right, the vertical faces are vertically right-adjoinable. Then the front and back face of the cube on the left are vertically right-adjoinable by Lemma \ref{lem:parshall_scott}, and the leftmost face is vertically right-adjoinable by Lemma \ref{lem:another_cube}.

Now let $i:U\hookleftarrow V$ be an $I$-morphism in $\span(\C)$. We can form the commutative diagram 
\begin{center}
\begin{tikzcd}
G(U) \arrow[d, "\alpha_U"'] \arrow[r] & G(V) \arrow[d, "\alpha_V"] \arrow[r] & G(V\setminus U) \arrow[d, "\alpha_{V\setminus U}"] \\
G'(U) \arrow[r]                       & G'(V) \arrow[r]                      & G'(V\setminus U)                                  
\end{tikzcd}
\end{center}
where the rows are left-recollements.
Since $V\setminus U \twoheadrightarrow V$ is proper, the square on the left is horizontally right-adjoinable. By Lemma \ref{lem:parshall_scott}, the square on the right is horizontally right-adjoinable, as desired. 
\end{proof}

\begin{thm}\label{thm:cs6ff}
    There is an equivalence of $\infty$-categories
    $$\textup{\textbf{6FF}}^\loc \simeq \alg^{BC,AC}_{\comp(\C)^\op}(\PrSt)$$
\end{thm}
\begin{proof}
    Combine Proposition \ref{prop:6ff_to_span} and Proposition \ref{prop:cs6ff}.
\end{proof}

Now that we have chased $\textup{\textbf{6FF}}^\loc$ through the equivalence (\ref{eq:compactly_supportes_sheaf_theories}), one wonders if, for the noetherian Nagata setup of varieties over a field of characteristic zero, we can give a similar description of the corresponding subcategory of $\alg_{\smcomp^\op}^{\tau_B}(\Prst)$ by chasing $\textup{\textbf{6FF}}^\loc$ further through the equivalence (\ref{eq:compactly_supportes_sheaf_theories_2}). This proves to be difficult, the existence of this equivalence guarantees the following.

\begin{thm}\label{thm:6FF}
For $(\var,I,P)$ the noetherian Nagata setup of Example \ref{ex:varieties}, where $k$ is of characteristic zero, restricting the $(-)^*$-part of a six functors formalism to $\smcomp$ gives a faithful embedding
$$\textup{\textbf{6FF}}^\loc \hookrightarrow \alg_{\smcomp^\op}^{\tau_B}(\Prst).$$ 
\end{thm}

\begin{rmk}\label{rmk:smcomp}
We observe that for any functor $F:\smcomp_\times^\op \to \PrSt^\otimes $ in the image of this embedding,
we have 
\begin{enumerate}[label=(\arabic*)]
    \item for $f:X\to Y$ in $\smcomp$, the image $f^*:F(Y)\to F(X)$ is a left-left adjoint
    \item squares of type $(P_-^\op, P_\times^\op)$ are sent to vertically right-adjoinable squares.
\end{enumerate}
Moreover, a natural transformation in the image of this embedding will be compatible with the right adjoints $p_\natural$. However, we do not expect that there is an equivalence from $\textup{\textbf{6FF}}^\loc$ onto the category of lax symmetric monoidal hypersheaves such that the conditions (1) and (2) above hold, and natural transformations compatible with the right adjoints. One reason for this suspicion is that the 1-category $\smcomp$ does not have a lot of pullbacks, so if $F:\smcomp^\op \to \PrSt^\otimes $ sends all  squares of type $(P_-^\op, P_\times^\op)$ to vertically right-adjoinable squares, then this might not be true for the extension of $F$ to $\comp$ which has a lot more pullbacks. 
\end{rmk}

\section*{Conflicts of interest}
There are no conflicts of interest.
\section*{Financial support}
This work was supported by the ERC grant ERC-2017-STG
759082 through Dan Petersen.

\bibliographystyle{alpha}
\bibliography{bibliography}

\end{document}